\documentclass[12pt,oneside]{book}
\usepackage{amsmath,amssymb}
\usepackage{amsthm}
\usepackage{graphicx}
\usepackage{makeidx}

\def\AC{\mathcal{A}}
\def\PC{\mathcal{P}}

\def\MC{\mathcal{M}}

\def\PC{\mathcal{P}}

\def\C{\mathbf{C}}

\def\E{\mathbf{E}}

\def\N{\mathbf{N}}

\def\P{\mathbf{P}}
\def\R{\mathbf{R}}

\def\Z{\mathbf{Z}}

\def\1{\mathbf{1}}
\def\0{\mathbf{0}}

\def\sgn{\rm{sgn}}

\def\al{\alpha}
\def\be{\beta}
\def\pa{\partial}
\def\ep{\epsilon}
\def\de{\delta}
\def\ga{\gamma}
\def\ka{\varkappa}

\newtheorem{prop}{Proposition}[section]
\newtheorem{theorem}{Theorem}[section]

\newtheorem{remark}{Remark}

\newcommand{\la}{\lambda}

\newcommand{\om}{\omega}
\newcommand{\Ga}{\Gamma}

\newcommand{\Si}{\Sigma}
\newcommand{\Om}{\Omega}
\newcommand{\La}{\Lambda}

\makeindex

\begin{document}
\frontmatter

\title{Multi-agent interaction and nonlinear Markov games
\thanks{ This is the final draft of the 1st part of the book 'Many agent games in socio-economic systems: corruption, inspection, coalition building, network growth, security'
published by Springer in the Springer Series on Operational Research, 2019}}
\author{
Vassili N. Kolokoltsov\thanks{Department of Statistics, University of Warwick,
Coventry CV4 7AL UK, and associate member of the Fac. of Appl. Math. and Control Processes,
St.-Petersburg State Univ., Russia, Email: v.kolokoltsov@warwick.ac.uk}
and Oleg A. Malafeyev\thanks{Fac. of Appl. Math. and Control Processes, St.-Petersburg State Univ., Russia}}

\maketitle

\newpage

\section*{Preface}

 The general picture of game theoretic modeling dealt with in this book
 is characterized by a set of big players, also referred to as principals
or major agents, acting on the background of large pools of small players,
the impact of the behavior of each small player in a group on the overall
 evolution decreasing with the increase of the size of the group.

Two approaches to the analysis of such systems are clearly distinguished
and are dealt with in Parts I and II.

(1) Players in groups are not independent rational optimizers. They are either directly controlled by principals
 and serve the interests of the latter (pressure and collaboration setting) or they resist the actions of the
 principals (pressure and resistance setting) by evolving their strategies in an 'evolutionary manner' via interactions
 with other players subject to certain clear rules, deterministic or stochastic. Such interactions, often referred to
 as myopic or imitating, include the exchange of opinions or experience, with some given probabilities of moving to
  more profitable strategies. They can also evolve via the influence of social norms.

The examples of the real world problems involved include government representatives
(often referred to in the literature as benevolent dictators) chasing corrupted bureaucrats,
inspectors chasing tax-paying avoidance, police acting against terrorist groups or models
describing the attacks of computer or biological viruses. This includes the problem of optimal
allocation  of the budget or efforts of the big player to different strategies affecting small
players, for instance, the allocation of funds (corrected in real time) for the financial
support of various business or research projects. Other class of examples concerns appropriate
(or better optimal) management of complex stochastic systems consisting of large number
of interacting components (agents, mechanisms, vehicles, subsidiaries, species, police units,
robot swarms, etc), which may have competitive or common interests. Such management can also
deal with the processes of merging and splitting of functional units (say, firms or banks) or
the coalition building of agents. The actions of the big players effectively control the
distribution of small players among their possible strategies and can influence the rules of
their interaction. Several big players can also compete for more effective pressure on small
players. This includes, in particular, the controlled extensions of general (nonlinear)
evolutionary games. Under our approach the classical games of evolutionary biology,
like hawk and dove game, can be recast as a controlled process of the propagation of violence,
say in regions with mixed cultural and/or religious traditions.

For discrete state spaces the games of this kind
were introduced in \cite{Ko09}, \cite{Ko10},\cite{Ko12}
under the name of nonlinear Markov games. Similar
models with continuous state space were developed in
\cite{BenFr} under the name of mean-field-type games,
see also \cite{Djehiche15}.

(2) Small players in groups are themselves assumed to be rational optimizers, though in the limit of large number
of players the influence of the decisions of each individual player on the whole evolution becomes negligible.
The games of this type are referred to as mean field games. They were introduced in \cite{HCM3} and \cite{LL2006},
with the finite state space version put forward in \cite{CMS2010}, and since then developed into one
of the most active directions of research in game theory. We shall discuss this setting mostly in combination with
the first approach, with evolutionary interactions (more precisely, pressure and resistance framework) and individual
decision making taken into account simultaneously. This combination leads naturally to two-dimensional arrays of possible states
of individual players, one dimension controlled by the principals and/or evolutionary interactions and another dimension
by individual decisions.

Carrying out a traditional Markov decision analysis for
a large state space (large number of players and particles)
is often unfeasible. The general idea for our analysis is
that under rather general assumptions, the limiting problem
for a large number of agents can be described by a well
manageable deterministic evolution, which represents a
performance of the dynamic law of large numbers (LLN).
This procedure turns the 'curse of dimensionality' to
the 'blessing of dimensionality'. As we show all basic
criteria of optimal decision making (including competitive
control) can be transferred from the models of a large
number of players to a simpler limiting decision problem.

 Since even the deterministic limit of the combined rational decision making processes and evolutionary type models
 can become extremely complex, another key idea of our analysis is in searching for certain reasonable asymptotic regimes,
 where explicit calculations can be performed. Several such regimes are identified and effectively used.

We will deal mostly with discrete models, thus avoiding technicalities arising in general models (requiring
stochastic differential equations or infinite-dimensional analysis). Extensions dealing with general jump-type
 Markov processes are often straightforward, see e.g. \cite{Ko17}.

From the practical point of view the  approaches developed here are mostly appropriate for dealing with socioeconomic processes
that are not too far from an equilibrium. For those processes the equilibria play the role of the so-called turnpikes, that is,
attracting stationary developments. Therefore much attention in our research is given to equilibria (rest points of the dynamics)
and their structural and dynamic stability. For processes far from equilibria other approaches seem to be more relevant,
for instance the methods for the analysis of turbulence. Another problem needed to be addressed for concrete applications of our models
lies in the necessity to get hold of the basic parameters entering its formulation, which may not be that easy. Nevertheless,
the strong point of our approach is that it requires to identify really just a few real numbers, which may be derived in principle
 from statistical experiments or field observations, and not any unknown multi-dimensional distributions.

In the first introductory chapter we explain the main
results and applications with the minimal use of technical
language trying to make the key ideas accessible to readers
with only rudimentary mathematical background. The core of
the book is meant for readers with some basic knowledge of
probability, analysis and ordinary differential equations.

For the sake of transparency we systematically develop
the theory with the sequential increase of complexity,
formulating and proving results first in their simplest
 versions and then extending generality.

\section*{Basic notation}

Notations introduced here are used in the main text systematically without further reminder.

$\Z$, $\C$, $\R$, $\N$ denote the sets of integer, complex,
real and natural numbers, $\Z_+$ and $\R_+$ the
subsets of non-negative elements of the corresponding sets.

Letters $\E$, $\P$ are reserved for the expectation
and probability with respect to various Markov chains.

For a convex closed subset $Z$ of $\R^d$, let $C(Z)$ denote the
space of bounded continuous functions equipped with the sup-norm:
$\|f\|=\|f\|_{sup}=\sup_x |f(x)|$, and $C_{\infty}(Z)$ its closed
 subspace of functions vanishing at infinity.

The vectors in $x\in \R^d$ can be looked at as functions on $\{1, \cdots, d\}$, in which case the natural norm
is the sup-norm $\|x\|_{sup}=\max_i|x_i|$ (rather than more standard rotation invariant Euclidean norm). But much
 often we shall interpret these vectors as measures on  $\{1, \cdots, d\}$,
in which case the natural norm becomes the $l_1$-norm or the integral norm
$\|x\|=\|x\|_1=\sum_j |x_j|$. This norm will be used as default in $\R^d$,
unless stated otherwise in specific cases. Accordingly, for functions
 on $Z\subset \R^n$ we define the Lipschitz constant
 \begin{equation}
\label{eqdefLipnorml1}
\|f\|_{Lip} = \sup_{x\neq y} \frac{|f(x)-f(y)|}{\|x-y\|}.
\end{equation}
As is easy to see, this can be equivalently rewritten as
 \begin{equation}
\label{eqdefLipnorml1a}
\|f\|_{Lip} =\sup_j \sup \frac{|f(x)-f(y)|}{|x_j-y_j|},
\end{equation}
where the last $\sup$ is the supremum over the pairs $x,y$ that differ
 only in its $j$th coordinate. By $C_{bLip}(Z)$ we denote the space of
 bounded Lipschitz functions with the norm
 \begin{equation}
\label{eqdefLipnorm}
 \|f\|_{bLip}=\|f\|+\|f\|_{Lip}.
\end{equation}

Norms of $d\times d$-matrices are identified with their norms as linear operators in $\R^d$:
 $\|A\|=\sup_{x\neq 0}\|Ax\|/\|x||$. This depends on the norm in $\R^d$. We will look at matrices
 as acting by usual left multiplications $x\to Ax$ in $\R^d$ equipped with the sup-norm or
 as acting by the right multiplications $x\to xA=A^Tx$ in $\R^d$ equipped with the integral norm,
 so that in both cases (as is checked directly)
 \begin{equation}
\label{eqourmatrixnorm}
\|A\|=\sup_i \sum_j |A_{ij}|.
\end{equation}

By $C^k(Z)$ we denote the space of $k$ times continuously differentiable
functions on $Z$ with uniformly bounded derivatives equipped with the norm
\[
\|f\|_{C^k(Z)}=\|f\|+\sum_{j=1}^k \| f^{(j)}(Z)\|,
\]
where $\|f^{(j)}(Z)\|$ is the supremum of the magnitudes of all partial derivatives of $f$ of order $j$ on $Z$.
In particular, for a differentiable function, $\|f\|_{C^1}=\|f\|_{bLip}$ (by \eqref{eqdefLipnorml1a}).
For a function $f(t,x)$ that may depend on other variables, say $t$, we shall also use the notation
\[
 \| f^{(j)}\|=\| \frac{ \pa ^j f}{\pa x^j}\|=\|f(t,.)^{(j)}\|,
 \]
 if we need to stress the variable $x$ with respect
 to which the smoothness is assessed. We shall also
 use shorter notations $\| f^{(j)}\|$, $C^k$ or $C_{bLip}$
for the norms $\| f^{(j)}(Z)\|$ and the spaces $C^k(Z)$
or $C_{bLip}(Z)$, when it is clear which $Z$ is used.

The space  $C(Z,\R^n)$ of bounded continuous vector-valued
 functions $f=(f_i):Z\to \R^n$, $Z\subset \R^d$, will be also
usually denoted $C(Z)$ (with some obvious abuse of notations),
their norm being
\[
\|f\|=\sup_{x\in Z} \|f(x)\|=\sup_{x \in Z} \sum_i |f_i(x)|,
\quad
\|f\|_{sup}=\sup_{i,x} |f_i(x)|.
\]
Similar shorter notation will be used for spaces of smooth
vector-valued functions $f=(f_i):Z\to \R^n$ equipped with the norm
\[
\|f\|_{C^k(Z)}=\|f\|+\sum_{j=1}^k \| f^{(j)}\|,
\quad \| f^{(j)}\|= \| f^{(j)}(Z)\|=\sum_i \|f_i^{(j)}(Z)\|.
\]

In chapter \ref{secmodgrowthpres} we will work with the functions
on infinite sequences, namely on the space $l_1$ of sequences
$\{x=(x_1,x_2, \cdots )\}$ having finite norm  $\|x\|=\sum_j |x_j|$.
All the notations above will be used then for closed convex subsets $Z$ of $l_1$.

\section*{Standard abbreviations}

\hspace{5mm}

r.h.s. right-hand side

l.h.s. left-hand side

LLN law of large numbers

ODE ordinary differential equation

PDE partial differential equation

PDO partial differential operator

\section*{Acknowledgements}

The authors are grateful to many colleagues for fruitful discussions
and especially to our collaborators on the related research developments:
Yu. Averbukh, A. Bensoussan, A. Hilbert, S. Katsikas, L. A. Petrosyan,
M. Troeva, W. Yang. We also thank R. Avenhaus, N. Korgyn, V. Mazalov
and the anonymous referees for careful reading the preliminary versions
and making important comments.

The authors gratefully acknowledge support of the RFBR grants 17-01-00069
(Part II) and 18-01-00796 (Part I).

\newpage

\tableofcontents

\newpage

\mainmatter

\chapter{Introduction: main models and LLN  methodology}
\label{chapintkolMaCor}

First sections of this chapter are devoted to a brief introduction to our basic model
of controlled mean-field interacting particle systems, its main performance in the pressure
and resistance framework, and to the fundamental paradigm of the scaling limit (dynamic law of large numbers)
allowing one to reduce the analysis of these Markov chains with exponentially large
state spaces to a low dimensional dynamic system expressed in terms of simple ordinary
 differential equations (ODEs), so called kinetic equations. The rigorous mathematical
 justification of this reduction being postponed to the next chapter, in the rest of this chapter
 we describe various classes of real life socio-economic processes fitting to the general scheme:
 inspection, corruption, cyber-security, counterterrorism, coalition building, merging and splitting,
 threshold control behavior in active network structures, evolutionary games with controlled
 mean-field dependent payoffs and many other. In fact, our general
 framework allows us to bring a big variety of models (often studied independently) under a single umbrella
 amenable to simple unified analysis. Moreover, many standard socio-economic models, dealt with traditionally
 via small games with 2 or 3 players, arise in our presentation in new clothes allowing for the analysis of
 their many-player extensions. In the last section we give a brief guide to further content of the book.

\section{What is a Markov chain}
\label{secMarch}

The notion of Markov chains, the simplest models of random evolutions, is a cornerstone of our analysis.
Here we describe the basic ideas with the minimal usage of technical terms.

One distinguishes Markov chains in discrete and continuous time.

The Markov chains in discrete time are characterized by
a finite set $\{1,...,d\}$, called the {\it state space}, of possible positions of an object under investigation (particle or agent),
and a collection of non-negative numbers $P=\{P_{ij}\}$, $i,j$ from $\{1,...,d\}$,
called the {\it transition probabilities}\index{transition probabilities} from $i$ to $j$, such that $P_{i1}+\cdots +P_{id}=1$ for any $i$.
Collection of numbers $P=\{P_{ij}\}$ satisfying these conditions is called a {\it matrix of transition probabilities},
or a {\it transition matrix}\index{transition matrix} or a {\it stochastic matrix}.
The {\it Markov chain}\index{Markov chain} specified by the collection $P$ is the process of random transitions between
the states  $\{1,...,d\}$ evolving in discrete times $t=0,1,2, \cdots$ by the following rule. If at some time $t$ the object
is in a state $i$, in the next moment $t+1$ it migrates from state $i$ to some other state $j$,
choosing $j$ with the probability $ P_{ij}$. Then the process repeats starting from $j$ at time $t+1$,
and the same procedure continues ad infinitum.

The Markov chains in continuous time are characterized by
a finite set $\{1,...,d\}$, called the {\it state space}, of possible positions of an object under investigation (particle or agent),
and a collection of non-negative numbers $Q=\{Q_{ij}\}$, $i\neq j$, $i,j$ from $\{1,...,d\}$,
called the {\it transition rates}\index{transition rates} from $i$ to $j$.
To circumvent the restriction $i\neq j$ it is convenient to introduces the negative numbers
\begin{equation}
\label{defintenMC}
Q_{ii}=-\sum_{j\neq i} Q_{ij}.
\end{equation}
The positive numbers $|Q_{ii}|=-Q_{ii}$ are called the {\it intensities of jumps}\index{intensities of jumps}
at state $i$. The collection $Q=\{Q_{ij}\}$,  $i,j\in\{1,...,d\}$, is then called a $Q$-{\it matrix} or a
{\it Kolmogorov's matrix}\index{Kolmogorov's matrix}. The {\it Markov chain}\index{Markov chain} specified by the
collection $Q$ is the process of random transitions between the states  $\{1,...,d\}$ evolving by the following
rule. If at some time $t$ the object is in a state $i$, the object sits at $i$ a random time $\tau$
characterized by the following waiting probability:
\[
\P (\tau >s)=\exp\{-s|Q_{ii}|\},
\]
such random time being referred to as $|Q_{ii}|$-{\it exponential waiting time}.
At time $\tau$ the object instantaneously migrates from state $i$ to some other state $j$,
choosing $j$ with the probability $ Q_{ij}/ |Q_{ii}|$. Then the process repeats starting from $j$,
and the same procedure continues ad infinitum.

Models with discrete or continuous time are chosen from practical convenience. If we observe a process
with some fixed frequency (say, measure a temperature in the sea every morning), discrete setting is more appropriate.
If we observe a process steadily, continuous-time-setting is more appropriate.

It is useful and popular to represent Markov chains geometrically via
oriented graphs. Namely, the {\it graph of a continuous time chain with
the rates}\index{Markov chain!graph of} $Q_{ij}$, $i,j \in \{1, \cdots , d\}$, is
a collection of $d$ vertices with the oriented edges $e_{ij}$ (arrows
directed from $i$ to $j$) drawn for the pairs $(i,j)$ such that $Q_{ij} \neq 0$. To
complete the picture the values $Q_{ij}$ can be placed against each edge $e_{ij}$.

The typical examples of Markov chains with two states represent radioactive atoms that can decay
after some random time (decay time) or firms that can be defaulted after some random time
(time to default). In fact, Markov chains are indispensable tools in all branches of science.
For instance, in genetics they model the propagation of various genes through the reproduction
(that can be randomly chosen from father of mother), in biology they control the processes
of evolution, in finances they model the dynamics of prices of financial instruments (say,
stocks or options), in economics they are used to model firms' growth and creditability
or the competitions between banks or supermarkets, etc.

The state spaces of Markov chains that we will be mostly concern with here
are the strategies of individuals (say, being corrupted or honest inspector,
level of illegal activity, etc) or the working conditions of agents or devices
(say, infected or susceptible computer or individual, level of defence, etc).

\section{Mean-field interacting particle systems}
\label{secmeanfieldch1}

 Let us turn to the basic setting of mean-field interacting particle systems with a finite number of types.
 We want to model a large number $N$ of particles or agents evolving each according to a Markov chain on some
 common state space $\{1,...,d\}$, but with the transition rates depending on the overall distribution of agents
 among the states. Thus the overall state of the system of such agents is given by a collection of integers
 $n=(n_1, \cdots, n_d)$, where $n_j$ is the number of agents in the state $i$, so that $N=n_1+\cdots +n_d$.

 Usually it is more convenient to work with frequencies, that is, instead of $n$, to use the normalized quantities
  $x=(x_1, \cdots , x_d)=n/N$, that is,  $x_j=n_j/N$. All such vectors $x$ belong to the unit simplex
  $\Si_d$, defined as the collections of all $d$ non-negative numbers summing up to unity:
  \[
 \Si_d=\{x=(x_1,...,x_d): x_j\ge 0 \,\, \text{for all}\,\, j \,\, \text{and} \,\, \sum_{j=1}^d x_j=1\}.
 \]
  To describe transitions of individuals depending on the overall
  distribution $x$, we have to have a family $\{Q (x)\}=\{(Q_{ij})(x)\}$ of Kolmogorov's matrices
  depending on a vector $x$ from $\Si_d$.

Thus, for any $x$, the family $\{Q (x)\}$ specifies a Markov chain on the state space $\{1,...,d\}$.
But this is not what we are looking for. We are interested in a system with many agents, i.e. a Markov chain
on the states $n=(n_1, \cdots, n_d)$. Namely, the {\it mean-field interacting particle system}
\index{mean-field interacting particle system} with $N$ agents specified by
the family $\{Q(x)\}$ is defined as the Markov chain on the set of collections
  $n=(n_1, \cdots, n_d)$ with $N=n_1+\cdots +n_d$ evolving
according to the following rule.  Starting from any time and current state $n$
one attaches to each particle a $|Q_{ii}|(n/N)$-exponential random waiting time (where $i$ is the
type of this particle). These $N$ times are random and their minimum is therefore also a random time,
which is in fact (as can be shown) a $\sum_i n_i|Q_{ii}(n/N)|$-exponential random waiting time.
If this minimal time $\tau$ turns out to be attached to some particle
of type $i$, this particle jumps to a state $j$ according to the probability law $(Q_{ij}/|Q_{ii}|)(n/N)$.
After such a transition the state $n$ of our system turns to the state $n^{ij}$,
which is the state obtained from $n$ by removing one particle of type $i$ and adding a particle of type $j$,
that is $n_i$ and $n_j$ are changed to $n_i-1$ and $n_j +1$ respectively.
After any such transition the evolution continues starting from the new state $n^{ij}$, etc.

Equivalently the transition from a state $n$ can be described by a single clock, like in the standard setting
for Markov chains. Namely, to a state $n$ one attaches a $\sum_i n_i|Q_{ii}(n/N)|$-exponential random waiting time. Once the bell rings one chooses the transition $i\to j$, $i\neq j$, with the probability
\begin{equation}
\label{deftransprobmfMC}
\P(i\to j)= \frac{n_iQ_{ij}(n/N)}{\sum_k n_k |Q_{kk}(n/N)|}.
\end{equation}
These are well defined probabilities, since
\[
\sum_i \sum_{j\neq i} \P(i\to j)= 1.
\]

As was mentioned, as the number of particles $N$ becomes large, the usual
analysis of such Markov chain becomes unfeasible, the situation often referred
to as the 'curse of dimensionality'. To overcome this problem one is searching
for a simpler limiting evolution as $N\to \infty$. The situation is quite similar
to the basic approaches used in physics to study the dynamics of gases and liquids:
instead of following the random behavior of immense number of individual molecules,
one looks for the evolution of the main mean characteristics, like bulk velocities
or temperatures that often turns out to be described by some
deterministic differential equations.

As will be shown, the limiting evolution of the frequencies
$x=(x_1, \cdots, x_d)$ evolving according to the Markov chain described above
 is governed by the system of ordinary differential equations
\begin{equation}
\label{eqdefgenmeanfieldchainkineq}
\dot x_k =\sum_{i\neq k}( x_i Q_{ik}(x)-x_kQ_{ki}(x))
=\sum_{i=1}^d x_i Q_{ik}(x), \quad k=1,...,d
\end{equation}
(the last equation arising from \eqref{defintenMC}), called the
{\it kinetic equations}\index{kinetic equation} for the process
 of interaction described above.

More precisely, under some continuity assumptions on the family
$\{Q(x)\}$ one shows that if $X^N_x(t)$ is the position of our
mean-field interacting particle system at time $t$ when started
at $x$ at the initial time and $X_x(t)$ is the solution of system
\eqref{eqdefgenmeanfieldchainkineq} at time $t$ also started at
$x$ at the initial time, then the probability for the deviation
$|X^N_x(t)-X_x(t)|$ to be larger than any number
$\ep$ tends to zero as $N\to \infty$:
\begin{equation}
\label{eqdynLLNprob}
\P (|X^N_x(t)-X_x(t)|>\ep) \to 0, \quad N\to \infty.
\end{equation}

Let us stress again that system \eqref{eqdefgenmeanfieldchainkineq}
is deterministic and finite dimensional, which is an essential
simplification as compared with Markov chains with increasingly large state spaces.

In this most elementary setting the result \eqref{eqdynLLNprob} is well known.
We will be interested in more general situations when the transition rates $Q$
are controlled by one or several big players that can influence the dynamics of
the chain to their advantages. Moreover, we are concerned with the rates of
convergence under mild regularity assumptions on $Q$. The rates are better
characterized in terms of the convergence of some bulk characteristics
(so-called weak convergence), rather than directly via the trajectories.

The convergence rates are crucial for any practical calculations. For instance,
if the difference of some characteristics of Markov chain $X^N_x(t)$ and its
limit $X_x(t)$ is of order $1/N$, then applying the approximation $X_x(t)$ to
the chain $X^N_x(t)$ yields the error of order $1\%$ for $N=100$ players and
the error of order $10\%$ for $N=10$ players. Thus the number $N$ does not have
to be very large for the approximation \eqref{eqdefgenmeanfieldchainkineq}
of 'infinitely many players' to give reasonable predictions.

An important question concerns the large-time-behavior of mean-field interacting particle systems with finite $N$.
It is usually possible to show that this system spends time $t$ of order of the number of particles, $t\sim N$, in
a neighborhood of a stable set of fixed points of dynamics  \eqref{eqdefgenmeanfieldchainkineq}.
Often more precise conclusions can be made.

\section{Migration via binary interaction}
\label{secmigbin}

An important particular case of the scheme above occurs from the setting of
{\it binary interactions}\index{binary interaction}.
Namely, let us assume that any player from any state $i$ can be influenced by any player in any other state $j$
to migrate from $i$ to $j$, the random events occurring with some given intensity $T_{ij}(n/N)/N$
(the appearance of the multiplier $1/N$ being the mark of the general approach according to which
each individual's contribution to the overall evolution becomes negligible in the limit of large number of players $N$).
More precisely, to any ordered pair of agents $A$, $B$ in different states, say $i \neq j$, one attaches random
$T_{ij}(n/N)/N$-exponential waiting time. The minimum of all these waiting time is again a random waiting time.
If the minimum occurs on a pair $(A,B)$ from some states $i$ and $j$ respectively, then $A$ migrates
from the state $i$ to the state $j$, and then the process continues from the state $n^{ij}$. Equivalently, this process
can be again described by one clock with $\sum_{i\neq j} n_in_jT_{ij}(n/N)/N$-exponential waiting time, attached to a state $n$.
When such clock rings,  the transition $i\to j$, $i\neq j$, is chosen with probability
\[
\P(i\to j)= \frac{n_in_jT_{ij}(n/N)/N}{\sum_{k\neq l} n_kn_l T_{kl}(n/N)/N}
=\frac{n_ix_jT_{ij}(x)}{\sum_{k\neq l} n_kx_l T_{kl}(x)}.
\]

As is seen directly, the process is the same as the one described above with the matrix $Q_{ij}(x)$ of the type
\[
Q_{ij}(x)= x_j T_{ij}(x), \quad i\neq j.
\]

\begin{remark}
The possibility of such full reduction of binary interaction to a mean-field interaction model is due to the
simple model of pure migration that we considered here, and it does not hold for arbitrary binary interactions.
\end{remark}

\begin{remark}
Another point to stress is that the ability of any pair to meet is the mathematical expression of full mixing.
If it were not assumed (say, agent spatially separated were not able to meet), the number of interactions would
not depend essentially on the size of the population, but rather on the size of a typical neighborhood,
and the modelling would be essentially different.
\end{remark}

\section{Introducing principals}
\label{secprinplaysim}

As was already mentioned, the next step is to include a major player or principal
that can influence the evolution of our mean-field interacting particle system.
As examples, one can think about the management of complex stochastic systems
consisting of large number of interacting components: agents, mechanisms, robots,
vehicles, subsidiaries, species, police units, etc. To take this into account, our
family of $Q(x)$ matrices should become dependent on some parameter $b$ controlled
by the principle: \{Q=Q(x,b)\}. This can be a real parameter (say, the budget used
by the principle for management) or a vector-valued parameter, for instance,
specifying resources for various directions of development. In any case, $b$ is
supposed to belong to some domain in a Euclidean space, usually bounded,
as the resources can not be infinite.

We shall  also assume that the principal has some objectives in this process, aiming to maximize
some profit or minimize certain costs arising from the functioning of the Markov chain.
In the simplest situation we can assume the principal to
be a {\it best response principal}\index{best response principal},
which instantaneously chooses the value of $b^*$  maximizing
 some current profit $B(x,b,N)$ for given $x, N$:
\begin{equation}
\label{eqbestrespprin}
b^*=b^*(x,N)=argmax B(x,.,N).
\end{equation}

This procedure reduces the dynamics to the previous case but with $Q^N(x)=Q(x, b^*(x,N))$,
which can now depend on $N$. Convergence \eqref{eqdynLLNprob} can still be shown if $b^*$
stabilizes as $N\to \infty$, that is, $b^*(x,N)\to b^*(x)$, as $N\to \infty$,
for some function $b^*(x)$. The limiting evolution  \eqref{eqdefgenmeanfieldchainkineq}
 turns to the evolution
\begin{equation}
\label{eqkineqprin}
\dot x_k =\sum_{i=1}^d x_i Q_{ik}(x, b^*(x)), \quad k=1,...,d.
\end{equation}

More realistic modeling would require the principal to choose the control parameter strategically
with some planning horizon $T$ and a final goal at the end. Moreover, the model should include costs
for changing strategies. This leads to a dynamic control problem, where the principal chooses
controls at fixed moments of time (discrete control setting) or in continuous time (continuous time setting).

Further extensions deal with several principals competing on the background of small
players, where the pool of small players can be common for the major agents (think
about advertising models, or several states fighting together a terrorist group)
or be specific to each major agent (think about generals controlling armies, big
banks controlling their subsidaries, etc). The development of such models and their
dynamic LLN (law of large numbers), leading to better manageable finite-dimensional
systems is the main concern in this book.

\section{Pressure and resistance framework}
\label{secpresres}

Let us present now a key particular performance of the above general scheme introduced
 by one of the authors in \cite{Ko17} and called there the {\it pressure and resistance game}.
In this setting the interests of the major player $P$ and the small players are essentially opposite, and
the mean-field dependence of the actions of small players arises from their binary
interactions. Namely, as in the above general setting, the strategies of player $P$ are
 given by the choices of controls $b(x,N)$, from a given  subset of Euclidean space, based
on the number of small players $N$ and the overall distribution $x$ of the strategies of small
 players. The evolution of the states of small players goes as in Section \ref{secmigbin} with
  $T_{ij}(x,b)$ depending also on $b$, so that

\begin{equation}
\label{eqbinwithprin}
Q_{ij}(x,b)= x_j T_{ij}(x,b), \quad i\neq j.
\end{equation}

The resulting limiting process described by ODE \eqref{eqkineqprin} becomes
\begin{equation}
\label{eqbinmiglimevol}
\dot x_j=x_j \sum_i x_i [T_{ij}(x,b^*(x))-T_{ji}(x,b^*(x))], \quad j=1,...,d.
\end{equation}

More concrete performance of this model is the following
{\it pressure and resistance model}\index{pressure and resistance game}.
Let us assume that each small player enters certain interaction with $P$ (small players try to resist
 the pressure exerted on them by $P$) defined by the strategy $j\in \{1, \cdots , d\}$
and that the result of this interaction is the payoff $R_j(x,b)$ to the small player.
Moreover, with rate $\ka/N$ ($\ka$ being some constant) any pair of agents can meet and discuss their payoffs,
which may result in the player with lesser payoff $R_i$ switching to the strategy with the better payoff $R_j$,
which may occur with a probability proportional to $(R_j-R_i)$.

In other words, this is exactly the scheme described above by \eqref{eqbinwithprin} with
\begin{equation}
\label{eqtransbinpresres}
T_{ij}(x,b)= \ka (R_j(x,b)-R_i(x,b))^+,
\end{equation}
where we used the standard notation $y^+=\max(0,y)$.
Therefore, this model is also a performance of the model of Section \ref{secprinplaysim} with
\begin{equation}
\label{eqQmapresres}
Q_{ij}(x,b)= \ka x_j(R_j(x,b)-R_i(x,b))^+.
\end{equation}
The resulting limiting process described by ODE \eqref{eqbinmiglimevol} becomes
\begin{equation}
\label{eqpresreslimevol}
\dot x_j=\sum_i \ka x_i x_j[R_j(x,b^*(x))-R_i(x,b^*(x))], \quad j=1,...,d.
\end{equation}

\begin{remark}
We are working here with a pure myopic behavior for simplicity. Introduction of random mutation
on global or local levels (see e. g. \cite{KaMaRo} for standard evolutionary games) would
not affect essentially the convergence result below,
but would lead to serious changes in the long run of the game, which are worth being exploited.
\end{remark}

Thus we are dealing with a class of models, where a distinguished 'big' player exerts certain level $b$ of pressure
on (or interference into the affairs of) a large group of $N$ 'small' players that can respond with a strategy $j$.
Often these $j$ are ordered reflecting the strength or the level of the resistance.
The term 'small' reflects the idea that the influence of each particular player becomes negligible as $N\to \infty$.
As an example of this general setting one can mention the interference of humans on the environment (say, by hunting or fishing)
or the use of medications to fight with infectious bacteria in a human body,
with resisting species having the choice of occupying the areas of ample foraging but more dangerous interaction with the big player
(large resistance levels $r$) or less beneficial but also less dangerous areas (low resistance level). Another example can be the level of
 resistance of the population on a territory occupied by military forces. Some key examples will be discussed in more detail in the next sections.

In many situations, the members of the pool of small players have an alternative class
of strategies of collaborating with the big player on various levels. The creation of
such possibilities can be considered as a strategic action of the major player (who can
thus exert some control on the rules of the game). In biological setting this is, for
instance, the strategy of dogs joining humans in hunting their 'relatives' wolves or
foxes (nicely described poetically as the talk between a dog and a fox in the famous
novel \cite{SBambi}). Historical examples include the strategy of slaves helping their
 masters to terrorize and torture other slaves and by doing this gaining for themselves
more beneficial conditions, as described e.g. in the classics \cite{TomCab}. As a military
example one can indicate the strategy of the part of the population on a territory occupied
by foreign militaries that joins the local support forces for the occupants, for US troops
in Iraq this strategy being well discussed in Chapter 2 of \cite{Mesq}. Alternatively, this
is also the strategy of population helping police to fight with criminals and/or terrorists.
In the world of organized crime it is also a well known strategy to play simultaneously
both resistance (committing crime) and collaboration (to collaborate with police to get
rid of the competitors), the classic presentation in fiction being novel \cite{Fieldi}.

In general pressure and resistance games, the payoff
$R_j(x,b)$ has often the following special features: $R$ increases
in $j$ and decreases in $b$. The dependence of $R$ and $b^*$ on $x$
is more subtle, as it may take into account social norms of various character.
In case of the pressure game with resistance and collaboration, the strategic
parameter $j$ of small players naturally decomposes into two coordinates $j=(r,c)$,
the first one reflecting the level of resistance and the second
the level of collaboration.
If the correlation between these activities are not taken into account
the payoff $R$ can be decomposed into the sum
of rewards $R=R^1_j(x,b)+R^2_j(x,b)$ with $R^1$
having the same features as $R$ above, but with
$R^2$ increasing both in $j$ and $b$.

Another extension to mention is the possibility to have different classes of small players, which
brings us to the domain of optimal allocation games, but now in the competitive evolutionary setting, where
the principal (say an inspector) has the task to distribute limited resources as efficiently as possible.

\section{Nash equilibria and the fixed points of kinetic equations}
\label{secNasheqfixedp}

Since the limiting behavior of the controlled Markov chain in the pressure-resistance setting
is given by the kinetic equations \eqref{eqpresreslimevol}, one can expect
that eventually the evolution will settle down near some stable equilibrium points of
this dynamic systems. This remark motivates the analysis of these equilibria.

Moreover, these equilibria have an important game-theoretic interpretation, which is
independent of the myopic hypothesis defining the Markov evolution.

Let us consider explicitly the following game $\Ga_N$ of $N+1$ players (that was tacitly borne in mind
 when discussing dynamics).
 When the major player chooses the strategy $b$ and
each of $N$ small players chooses the state $i$, the major player receives the payoff
$B(x,b,N)$ and each player in the state $i$ receives $R_i(x,b)$, $i=1, \cdots, d$ (as above, with $x=n/N$ and
$n=(n_1, \cdots , n_d)$ the realized occupation numbers of all the states). Thus a strategy profile
of small players  in this game can be specified either by a sequence of $N$ numbers
(expressing the choice of the state by each agent), or more succinctly, by the resulting collection
of frequencies $x=n/N$.

The key notion of the game theory is that of the {\it Nash equilibrium}\index{Nash equilibrium},
which is any profile of strategies $(x_N,b_N^*)$ such that for any player
changing its choice unilaterally would not be beneficial. For the game $\Ga_N$, this is
a profile of strategies $(x_N,b_N^*)$ such that
\[
b_N^* =b_N^*(x_N,N)= argmax \, B(x_N,b,N)
\]
and for any $i,j\in \{1, \cdots , d\}$
\begin{equation}
\label{eqNashevolprin10}
R_j(x-e_i/N+e_j/N,b_N^*)\le R_i(x,b_N^*).
\end{equation}

As finding Nash equilibria is a difficult task, one often preforms
approximate calculations via some numeric schemes leading to approximate
Nash equilibria or $\ep$-Nash equilibria A profile of strategies is an
$\ep$-{\it Nash equilibrium}\index{$\ep$-Nash equilibrium} if the above
conditions hold up to an additive correction term not exceeding $\ep$.

As it will be shown, the rest points of \eqref{eqpresreslimevol}
describe all approximate Nash equilibria for $\Ga_N$. These approximate
Nash equilibria are $\ep$-Nash equilibria with $\ep$ of order $1/N$.

Let us discuss the main concrete examples of the pressure and resistance framework.

\section{Inspection and corruption}

In the inspection game with a large number of inspectees, see \cite{KoPaYa},
any one from a large group of $N$ inspectees has a number of strategies
parametrized by a finite or infinite set of nonnegative numbers $r$ indicating
the level at which she chooses to break the regulations ($r=0$ corresponds to
the full compliance). These can be the levels of tax evasion, the levels
of illegal traffic through a check point, the amounts at which the arms
production exceeds the agreed level, etc. On the other hand, a specific player,
the inspector, tries to identify and punish the trespassers. Inspector's strategies
are real numbers $b$ indicating the level of her involvement in the search process,
for instance, the budget spent on it, which is related in a monotonic way to the
probability of the discovery of the illegal behavior of trespassers. The payoff
of an inspectee depends on whether her illegal behavior is detected or not. If
social norms are taken into account, this payoff will also depend on the overall
crime level of the population, that is, on the probability distribution of inspectees
playing different strategies. The payoff of the inspector may depend on the fines
collected from detected violators, on the budget spent and again on the overall crime
level (that she may have to report to governmental bodies, say).  As time goes by,
random pairs of inspectees can communicate in such a way that one inspectee of the
pair can start copying the strategy of another one if it turns out to be more beneficial.

 A slightly different twist to the model presents the whole class of games modeling
 corruption (see \cite{Aidt}, \cite{Jain}, \cite{LaMoMaRa09}, \cite{Mal14} and \cite{KolMa15}
and references therein for a general background). For instance, developing the initial simple
 model of \cite{Beck}, a large class of these games studies the strategies of a benevolent
principal (representing, say, a governmental body that is interested in the efficient development
of economics) that delegates a decision-making power to a non-benevolent (possibly corrupt) agent,
whose behavior (legal or not) depends on the incentives designed by the principal. The agent
can deal, for example, with tax collection of firms. The firms can use bribes to persuade
a corrupted tax collector to accept falsified revenue reports. In this model the set of
inspectors can be considered as a large group of small players that can choose the level of
corruption (quite in contrast to the classical model of inspection) by taking no bribes at all,
or not too much bribes, etc. The strategy of the principal consists in fiddling with two instruments:
choosing wages for inspectors (to be attractive enough, so that the agents should be afraid to
loose it) and investing in activities aimed at the timely detection of the fraudulent behavior.
Mathematically these two types are fully analogous to preemptive and defensive methods discussed
in models on counterterrorism, see below.

In the standard setting of inspection games with a possibly tax-evading inspectee
(analyzed in detail in \cite{KoPaYa} under some particular assumptions),
the payoff $R$ looks as follows:
\begin{equation}
\label{eqdefinsppay}
R_j(x,b) =r+(1-p_j(x,b))r_j-p_j(x,b) f(r_j),
\end{equation}
where $r$ is the legal payoff of an inspectee, various $r_j$
denote various amounts of not declared profit, $j=1, \cdots , d$,
$p_j(x,b)$ is the probability for the illegal behavior of an inspectee
to be found when the inspector uses budget $b$ for searching operation
and $f(r_j)$ is the fine that the guilty inspectee has to pay when being discovered.

In the standard model of corruption 'with benevolent principal', see e. g. \cite{Aidt},
one sets the payoff of a possibly corrupted inspector (now taking the role
of a small player) as
\[
(1-p)(r+w)+p(w_0-f),
\]
where $r$ is now the bribe an inspector asks from a firm to agree not to publicize its profit
(and thus allowing her not to pay tax), $w$ is the wage of an inspector, $f$ the fine she has to pay
when the corruption is discovered and $p$ the probability of a corrupted behavior to be discovered by the
benevolent principal (say, governmental official). Finally it is assumed that when the corrupted behavior
is discovered the agent not only pays fine, but is also fired from the job
and has to accept a lower level activity with the reservation wage $w_0$. In our strategic model
we make $r$ to be the strategy of an inspector with possible levels $r_1, \cdots, r_d$
(the amount of bribes she is taking) and the probability $p$ of discovery to be dependent
on the effort (say, budget $b$) of the principal and the overall level of corruption $x$,
with fine too depending on the level of illegal behavior. This natural extension of the standard model
 leads to the payoff
\begin{equation}
\label{eqdefcorrpay}
R_j(x,b) =(1-p_j(x,b))(r_j+w)+p_j(x,b)(w_0- f(r_j)),
\end{equation}
 which is essentially identical to \eqref{eqdefinsppay}.

\section{Cyber-security, biological attack-defence and counterterrorism modeling}
\label{seccyber-secint}

A 'linguistic twist' that changes 'detected agents' to 'infected agents' brings us directly
to the (seemingly quite different) setting of cyber-security or biological attack-defence games.

For instance, let us look at the applications to the botnet defense (for example, against the
famous conflicker botnet), widely discussed in the contemporary literature, since botnets
(zombie networks) are considered to pose the biggest threat to the international cyber-security,
see e. g. review of the abundant bibliography in \cite{BenKaHo}. The comprehensive game
theoretical framework of \cite{BenKaHo} (that extends several previous simplified models)
models the group of users subject to cybercriminal attack of botnet herders as a differential
game of two players, the group of cybercriminals and the group of defenders.

Our approach adds to this analysis the networking aspects by allowing the defenders to communicate
and eventually copy more beneficial strategies. More concretely, our general model of inspection or
corruption becomes almost directly applicable in this setting by the clever linguistic change of
'detected' to 'infected' and by considering the cybecriminal as the 'principal agent'!
Namely, let $r_j$ (the index $j$ being taken from some discrete set here, though more advanced
theory of the next sections allows for a continuous parameter $j$) denote the level of defense
applied by an individual (computer owner) against botnet herders (an analog of the parameter $\ga$
of  \cite{BenKaHo}), which can be the level of antivirus programs installed or the measures envisaged
to quickly report and repair a problem once detected (or possibly a multidimensional parameter
reflecting several defense measures).
Similarly to our previous models, let $p_j(x,b)$ denote the probability for a computer of being infected
given the level of defense measures $r_j$, the effort level $b$ of the herder (say, budget or time spent)
and the overall distribution $x$ of infected machines (this 'mean-field' parameter is crucial in the present setting,
since infection propagates as a kind of epidemic). Then, for a player with a strategy $j$,
the cost of being (inevitably) involved in the conflict can be naturally estimated by the formula
\begin{equation}
\label{eqdefcybersecpay}
R_j(x,b) =p_j(x,b)c+r_j,
\end{equation}
where $c$ is the cost (inevitable losses) of being infected (thus one should aim
at minimizing this $R_j$, rather then maximizing it, as in our previous models). Of course, one can extend the model
to various classes of customers (or various classes of computers) for which values of $c$ or $r_j$ may vary
and by taking into account more concrete mechanisms of virus spreading, as described e. g. in
\cite{LiLiSt} and \cite{LyWi}.

Similar models can be applied to the analysis of defense against a biological weapon, for instance by adding
the active agent (principal interested in spreading the disease), into the general mean-field epidemic model
of \cite{LiuTaYa} that extends the well established SIS (susceptible-infectious-susceptible) and
SIR (susceptible-infectious-recovered) models.

Yet another set of examples represent the models of terrorists' attacks and counterterrorism
 measures, see e. g. \cite{ArSa05},  \cite{SaAr03}, \cite{SandLa88}, \cite{BrKi88}
for the general background on game -theoretic models of terrorism, and \cite{FaAr12} for more
recent developments. We again suggest here a natural extension to basic models to the possibility
of interacting large number of players and of various levels of attacks, the latter extension being
in the line with argument from \cite{ClYoGl07} advocating consideration of 'spectacular attacks' as
part of a continuous scale of attacks of various levels.

In the literature, the counterterrorists' measures are usually
 decomposed into two groups, so called proactive (or preemptive),
like direct retaliation against the state-sponsor
and defensive (also referred to as deterrence), like strengthening security at an airport, with
the choice between the two considered as the main strategic parameter.
As stressed in \cite{RosSand04} the first group of action is 'characterized in
the literature as a pure public good, because a weakened terrorist group poses less of a
threat to all potential targets', but on the other hand, it 'may have a downside by creating more grievances in reaction
to heavy-handed tactics or unintended collateral damage' (because it means to
'bomb alleged terrorist assets, hold suspects without
charging them, assassinate suspected terrorists, curb civil freedoms, or impose retribution
on alleged sponsors'), which may result in the increase of terrorists' recruitment.
Thus, the model of \cite{RosSand04} includes the recruitment benefits of terrorists as a positively correlated function
of preemption efforts.

A direct extension of the model of \cite{RosSand04} in the line indicated above
(large number of players and the levels of attacks) suggests to write down the reward of a terrorist, or a terrorist group,
considered as a representative of a large number of small players, using one of the
levels of attack $j=1, \cdots, d$ (in \cite{RosSand04} there are two levels, normal and spectacular only), to be
\begin{equation}
\label{eqdefterrpay}
R_j(x,b) =(1-p_j(x,b))r_j^{fail}(b)+p_j(x,b)(S_j+r_j^{succ}(b)),
\end{equation}
where $p_j(x,b)$ is the probability of a successful attack (which depends on the level $b$ of preemptive efforts
of the principal $b$ and the total
 distribution of terrorists playing different strategies), $S_j$ is the direct benefits in case of a success and
 $r_j^{fail}(b)$, $r_j^{succ}(b)$ are the recruitment benefits in the cases of failure or success respectively.
 The costs of principal are given by
 \[
 B(x,b)=\sum_j x_j \left[(1- p_j(x,b))b+p_j(b)(b+S_j)\right].
 \]
 It is seen directly that we are again in the same situation as described by \eqref{eqdefcorrpay}
 (up to constants and notations). The model extends naturally to account for possibility of the actions of two types,
 preemption and deterrence. Also very natural can be the extension to several major players
For instance, USA and EU cooperation infighting terrorism was considered in \cite{ArSa05}.

\section{Coalition building, merging and splitting, network growth, queues}

In a large number of problems it is essential to work with an infinite state-space
of small players, in particular, with the state-space being the set of all natural numbers.
Mathematical results are much rare for this case, as compared with finite state-spaces.
This infinite-dimensional setting is crucial for the  analysis of models with growth,
like merging banks or firms (see \cite{Pushkin04} and \cite{SaMaSo})
or the evolution of species and the development of networks with preferential attachment
 (the term coined in \cite{BaAl99}), for instance scientific citation networks or the
network of internet links (see a detailed discussion in \cite{KraRed}).

Apart from the obvious economic examples mentioned above, similar process of the growth of
coalitions under pressure can be possibly used for modeling the development of human cooperation
(forming coalitions under the 'pressure' exerted by the nature) or the creation of liberation armies
(from the initially small guerillas groups) by the population of the territories oppressed by an
external military force.

Yet another example can represent the analysis of the financial support of various projects
(developmental or scientific), when the funding body has tools to merge or split projects and
 transfer funds between them in certain periods of time.

Our evolutionary theory of coalition building (properly developed in
Chapter \ref{secmodgrowthpres}) can be considered as an alternative to the well
 known theory of coalitional bargaining, see e.g. \cite{Chatterjee93}.

Models of growth are known to lead to power laws in equilibrium, which are verified in
a variety of real life processes, see e.g. \cite{SaMaSo} for a general overview and
\cite{Rich48} for particular applications in crime rates. In Chapter \ref{secmodgrowthpres}
 we will deal with the
response of a complex system, that includes both the ability of growth and the coalition
 building mechanisms, to external parameters that may be set by the principal (say,
by governmental regulations) who has her own agenda (may wish to influence the growth
of certain economics sectors).

 The modern literature on the network growth is very extensive and we will not try
 to review it referring for this purpose, for instance, to \cite{BellNet17} and references therein.

Of course these processes of growth have a clear physical analogs, say the formation
of dimers and trimers by the molecules of gas with eventual condensation
under (now real physical) pressure. The relation with the Bose-Einstein
condensation is also well known, see e. g. \cite{SimRoy}.

 We shall pay attention only to infinite spaces that are countable, which are
 possibly mostly relevant for practical studies of the evolutionary growth.
For this class of models the number $N$ of agents become variable (and usually
growing in the result of the evolution) and the major characteristics of the
system becomes just the distribution $x=(x_1, x_2, \cdots )$ of the sizes
of the groups. The analysis of the evolution of these models has a long history,
see \cite{SimRoy}, though in the abundant literature on the models of evolutionary
growth, the discussion usually starts directly with the deterministic limiting
model, with the underlying Markov model being just mentioned as a motivating
heuristics. Mathematically the analysis is similar to finite state spaces, though
 serious technical complications may arise. We develop the 'strategically enhanced
model' in Section \ref{secbirthdeath} analyzing such evolutions under the 'pressure'
of strategically varying parameters set by the principal.

As another insightful example of games with a countable state space let us mention the model of the
propagation of knowledge from \cite{Duffie} and \cite{CarDelbook18}. In this model an agent is
characterized by a number $n\in \N$, which designates the amount of information the agent has
about certain subject. Meetings of two agents with the information $n$ and $m$ occur with certain
intensities (which represent the control parameters) and result in each of them acquiring the total
information $m+n$. Summation of the information of agents is based, of course, on a highly idealized
assumption of the independent knowledge. Anyway, unlike the model of merging the number of agents
here remains constant, but their information (the analog of mass) is increasing.

Yet another big class of models with countable state spaces represent various models of queues parametrized
by the sizes of the queues at different servers, see e.g \cite{Ga17}, \cite{Ga18} and references therein for detail.

\section{Minority game, its strategic enhancement and sex ratio game}
\label{minorgamedef}

An important example of the pressure and resistance framework represents the famous
minority game designed to model the evolution of financial markets.
The remarkable properties of the minority game were revealed partially by the
extensive computer simulations and partially by the application of the methods
 of statistical mechanics.

 Recall that in the {\it minority game}\index{minority game}
 each of $N$ players declares (or adopts) independently and simultaneously one of the
 two possible strategies, denoted  $0$ and $1$. The payoff $1$ is paid to all players that
 declared the strategy that turned out to be in minority, and other players pay
  fine of amount $1$. In a rare event of equal amounts of declared two strategies, all
  payoffs can be set to zero. The game is supposed to be repeated many times, so that
  players can learn to apply better strategies.

 Let $n_0,n_1$ denote the amount of players that have chosen $0$ or $1$ respectively
 in one step of the game. Then $n_0+n_1=N$. Let $x=n_0/N$. We are exactly in the
 situation described in Sections \ref{secpresres} \ref{secNasheqfixedp}
(without the major player) with $d=2$ and the payoffs
\begin{equation}
\label{eqpresresminor}
R_1(x)=-R_0(x)=\, {\sgn} \,  (x-1/2)=\left\{
\begin{aligned}
& 1, \quad x>1/2 \\
& 0, \quad x=1/2 \\
& -1, \quad x<1/2. \\
\end{aligned}
\right.
\end{equation}

The only specific feature of this setting is the obvious discontinuity of the payoffs $R_j(x)$.
However, for large $N$ the chance of a draw is really negligible, so that the model will
be not changed essentially if one takes a smooth version of the payoffs,
\begin{equation}
\label{eqpresresminor1}
R_1(x)=-R_0(x)=\theta  (x-1/2),
\end{equation}
where $\theta(x)$ is a smooth non-decreasing function such that $\theta(x)=\, {\sgn} \,(x)$
for $|x|\ge \ep$ with some small $\ep$.
The kinetic equations \eqref{eqdefgenmeanfieldchainkineq}
can be written  down in terms of a single variable $x$ as
\begin{equation}
\label{eqpresresminor2}
\dot x=-x(1-x) \theta  (x-1/2).
\end{equation}

The analysis of the dynamics arising from this game is given in Sections
\ref{sectwostatemod} and \ref{secdiscontrates}.

Placing the minority game in this general setting suggests various reasonable extensions
of this market model, for instance, by including major players (say, a market makers) that
have tools to control the interactions or the payoffs of minor players, or by a more subtle
 grouping of players. Moreover, it reveals a nice link with the so-called {\it sex-ratio game}
 \index{sex-ratio game} of evolutionary biology. The latter game is designed to explain the
common $1/2$ probability for births of males and females in basically any population. The rules
of the sex-ratio game arise from the observation that, if there are $n_0$ males and $n_1$ females
in a population (its current generation), then producing a male (resp. female) offspring would
 yield to this offspring an expected number of $n_1/n_0$ (resp. $n_0/n_1$) of its own offsprings.
Thus we find ourselves in the situation described by the rules above with payoff
\eqref{eqpresresminor1} having the form
\begin{equation}
\label{eqpresressexra}
R_1(x)=\theta  (x-1/2)=x/(1-x).
\end{equation}

\section{Threshold collective behavior}
\label{secmobcontr}

A similar two-state model arises in a seemingly different setting of the models of
{\it threshold collective behavior}\index{threshold collective behavior}, for example,
the well known model of mob control (see \cite{BreerNovMob17} for the general introduction).
It is assumed there that a crowd consists of a large number $N$ of agents that can be in
two states, excited (say, aggressive) and quiet, denoted $1$ and $0$ respectively.
In discrete time modeling the excitation is supposed to propagate by copying the
neighbors. Namely, each agent $i$ is characterized by the threshold level
$\theta_i\in [0,1]$ and the influence vector $w_{ji}\ge 0$, $j\neq i$,
where $w_{ji}$ is the weight of the opinion of $i$ for agent $j$. If at
some time $t$ the states of the agents are given by the profile vector
$\{y_1^t, \cdots, y^t_N\}$, $y_j^t\in \{0,1\}$, an agent $i$ in the next
 moment of time $t+1$ becomes excited if the weighted level of excitation
of its neighbors exceeds the threshold $\theta_i$, that is, $y^{t+1}_i$
turns to $1$ or $0$ depending on whether $\sum_{j\neq i} w_{ij} y_j \ge \theta_i$
or otherwise, respectively. Thus each agent acts as a classical {\it perseptron}
(threshold decision unit) of the theory of neural networks. Quite similar are more
general models of the so called {\it active network structures} that are used to
model the propagation of opinion in social media like Facebook or Twitter.

Concerning $\theta_j$ it is often assumed that they are random identically
 distributed random variables with some distribution function
 $F(x)=\P(\theta_i \le x)$. To fit this process to our general model
 it is necessary to assume that the mob is decomposed in a finite number
of classes with the identical weights, so that $w_{ij}$ can be characterized
 as the weights of influence between representatives of different classes,
 rather that individual agents, and in fact the theory usually develops under
an assumption of this kind. The standard simplifying assumption is the
so-called {\it anonymity case}, where all influences are considered to be
identical, that is $w_{ij}=1/(N-1)$ for all $i\neq j$. Let us assume this
is the case. One can model the propagation of excitation both in discrete
or continuous time. For definiteness, let us choose the second approach
(more in line with our general modeling). Then the propagation is measured
in terms of the rates of transitions, rather than probabilities, and we
naturally assume here that the rate of transition from the state
 $0$ to the state $1$ is proportional to $\P(\theta_i\le x)=F(x)$
and from the state $1$ to the state $0$ to $\P(\theta_i>x)=1-F(x)$, where
$x\in [0,1]$ is the proportion of the excited states among the agents. Thus
we find ourselves  fully in our general setting with a two state model,
$d=2$, and the Kolmogorov matrix $Q_{ij}$, $i,j\in \{0,1\}$ with
\[
Q_{01}=\al F(x), \quad Q_{10}=\al (1-F(x)),
\]
with some constant $\al$, and thus $Q_{00}=-\al F(x), Q_{11}=-\al (1-F(x)$.
A constant $\al$ can be used to extend the model to the case when any agent is influenced only
 by the neighbors, where $\al$ can be characterized by the average number of these neighbors.

As for any model with $d=2$, the kinetic equations \eqref{eqdefgenmeanfieldchainkineq}
can be written in terms of just one equation for $x$:
\[
\dot x= xQ_{11}+(1-x)Q_{01}=-x\al (1-F(x))+(1-x) \al F(x)=\al (F(x)-x),
\]
that is, finally as
\begin{equation}
\label{eqmobthres}
\dot x=\al (F(x)-x),
\end{equation}
which is a standard equation used in modeling the
 threshold collective behavior (see \cite{BreerNovMob17}).

\section{Projects selection, investment policies and optimal allocation}
\label{projectselectdef}

As yet another example let us introduce a simple model for
{\it project selection}\index{project selection}. The specific feature of
 this model is the necessity to work with a two-parameter family of states.
 Assume an investment body aims at developing several directions
 $j\in \{1, \cdots , d\}$ of business or science, ordered according
to their importance, with some budget allocated to each of these directions.
When a call for bidding for the support of projects is declared, the
interested parties have to choose both the project to bid for and the level
of the investment of their efforts in preparation of the bid, having some
possible costs, $c\in \{c_1, \cdots, c_m\}$, with  $c_1< \cdots < c_m$.
The probability of a project being selected, $p(j,c,x_j)$,  will be thus
an increasing function of $j$, an increasing function of $c$ and a decreasing
function of $x_j=n_j/N$, the fraction of bids in the $j$th direction.
 Therefore if $A_j$ is an award (the value of a grant) for a project in the
$j$th direction, the average payoff of small players can be written as
 \begin{equation}
\label{eqinvestpolproj}
R_{j,c}(x)=-c+A_jp(j,c,x_j).
\end{equation}

The kinetic equations, see \eqref{eqdefgenmeanfieldchainkineq} and
\eqref{eqQmapresres}, in this two-dimensional index model take the form
 \begin{equation}
\label{eqinvestpolproj1}
\dot x_{jc}=\sum_{i,\tilde c} \ka x_{i\tilde c} x_{jc}(R_{i,\tilde c}(x)-R_{j,c}(x)).
\end{equation}

Here both $A_j$ and the functions $p(j,c,x)$ are the parts of the mechanism
design of the principal (investment body). They can be encoded, of course,
by certain parameters $b$ that are chosen by the principal.

The model is a performance of the problem of optimal allocation in a competitive environment.
It can be extended in various direction, for instance, by taking into account the initial expertise
of the bidders and the related costs for switching between the topics.

\section{Games on networks and positivity preserving evolutions}
\label{secposevolandnet}

Here we stress two rather trivial aspects of evolution \eqref{eqdefgenmeanfieldchainkineq}.
The first one is its universality as a positivity preserving evolution. Namely, it is not difficult to show
(see e.g. \cite{Ko10}, Sections 1.1 and 6.8) that any sufficiently regular transformation of $\R^d$
preserving positivity and the 'number of particles' $x_1+\cdots +x_n$, in other words any smooth
transformation of the simplex $\Si_d$, can be described by an ODE of type \eqref{eqdefgenmeanfieldchainkineq}.
This is in fact a performance of a more general representation result for positivity preserving transformations
of measures.

The second aspect is just a geometrical rewording. Following the general idea
of the graphical representation of Markov chains (see Section  \ref{secMarch})
we can  similarly draw a graph of the mean-field interacting particle system of
Section \ref{secmeanfieldch1},  which will be a graph with $d$ vertices and edges
$e_{ij}$ drawn whenever $Q_{ij}(x)$ does not vanish identically. The evolution
\eqref{eqdefgenmeanfieldchainkineq} can then be looked at as the evolution of
the 'weights' of the vertices. Due to this geometric interpretation  mean-field
 interacting particle system and their limits \eqref{eqdefgenmeanfieldchainkineq}
 are sometimes referred to as the {\it networked models}\index{networked models},
 see e.g. \cite{Nowzari17} for a concrete epidemic model presented in this way.

\section{Swarm intelligence and swarm based robotics}
\label{secswarm}

{\it Swarm based robotics}\index{swarm based robotics} is often understood as
an attempt to build algorithms or mechanisms
for the functioning of complex systems of interacting robots by the analogy with the
observed behavior in nature,  mostly social insects, like ants or bees. The evolutions
of such mechanisms are usually represented as mean-field interacting particle systems,
thus given approximately (in the LLN limit) by kinetic equations (positivity preserving
evolutions on the simplexes, see Section \ref{secposevolandnet} above). Typical example
(taken from \cite{BonaSwarmbook}) is as follows. Two types of workers (for instance,
so-called minor and major in some ant species) of amount $n_1$ and $n_2$ respectively,
 with the total number $N=n_1+n_2$, can perform a certain task (say, clean the nest
 from the dead bodies). The level of unfulfilled task $s$ increases in time during
functioning of the nest and decreases proportionally to the number of workers engaging in it:
  \begin{equation}
\label{eqevolant1}
 \dot s=\de -\al (N_1+N_2)/N,
 \end{equation}
 where $N_i$ is the number of $i$th type workers performing the job and $\de, \al$ some positive
constants. The workers of two types have propensity or inclination to do the task $T_1(s)$ and $T_2(s)$
 (mathematically the rates with which they take it) depending on the level of the stimulus specified
 by $s$ (say, difficulty to move in a nest full of rubbish). If $p$ is the probability to give up job
after starting it, the evolution of the frequencies $x_i=N_i/n_i$ is given by the (easy to interpret)
equations
\begin{equation}
\label{eqevolant2}
 \dot x_1=T_1(s)(1-x_1)-px_1, \quad \dot x_2=T_2(s)(1-x_2)-px_2.
 \end{equation}
For system \eqref{eqevolant1} - \eqref{eqevolant2} the rest points can be found explicitly.
More general situations (with several tasks and several types of workers) were analyzed numerically
giving reasonable agreement with the concrete observations of the ants and then were used to build
algorithms for robots.

A well known analytic model of collective behaviour of a swarm of
interacting Brownian oscillators was proposed by Y. Kuramoto, see
\cite{AcerOnKuram05} for its review and extensions. As was shown in
\cite{YinOnKuram12}, the corresponding dynamics can be also obtained
as the solution to a certain mean-field game.

The theory developed in this book offers the precise relations between the processes with a finite number of agents
and their LLN limits like \eqref{eqevolant1} - \eqref{eqevolant2}. Moreover, it suggests to choose the parameters
available to the major player (say, the designer of the swarm) strategically via forward looking perspectives
with finite or infinite horizon (in the spirit of Chapter \ref{strategicprincipal}).

\section{Replicator dynamics with mean-field dependent payoffs: controlled hawk and dove evolution}
\label{repdynmf}

As a final example let us touch upon the applications to the well established
theory of evolutionary games and evolutionary dynamics reducing the discussion
to the variations on most popular concrete model.

 Recall the celebrated {\it Hawk-and-Dove game}\index{Hawk-and-Dove game}, which
  is a symmetric two-player game with the matrix of payoffs
(the letters in the table denote the payoffs to the row player)
\[
 \begin{tabular}{|c|c|c|}
\hline
&hawk& dove\\
\hline
hawk &$H$&$V$\\
\hline
dove &0&$D$\\
\hline
\end{tabular}
\]

The game can be interpreted as being played by individuals sharing
common resource. If a hawk and a dove meet, the dove abandons the site
without fight leaving the whole resource, $V$, to the hawk. If two
 doves meet, they share resource yielding $D$ each (that could be
more or less than $V$ depending on the level of collaboration).
If two hawks meet, they fight inflicting strong damages to each
other so that the average payoff can become even negative.

The {\it replicator dynamics}\index{replicator dynamics}, which is one
of the most basic dynamics for the theory of evolutionary games
(see \cite{KolMal1} or \cite{HS} for general background) is given by
the equation
\[
\dot x=x(1-x) [x(H+D-V)-(D-V)],
\]
which is of course a particular case of dynamics \eqref{eqdefgenmeanfieldchainkineq},
\eqref{eqQmapresres}. This evolution may have only one internal rest point
(point of internal equilibria),
\[
x^*= \frac{V-D}{V-D-H}.
\]
If $V>D$, as is usually assumed, then $x^*\in (0,1)$ if and only if $H<0$.
In this case, this point is known (and easily checked) to be  a stable
equilibrium of the replicator dynamics.

\begin{remark}
This equilibrium is also known to be ESS (evolutionary stable equilibrium) of the
Hawk-and-Dove game (see e.g. \cite{KolMal1} or \cite{HS} for the general background on ESS).
\end{remark}

The replicator dynamics can be obtained as the dynamic LLN (law of large number) for a Markov
model of $N$ interacting species, which can use one of two strategies, $h$ or $d$ (hawk or dove),
with each pair meeting randomly and playing the game described by the Table above.  The payoff for each player in
this game of $N$ players, the fraction $x$ (resp. $1-x$) of which playing $h$ (resp. $d$), is
$Hx+V(1-x)$ for being a hawk and $D(1-x)$ for being a dove.

The point we like to stress here is the following. If there are only a few hawks
and plenty of dove to rob, the hawks get less impetus to fight between themselves
meaning that $H$ may increase with the decay of $x$, the fraction of species playing $h$.
Similarly, if there are too many doves on the  restricted resource area, they may start
behaving more aggressively meaning that $D$ may decrease with the decay of $x$.
Choosing linear functions to express this dependence in the simplest form,
the modified table of this {\it nonlinear Hawk-and-Dove game}
\index{Hawk-and-Dove game!nonlinear} becomes
\[
 \begin{tabular}{|c|c|c|}
\hline
&hawk& dove\\
\hline
hawk &$H+ax$ &$V$\\
\hline
dove &0&$D-bx$\\
\hline
\end{tabular}
\]
with some constants $a,b>0$.

In this setting the game of two players stops being defined independently, but
the game of $N$ players choosing one of the strategies $h$ or $d$ remains perfectly
 well defined (with the theory mentioned in Section \ref{secNasheqfixedp} and
 developed in Chapters  \ref{chbestres}, \ref{strategicprincipal} fully applied)
with the payoffs of hawks and doves becoming
\[
\Pi_h(x)=(H+ax)x+V(1-x), \quad \Pi_d(x)=(D-bx)(1-x).
\]

The replicator dynamics generalizes to the equation
\begin{equation}
\label{repdynnew}
\dot x=x(1-x) [x(H+ax +D-bx-V)-(D-bx-V)],
\end{equation}
which remains to be a particular case of dynamics \eqref{eqdefgenmeanfieldchainkineq},
\eqref{eqQmapresres}. The internal fixed points of this dynamics are the solutions of
 the quadratic equation
\begin{equation}
\label{repdynnewfix}
(a-b)x^2-(V-D-H-b)x+(V-D)=0.
\end{equation}
It is seen that, say for $a=b$, the internal equilibrium becomes
\[
x^*= \frac{V-D}{V-D-H-b},
\]
which moves to the left with the increase of $b$. When the denominator in this expression
becomes negative, the internal equilibrium gets lost and the pure strategy $d$ becomes the
stable equilibrium. On the other hand, if $b=0$, then the roots of the quadratic equation
\eqref{repdynnewfix} have asymptotic expressions
\[
x_1= \frac{V-D}{V-D-H}\left(1+a  \frac{V-D}{(V-D-H)^2}+O(a^2)\right),
\quad x_2 =\frac{1}{a} (V-D-H)+O(1)
\]
for small $a$. Therefore, for small $a$, the internal rest point starts moving go the
right remaining the only internal rest point. When $a$ increases various bifurcations
can occur. For instance, if $-H<a< (V-D-H)/2$, the internal rest points disappear. If
\[
\frac{V-D-H}{2}<a < \frac{(V-D-H)^2}{4(V-D)},
\]
two internal fixed points appear (both solutions of \eqref{repdynnewfix} belong to $(0,1)$),
the left being stable, and the right unstable, so that in total we get two stable equilibria,
one mixed and one pure hawk.

Similar shifts of equilibria and bifurcations occur when mean-field dependence
is introduced in other popular two-player games. It is interesting to see whether
this kind of bifurcations could be observed in real field experiments.

The model can be used to describe the transformations of behavioral patterns and
social norms in society, how doves (social behavior) develop themselves into hawks
(aggressive behavior, violence) and vice versa.
Moreover, making the parameter $a$, $b$ depending on the some control parameter
$u$ of a major player  allows one to introduce the model of controlled hawk and
dove evolution, with the principal having a tool to drive the equilibria
in the direction desired. Results of Chapter \ref{strategicprincipal} provide
the general technique for the analysis of such controlled processes.

\section{Fluctuations}
\label{secflucbr}

An important direction that we are not pursuing here is the analysis of fluctuations of the
random Markov processes of $N$ agents and the deterministic LLN provided by the kinetic equations.
This problem can be approached in two slightly different ways. Firstly one can start with
the generator of the $N$ agent process and look for the next order of its approximation (as
compared with the first order approximation used for obtaining the kinetic equations). Expanding
the increments of functions to the second order yields directly the second order operator, and hence
the diffusion process. Such diffusion process may capture much better the long time behavior
of the $N$ agent process, for instance implying the dying out of certain strategies, which was
demonstrated both theoretically and experimentally on many concrete biological systems
with the evolution described by the models of evolutionary games, see \cite{Dobr18} and references therein.
For models with countable state space this analysis was performed in \cite{Ko04}. An alternative
approach is to write down explicitly the parameters (e.g. the generator) of the
process of fluctuations and look for its limit under appropriate scaling. Such limit can be often
found to be described by a certain Gaussian process thus leading to the
{\it dynamic central-limit-type}\index{dynamic central limit theorem}
theorems. Such results were obtained for various models, see references in Chapter \ref{chapbibl}.
Both approaches can be looked at as providing a stochastic law of large numbers for process
involving $N$ agents. Unlike deterministic LLN, these random LLN are less robust and much more
sensitive to the methods of approximations and scaling, with a variety of limits that can be
obtained depending on the choice of the latter.

\section{Brief guide to the content of the book}
\label{secguidetocont}

As was already mentioned, the first part of the book deals with the models of a large number $N$ of
minor players interacting with major players, the strategies of small players propagating
by direct imitation (myopic behavior) of more successful strategies of the surrounding players.
The purpose of the analysis is to turn the 'curse of dimensionality' into the 'blessing of
dimensionality' by obtaining better manageable deterministic limits, that is, kinetic equations,
as $N\to \infty$, and providing effective rates of convergence. Numerous examples of the concrete
models were presented above.

The first part is mostly based on papers \cite{Ko12}, \cite{Ko17},
with a gap from \cite{Ko12} corrected (see Remark \ref{remongapinme})
and with the results of \cite{Ko17} improved and strengthened in many ways,
specifically for countable state-spaces.

In Chapter \ref{chbestres} we shall work with the case of a finite number $d$
of strategies of small players and major players acting with the instantaneous
best response. Together with convergence rates to the deterministic limit we
present the game -theoretic interpretation of the fixed points of the limiting
dynamics providing the link between the large time behavior of the limiting,
$N\to \infty$, and approximating, finite $N$, dynamics. The results of the last
Sections \ref{stableanalin} - \ref{secfracLLN} are essentially new, though the
last section can be looked at as a performance of the general results from \cite{Ko07a} or
\cite{Meerbook} in the specific setting of interacting particles.

In Chapter \ref{strategicprincipal} the behavior of major players is modified
by a more realistic requirement that they choose their control parameters
strategically aiming at some optimal payoff during either a fixed finite
period or over infinite horizon. The dynamics is considered in the
multi-step version, where major players choose their controls at discrete
times $k\tau$, $k\in \N$, with a fixed $\tau$, and in the continuous-time limit,
obtained from the former by letting $\tau \to 0$. The deterministic continuous-time
problem is obtained when $\tau \to 0$ and $N\to \infty$ simultaneously subject to
some bounds on the dependence of $\tau$ and $N$. Long time behavior is analyzed
both for the discounted payoff and for the average payoff, the latter often leading
to the so-called turnpike behavior of the trajectories, well known in classical
mathematical economics.

Chapter \ref{secmodgrowthpres} extends the results of Chapters \ref{chbestres}
and \ref{strategicprincipal} to the case of countable state-space of small players.
This extension is carried out in order to include important models of evolutionary
coalition building and strategically enhanced preferential attachment. The mathematics
of this chapter is more demanding, as it heavily exploits the theory of
infinite-dimensional ODEs. The results of Section \ref{secunboundgrow} are new.

In the second part we model small players as rational optimizers bringing the
analysis into the realm of the so-called mean-field games (MFG), widely discussed
in modern research literature. We start in Chapter \ref{chapthreestate} with some
general introduction to the MFGs on finite state spaces including the theorem on
the approximations of finite-number-of-player games by the limiting MFG evolutions
following mostly the papers \cite{BasRa14}, \cite{BasRa17}, but with significant
improvements and simplifications.

The rest of the second part is based essentially
on papers \cite{KolMa15}, \cite{KolMa17}, \cite{KoBens} and \cite{KatKol}. As the
complexity of MFGs increases with the increase of the state-space of small players
we start with the detailed description of toy models of three-state and four-state
games in Chapters \ref{chapthreestate} and \ref{chapfourstate} respectively. These
games represent remarkable and rare examples, where all stationary solutions to the
MFG consistency problem can be found explicitly. The analysis of these toy models
reveals and provides explicit description of the nontrivial properties of MFGs, like
phase transitions, and of the key parameters that control these transitions. It also
signifies a reasonably large class of models, where the state-space of small players
is two-dimensional, one direction being controlled by the individual player (say, the
level of tax evasion or the costs spent on a defence system) and another by the
principal or binary interactions (say, the position of agent on the bureaucratic
staircase or the state of computer being infected or not).

In Chapter \ref{chapturnpike} a more general theory is developed revealing
the turnpike property in the MFG context.
Roughly speaking the turnpike property means that time-dependent control evolutions
with large horizon spend most of the time around some fixed stationary solution
(a turnpike), which leads to the crucial simplification of the analysis,
because stationary solutions are usually much easier to identify and calculate.

A very detailed analytic description of models is not achievable for models with larger state-spaces.
For such models numeric calculations can be effectively exploited for getting the detailed picture
 of the dynamics, its phase portrait and bifurcations. We are not touching numerics in this book.
 Instead of numerics we employ the asymptotic analysis to describe the dynamics under reasonable
 asymptotic regimes. This approach is akin to physics, where the search for a relevant small
 parameter is the key tool for getting insights into a complicated phenomenon. We identify and exploit
 three natural small parameters in our MFGs leading to the three asymptotic regimes: fast execution
 of personal decisions (why should one wait long to execute one's own decisions?), small discounts
 (means in practical terms that the planning horizon is not very small) and small coupling constants
  of binary interaction, the last one being of course the most standard in physics.

Finally we present some literature review indicating main trends
both in our main methodology (dynamic law of large numbers) and in the
concrete areas of applications (inspection, corruption, etc),
most closely related to our methods and objectives.
The literature is quite abundant and keeps growing rapidly. 



\chapter{Best response principals}
\label{chbestres}

In this chapter our basic tool, the dynamic law of large numbers
(LLN, whose descriptive explanation was given in Chapter \ref{chapintkolMaCor}),
is set on the firm mathematical ground. We prove several versions of this LLN, with different
regularity assumptions on the coefficients, with and without major players, and finally with
a distinguished (or tagged) player, the latter version being used later in Part II.
The convergence is proved with rather precise estimates of the error terms, which is of course
crucial for any practical applications. For instance, we show that in the case of smooth
 coefficients, the convergence rates, measuring the difference
between the various bulk characteristics of the dynamics of $N$ players and the limiting evolution
(corresponding to the infinite number of players), are of order $1/N$. For example, for
$N=10$ players this  difference is about $10\%$, showing that the number $N$ does not have
to be 'very large' for the approximation \eqref{eqdefgenmeanfieldchainkineq}
of 'infinitely many players' to give reasonable predictions.

In the presence of a major player, the theory is built here for the model of
what can be called the {\it best response principal}\index{best response principal},
where the major player chooses the strategy at any time as the instantaneous
best response (in case of several major players as an instantaneous Nash equilibrium)
 for the current distribution of small players. More realistic forward
looking major player will be analysed in the next chapter.

Further in this chapter
we develop the theory of approximate Nash equilibria for our mean-field interacting
systems of a large number of agents showing that these equilibria can be approximated
by the rest points of the system of the limiting kinetic equations. This result yields a
remarkable static game-theoretic interpretation of the rest points of kinetic equations,
which is independent of the myopic hypothesis and migration models used in the derivation
of the latter. When these rest points are dynamically stable, we show further that the
Markov evolution of the systems of many agents spends large periods of time in a neighborhood
of these rest points. Moreover, in some cases, the connection becomes even deeper: the supports
of the invariant measures of the Markov evolutions of $N$ agents turn out to be concentrated
around these  fixed points. This result makes the crucial extension of the LLN limit from
finite to infinite times.

Next we extend our modelling to the case of discontinuous coefficients linking uniqueness
and non-uniqueness of the solutions to the kinetic equations with the problem of the existence
of the invariant measures for the prelimiting game of a finite number of agents.
Discontinuous coefficients arise naturally, when small payers are considered as rational optimizers
(that is, in the setting of MFG) with a finite-valued control parameter, switching between
these values creating discontinuities, separating regular regimes.

Finally we extend the modelling to the so-called semi Markov chains allowing for the waiting
times between the jumps to be non-exponential, but to have heavy power tails. In this case the limiting
dynamics turns to the fractional PDE, with the solution process obtained by certain random time change of
the solutions of the initial kinetic equations. In particular, the dynamic LLN limit becomes a random
process, unlike all other situations discussed in this book. This brings the theory to the realm of the
fractional calculus, which becomes one of the major trends in the modern application of analysis to physics and
economics. In fact, many observations of naturally occurring processes give strong evidence of the presence
of heavy tails in the distribution of time between various transitions (see e. g. \cite{UchBook} or \cite{West}).

\section{Preliminaries: Markov chains}

Two basic notions that are crucial for our analysis are Markov chains and the characteristics
 of the first order linear partial differential equations (PDE). We briefly recall now these notions
(with more analytic detail than in Section \ref{secMarch}) referring for systematic exposition to
numerous textbooks, see e.g. \cite{Nobook} or \cite{Her2} for Markov chains
 and \cite{Petr1}, \cite{Petr2}, \cite{Konewbook} for ODEs and PDEs.

 Let us first recall the basic setting of continuous time Markov chains with
a finite state space $\{1,...,d\}$, which can be interpreted as the types of agents or
particles (say, possible opinions of individuals on a certain subject, or the levels of
fitness in a military unit, or the types of robots in a robot swarm). A Markov chain
on $\{1,...,d\}$ is specified by the choice of a $Q$-{\it matrix}
or a {\it Kolmogorov's matrix}\index{Kolmogorov's matrix} $Q$, which is a $d\times d$
square matrix such that its non-diagonal elements are non-negative  and the elements of
each row sum up to zero (and thus the diagonal elements are non-positive). As was mentioned
in Chapter \ref{chapintkolMaCor}, the {\it Markov chain}\index{Markov chain} with the
$Q$-matrix $Q$ is the process evolving by the following rule. Starting from any time $t$ and
a current state $i$ one waits a $|Q_{ii}|$-exponential random waiting time
$\tau$ and then the position jumps to a state $j$ according to the distribution
 $Q_{ij}/|Q_{ii}|$. Then at time $t+\tau$ the procedure starts again from position $j$, etc.

 To work with the Markov chain we need more detail of its analytic description.
 Let $X_i(t)$ denote the position of the chain at time $t\ge 0$, if it started
 at $i$ at the initial time $t=0$, and let $P_{ij}(t)$ denote the probability
 of the transfer from $i$ to $j$ in time $t$, so that
 \[
 \P(X_j(t)=j)=P_{ij}(t).
 \]
Then one can show that these {\it transition probabilities}\index{transition probabilities}
$P(t)=\{P_{ij}(t)\}$ satisfies the following
{\it Kolmogorov's forward equations}\index{Kolmogorov's forward equations}
 \[
\frac{d}{dt} P_{ij}(t)=\sum_{l=1}^d Q_{lj}P_{il}(t), \quad t\ge 0,
\]
or in matrix form
\[
\frac{d}{dt} P(t)=P(t) Q.
\]
On the other hand, the evolution of averages,
\[
T^tf(i)=\E f(X_i(t))=\sum_{j=1}^d P_{ij}(t)f(j)
\]
for any function $f$ on the state space $\{1,...,d\}$, i.e. for any vector $f\in \R^d$,
satisfies the  {\it Kolmogorov's backward equations}\index{Kolmogorov's backward equations}
\[
\frac{d}{dt} T^tf(i)=\sum_{j=1}^d Q_{ij}T^tf(j), \quad t\ge 0,
\]
or in the vector form
\[
\frac{d}{dt} T^tf=Q T^tf.
\]
Consequently $T^t$ can be calculated as the matrix exponent:
\[
T^tf=e^{tQ}f.
\]
\begin{remark} The values of the function $f$ above will be denoted both by $f_n$ and $f(n)$ reflecting
two interpretations of $f$ as a vector in $\R^d$ and a function on the state space.
\end{remark}

The matrix $Q$ considered as the linear operator in $\R^d$ is called the
{\it generator}\index{Markov chain!generator of} of the Markov chain $X_i(t)$
and the operators $T^t$ in $\R^d$ are called the {\it transition operators}.
These operators form a semigroup, that is, they satisfy the equation $T^tT^s=T^{t+s}$
 for any $s,t>0$, which is directly seen from the exponential representation $T^t=e^{tQ}$.

These definitions have natural extensions for the case of time dependent matrices $Q$.
Namely, let $\{Q (t)\}=\{(Q_{ij})(t)\}$ be a family of $d\times d$ square $Q$-matrices
or Kolmogorov matrices depending piecewise continuously on time
$t\ge 0$. The family $\{Q (t)\}$ specifies a (time non-homogeneous)
{\it Markov chain}\index{Markov chain} $X_{s,j}(t)$ on the state
space $\{1,...,d\}$, which is the following process. Starting from any time $t$ and
a current state $i$ one waits a $|Q_{ii}|(t)$-exponential random waiting time.
After such time $\tau$ the position jumps to a state $j$ according to the distribution
 $(Q_{ij}/|Q_{ii}|)(t)$. Then at time $t+\tau$ the procedure starts again from position $j$, etc.

Let us denote $X_{s,j}(t)$ the position of this process at time $t$ if it was
initiated at time $s$ in the state $i$. The {\it transition probabilities}
\index{transition probabilities}  $P(s,t)=(P_{ij}(s,t))_{i,j=1}^d$, $s\le t$,
defining the probabilities to migrate from $i$ to $j$ during the time segment
$[s,t]$, are said to form the {\it transition matrix} and the corresponding operators
\[
U^{s,t}f(i)=\E f(X_{s,i}(t))=\sum_{j=1}^d P_{ij}(s,t)f(j)
\]
are called the {\it transition operators}\index{transition operators} of the Markov chain.
The matrices $Q(t)$ define the time-dependent generator of the chain
acting in $\R^d$ as
\[
(Q(t)f)_n=\sum_{m\neq n}  Q_{nm}(t)(f_m-f_n), \quad f=(f_1, \cdots, f_d).
\]

The transition matrices satisfy the {\it Kolmogorov's forward equation}
\index{Kolmogorov's forward equations}
\[
\frac{d}{dt} P_{ij}(s,t)=\sum_{l=1}^d Q_{lj}(t)P_{il}(s,t), \quad s\le t,
\]
and the transition operators of this chain
 satisfy the {\it chain rule}\index{chain rule}, also called the {\it Chapman-Kolmogorov equation}\index{Chapman-Kolmogorov equation}, $U^{s,r}U^{r,t}=U^{s,t}$ for $s\le r \le t$,
 and the {\it Kolmogorov backward equations}\index{Kolmogorov's backward equations}
\[
\frac{d}{ds} (U^{s,t}f)(i)=-\sum_{j=1}^d Q_{ij} (U^{s,t}f)(j), \quad s\le t.
\]
A two-parameter family of operators, $U^{s,t}$, $s\le t$, satisfying the chain rule,
is said to from a (backward) {\it propagator}\index{propagator}.

Concerning discrete time Markov chains $X_i(t)$ specified by the transition matrix $P=\{P_{ij}\}$
(briefly mentioned in Section \ref{secMarch})
 let us recall the evident fact that the probabilities $P_{ij}^n$ of transitions $i\to j$ in time
 $n$ (i.e. the probabilities of being in $j$ at time $t$ conditioned on the initial state $i$ at time $0$)
form the matrix $P^n=\{P_{ij}^n\}$, which is the power of the transition matrix $P$: $P^n=(P)^n$, $n=0,1, \cdots$.
The corresponding transition operators describing the dynamics of the averages act as the multiplication by $P^n$,
and hence will be also denoted by $P^n$:
\[
P^nf(i)=\E f(X_i(n))=(P^n f)(i)=\sum_j P^n_{ij} f(j).
\]

A trivial but important modification to be mentioned is the similar setting
but with time between jumps being any fixed time $\tau$. This modification
allows one to establish close link between discrete and continuous time modeling.
Namely, for a continuous time Markov chain with the transition operators $T^t$
specified by the $Q$-matrix $Q$, let
\[
\tau<(\max_i|Q_{ii}|)^{-1}.
\]
Then one can define the discrete-time Markov chain on the same state space with the
transition matrix $P^{\tau}=\{P^{\tau}_{ij}\}$, where
\begin{equation}
\label{eqdiscanfconttimeMC}
P^{\tau}_{ii}=1-\tau |Q_{ii}|, \quad P^{\tau}_{ij}=\tau Q_{ij}, \quad j\neq i.
\end{equation}
In matrix form this rewrites as
\[
P^{\tau}=\1+\tau Q,
\]
where $\1$ denotes the unit matrix. Hence
\[
(P^{\tau})^n=(\1+\tau Q)^n.
\]
Consequently, if $\tau \to 0$ and $n\to \infty$ in such a way that $n\tau \to t$,
\begin{equation}
\label{eqdiscanfconttimeMC1}
\lim(P^{\tau})^n=e^{tQ}=T^t,
\end{equation}
yielding the important link between the semigroup $T^t$ and its discrete-time approximations $P^{\tau}$.

Extension of these results to time-dependent chains (with $P$ depending explicitly on $t$)
is more or less straightforward. The transition probabilities just become time dependent:
\begin{equation}
\label{eqdiscanfconttimeMC2}
P_{ii}^{\tau,t} =1-\tau \sum_i |Q_{ii}(t)|, \quad P_{ij}^{\tau,t}=\tau Q_{ij}(t), \quad i\neq j,
\end{equation}
or $P^{\tau,t}=\1+\tau Q(t)$.

\section{Preliminaries: ODEs and first order PDEs}

Let us turn to the auxiliary results on ODEs and linear PDEs that we shall need.
The results are mostly standard (see e.g. \cite{Petr1}, \cite{Petr2} and \cite{Konewbook})
and we will not prove them, but collect them in the form
convenient for direct referencing.

Let $\dot x=g (x)$ be a (vector-valued) ordinary differential equation (ODE) in $\R^d$. If
 $g$ is a Lipschitz function, its solution $X_x(t)$ with the initial condition $x$ is known to be well defined for
all $t\in \R$. Hence one can define the lifting of this evolution on functions:
\[
T^tf(x)=f(X_x(t)).
\]
These {\it transition operators} act as continuous contraction operators on the space $C(\R^d)$
and in its subspace $C_{\infty}(\R^d)$,
and they form a group, i.e. $T^tT^s=T^{s+t}$ for any $s,t$.
The operator
\[
Lf(x)= g(x) \cdot \frac{\pa f}{\pa x}
\]
is called the {\it generator} of this group, because for any $f\in C_{\infty}(\R^d)\cap  C^1(\R^d)$
\[
\frac{d}{dt}|_{t=0} T^tf=Lf.
\]

If $g\in C^1(\R^d)$, then
\begin{equation}
\label{eqPDS1storlin0}
\frac{d}{dt} T^tf=LT^tf=T^tLf
\end{equation}
for any $t$, so that the function $S(t,x)=T^tf(x)=f(X_x(t))$ satisfies
the linear first order partial differential equation (PDE)
\begin{equation}
\label{eqPDS1storlin}
\frac{\pa S}{\pa t} -g(x) \cdot \frac{\pa S}{\pa x}=0,
\end{equation}
with the initial condition $S(0,x)=f(x)$ for any $f\in C_{\infty}(\R^d)\cap  C^1(\R^d)$.
 Solutions $X_x(t)$ to the ODE $\dot x=g(x)$ are referred
 to as the {\it characteristics} of PDE \eqref{eqPDS1storlin}, because the solutions to \eqref{eqPDS1storlin}
 are expressed in terms of $X_x(t)$.

Let us recall the basic result on the smooth dependence of the solutions $X_x(t)$
on the initial point $x$ and the implied regularity of the solutions to
\eqref{eqPDS1storlin} (key steps of the proof are provided in Section \ref{secprelonl1}
in a more general case).

\begin{prop}
\label{ODEsmoo}
(i) Let $g\in C_{bLip}(\R^d)$. Then $X_x(t)\in C_{bLip}(\R^d)$ as a function of $x$ and
\begin{equation}
\label{eq0ODEsmoo}
\| X_x(t)\|_{Lip} \le \exp\{ t \| g\|_{Lip} \}.
\end{equation}

(ii) Let $g\in C^1(\R^d)$. Then $X_x(t)\in C^1(\R^d)$ as a function of $x$ and
(recall that  $\|g\|_{Lip}= \| g^{(1)}\|$)
\begin{equation}
\label{eq1ODEsmoo}
\left\| \frac{\pa X_x(t)}{\pa x}\right\|=\sup_{j,x} \left\|\frac{\pa X_x(t)}{\pa x_j}\right\|
\le \exp\{ t \| g^{(1)}\|\}=\exp\{ t \sup_{k,x} \left\| \frac{\pa g(x)}{\pa x_k}\right\| \}.
\end{equation}
Moreover, if $f\in C^1(\R^d)$, then $T^tf\in C^1(\R^d)$ and
\begin{equation}
\label{eq2ODEsmoo}
\| (T^tf)^{(1)}\|=\sup_{j,x} \left|\frac{\pa}{\pa x_j} f(X_x(t))\right|
\le \|f^{(1)}\| \,  \left\| \frac{\pa X_x(t)}{\pa x}\right\|
\le \|f^{(1)}\| \exp\{ t \|g^{(1)}\|\}.
\end{equation}

(iii)  Let $g\in C^2(\R^d)$. Then $X_x(t)\in C^2(\R^d)$ as a function of $x$ and
\begin{equation}
\label{eq3ODEsmoo}
\left\| \frac{\pa ^2X_x(t)}{\pa x^2}\right\|=\sup_{j,i,x} \left\|\frac{\pa^2 X_x(t)}{\pa x_i \pa x_j}\right\|
\le t \|g^{(2)}\| \exp\{ 3t \|g^{(1)}\|\}.
\end{equation}
Moreover, if $f\in C^2(\R^d)$, then $T^tf\in C^2(\R^d)$ and
\begin{equation}
\label{eq4ODEsmoo}
\| (T^tf)^{(2)}\|=\sup_{j,i,x} \left|\frac{\pa^2}{\pa x_j \pa x_i} f(X_x(t))\right|
\le  \|f^{(2)}\|  \exp\{ 2t \|g^{(1)}\|\}
+ t \|f^{(1)}\| \, \|g^{(2)}\| \exp\{ 3t \|g^{(1)}\|\}.
\end{equation}
\end{prop}

Let us also quote the result on the continuous dependence of the solutions on the r.h.s. and initial data.

\begin{prop}
\label{ODErhs}
Let $g_i\in C_{bLip}(\R^d)$, $i=1,2$ and let $X^i_x(t)$ denote the solutions
of the equations  $\dot x=g_i (x)$. Let $\|g_1-g_2\|\le \de$. Then
\begin{equation}
\label{eqODErhs}
| X^1_x(t)-X^2_y(t)| \le (t \de+|x-y|) \exp\{ t \| g_i\|_{Lip} \}, \quad i=1,2.
\end{equation}
\end{prop}

As follows from \eqref{eqPDS1storlin0},
\[
\frac{T^tf-f}{t} \to Lf, \quad t\to 0,
\]
for any $f\in C_{\infty}(\R^d)\cap  C^1(\R^d)$. As a corollary to previous
results let us obtain the rates of this convergence.

\begin{prop}
\label{PDEgenconv}
Let $g\in C^1(\R^d)$ and $f\in C_{\infty}(\R^d)\cap  C^2(\R^d)$. Then
\begin{equation}
\label{eq1PDEgenconv}
\|T^tf-f\|\le t\|g\|\, \|f\|_{bLip},
\end{equation}
\begin{equation}
\label{eq2PDEgenconv}
\|\frac{T^tf-f}{t} - Lf\|
\le t\|g\| (\|g\|\, \|f^{(2)}\|+\|g\|_{Lip} \|f\|_{Lip})
\le t\|g\|\, \|g\|_{bLip} \|f\|_{C^2}.
\end{equation}
\end{prop}

\begin{proof}
By \eqref{eqPDS1storlin0},
\[
\|T^tf-f\|=\|\int_0^t T^sLf \, ds\| \le t \|Lf\|,
\]
implying \eqref{eq1PDEgenconv}. Next, by \eqref{eqPDS1storlin0},
\[
\frac{T^tf-f}{t} - Lf=\frac{1}{t}\int_0^t T^sLf \, ds - Lf=\frac{1}{t}\int_0^t (T^sLf-Lf) \, ds.
\]
Applying \eqref{eq1PDEgenconv} to the function $Lf$ yields the first inequality in \eqref{eq2PDEgenconv},
the second inequality being a direct consequence.
\end{proof}

A more general setting of non-autonomous equations $\dot x=g (t,x)$ is required, where
 $g$ is Lipschitz in $x$ and piecewise continuous in $t$. Here the solutions  $X_{s,x}(t)$
 with the initial point $x$ at time $s$ are well defined and the transition operators
 form a two-parameter family:
\begin{equation}
\label{eqPDS1storlin1}
U^{s,t}f(x)=f(X_{s,x}(t)).
\end{equation}
These operators satisfy the chain rule $U^{s,r}U^{r,t}=U^{s,t}$ and
the function $S(s,x)=U^{s,t}f(x)$ satisfies the PDE
\begin{equation}
\label{eqPDS1storlin2}
\frac{\pa S}{\pa s} +g(s,x) \cdot \frac{\pa S}{\pa x}=0,
\end{equation}
with the initial (or terminal) condition $S(t,x)=f(x)$. As mentioned above,
the operators $U^{s,t}$ satisfying the chain rule are said to form a
{\it propagator}\index{propagator}. Moreover, these operators form a
{\it Feller propagator} meaning that they also act in the space $C_{\infty}(\R^d)$
and depend strongly continuous on $s$ and $t$. That is, if $f\in C_{\infty}(\R^d)$,
then $U^{s,t}f(x)$ is a continuous function $t\mapsto U^{s,t} \in C_{\infty}(\R^d)$
for any $s$ and a continuous function $s\mapsto U^{s,t} \in C_{\infty}(\R^d)$ for any $t$.

 The solutions $X_{s,x}(t)$  to the ODE $\dot x=g (t,x)$ are called the
 {\it characteristics of PDE}\index{characteristics of PDE} \eqref{eqPDS1storlin2}.
 The results of Propositions \ref{ODEsmoo} and \ref{ODErhs} (and their proofs) hold for
 the equation $\dot x=g (t,x)$, if all norms of functions of $t,x$ are understood as
 $\sup_t$ of their norms as functions of $x$, say
 \begin{equation}
\label{eqtimedepnorm}
 \|g\|_{Lip}=\sup_t \|g(t,.)\|_{Lip}, \quad \|g^{(2)}\|=\sup_t \|g^{(2)}(t,.)\|.
 \end{equation}
With this agreement the following estimates hold:
\begin{equation}
\label{eq0ODEsmootd}
\| X_{s,x}(t)\|_{Lip} \le \exp\{ (t-s) \| g\|_{Lip} \},
\end{equation}
\begin{equation}
\label{eq2ODEsmootd}
\| (U^{s,t}f)^{(1)}\|=\sup_{j,x} \left|\frac{\pa}{\pa x_j} f(X_{s,x}(t))\right|
\le \|f^{(1)}\| \exp\{ (t-s) \|g^{(1)}\| \},
\end{equation}
\begin{equation}
\label{eq3ODEsmootd}
\left\| \frac{\pa ^2X_{s,x}(t)}{\pa x^2}\right\|\le (t-s) \|g^{(2)}\| \exp\{ 3(t-s) \|g^{(1)}\| \},
\end{equation}
\begin{equation}
\label{eq4ODEsmootd}
\| (U^{s,t}f)^{(2)}\|\le  \|f^{(2)}\|  \exp\{ 2(t-s) \|g^{(1)}\|\}
+ (t-s) \|f^{(1)}\| \, \|g^{(2)}\| \exp\{ 3(t-s)  \|g^{(1)}\|\},
\end{equation}
with the derivatives on l.h.s. existing whenever the derivatives on the r.h.s. exist.

As is easily seen, for propagators estimate
\eqref{eq1PDEgenconv} still holds, that is,
\[
\|U^{s,t}f-f\|\le (t-s)\|g\|\, \|f\|_{bLip},
\]
and in estimate \eqref{eq2PDEgenconv} an additional term appears on the r.h.s. depending on
the continuity of $g$ with respect to $t$. For instance, if $g$ is Lipschitz in $t$ with the constant
\[
\ka =\sup_x \|g(.,x)\|_{Lip},
\]
the analog of estimate \eqref{eq2PDEgenconv} writes down as
\begin{equation}
\label{eq2PDEgenconvpr}
\|\frac{U^{s,t}f-f}{t-s} - L_sf\|
\le (t-s)\|g\|\, \|g\|_{bLip} \|f\|_{C^2}+(t-s)\ka \|f\|_{bLip}.
\end{equation}

\begin{remark}
Evolution $(s,x)\to X_{s,x}(t)$
can be looked at as the deterministic Markov process with the transition operators $U^{s,t}$.
\end{remark}

Finally let us remind the result on the sensitivity of linear evolutions with respect to a parameter.
Let us consider the following linear ODE in $\R^d$:
\begin{equation}
\label{eqlinevol}
\dot y=A(t,x)y, \quad y\in \R^d, \quad t\in [0,T],
\end{equation}
with the matrix $A$ depending smoothly on a parameter $x=(x_1, \cdots , x_n)\in \R^n$.
Here we shall work with $y$ using their sup-norms. All function norms of $A$ will be
meant as $\sup_{t\in[0,T]}$ of the corresponding norms as functions of $x$. For instance,
\[
\|A\|=\sup_{t\in[0,T],x} \sup_i \sum_j |A_{ij}(t,x)|.
\]
For the solution $y(t)=y(t,x)$ to \eqref{eqlinevol} with some initial data $y_0(x)$
that may itself depend on $x$ we have the evident estimate
\[
\|y(t)\|_{sup} \le \exp\{ t\|A\| \} \|y_0\|_{sup}.
\]

\begin{prop}
\label{linODEsens}
Let $A(t,.)\in C^2(\R^n)$. Then $y(t)$ is also twice differentiable with respect to $x$ and
\begin{equation}
\label{eq1linODEsens}
 \left\|\frac{\pa y}{\pa x_i}\right\|_{sup}\le  \exp\{ 2t\|A\| \}
 \left( \left\|\frac{\pa y_0}{\pa x_i}\right\|_{sup} +t \left\|\frac{\pa A}{\pa x_i}\right\| \, \|y_0\|_{sup}\right),
\end{equation}
\[
\left\|\frac{\pa ^2y}{\pa x_i \pa x_j}\right\|_{sup}\le  \exp\{ 3t\|A\| \}
\left( \left\|\frac{\pa^2 y_0}{\pa x_i \pa x_j}\right\|_{sup}
+t \left\|\frac{\pa ^2A}{\pa x_i \pa x_j}\right\| \, \|y_0\|_{sup}\right.
\]
\begin{equation}
\label{eq2linODEsens}
 \left.  +2t  \sup_k\left\|\frac{\pa y_0}{\pa x_k}\right\|_{sup}  \, \sup_k \left\|\frac{\pa A}{\pa x_k}\right\|_{sup}
 + 2t^2 \sup_k \left\|\frac{\pa A}{\pa x_k}\right\|^2 \|y_0\|_{sup} \right).
\end{equation}
\end{prop}

\begin{proof}
All these bounds arise from the solutions of the linear equations satisfied by the derivatives in question:
\[
\frac{d}{dt}\frac{\pa y}{\pa x_i} = A(t,x) \frac{\pa y}{\pa x_i} +\frac{\pa A}{\pa x_i}y,
\]
\[
\frac{d}{dt}\frac{\pa^2 y}{\pa x_i \pa x_j}
= A(t,x) \frac{\pa^2 y}{\pa x_i \pa x_j}
+\frac{\pa A}{\pa x_j}\frac{\pa y}{\pa x_i}
+\frac{\pa A}{\pa x_i}\frac{\pa y}{\pa x_j}
+\frac{\pa ^2A}{\pa x_i \pa x_j}y.
\]
For the justification (existence of the derivatives) we again refer to the textbooks in ODEs mentioned above.
\end{proof}

\section{Mean-field interacting particle systems}
\label{secmeanfieldint}

 Let us turn to the basic setting of mean-field interacting particle systems with a finite number of types.

Let $\{Q (t,x)\}=\{(Q_{ij})(t,x)\}$ be a family of $d\times d$ square $Q$-matrices or Kolmogorov matrices
depending Lipschitz continuously on a vector $x$ from the closed simplex
 \[
 \Si_d=\{x=(x_1,...,x_d)\in \R^d_+ : \sum_{j=1}^d x_j=1\},
 \]
and piecewise continuously on time $t\ge 0$, so that $\|Q\|_{bLip}=\|Q\|+\|Q\|_{Lip}$, where we define
(in accordance with our basic conventions, see \eqref{eqdefLipnorml1a} - \eqref{eqourmatrixnorm} in Preface)
\begin{equation}
\label{eqassumcontQ0}
\|Q\|=\sup_{i,t,x} \sum_j |Q_{ij}(t,x)| <\infty,
\quad
\|Q\|_{Lip}=\sup_{x\neq y} \sup_{i,t} \frac{\sum_j|Q_{ij}(t,x)-Q_{ij}(t,y)|}{|x-y|}.
\end{equation}

In view of the properties of $Q$-matrices this implies
\begin{equation}
\label{eqassumcontQ}
\|Q\|=2\sup_{i,t,x} |Q_{ii}(t,x)| <\infty,
\quad
\|Q\|_{Lip}\le 2\sup_{x\neq y} \sup_{i,t} \frac{\sum_{j\neq i}|Q_{ij}(t,x)-Q_{ij}(t,y)|}{|x-y|}.
\end{equation}

The norms on the derivatives of the families of Kolmogorov matrices $Q$ with respect to $x$ can be, of course,
introduced, in various equivalent ways. For definiteness, we will use the following quantities, most convenient for our purposes:
\begin{equation}
\label{eqassumsmooQ0}
\|Q\|_{C^1}=\|Q\|+\sup_i \sum_j \left|\sup_{k,t,x} \frac{\pa Q_{ij}}{\pa x_k}\right|,
\quad
\|Q\|_{C^2}=\|Q\|_{C^1}+ \sup_i \sum_j \left|\sup_{k,l,t,x} \frac{\pa^2 Q_{ij}}{\pa x_k \pa x_l}\right|.
\end{equation}
As above this implies the estimates in terms of the transitions $Q_{ij}$ with $j\neq i$:
\begin{equation}
\label{eqassumsmooQ}
\|Q\|_{C^1}\le \|Q\|+2\sup_i \sum_{j\neq i} \left|\sup_{k,t,x} \frac{\pa Q_{ij}}{\pa x_k}\right|,
\quad
\|Q\|_{C^2}\le \|Q\|_{C^1}+ 2\sup_i \sum_{j\neq i} \left|\sup_{k,l,t,x} \frac{\pa^2 Q_{ij}}{\pa x_k \pa x_l}\right|.
\end{equation}

In what follows the matrices $Q$ will depend on the additional parameter controlled by the principal,
but for the moment this dependence is not relevant and will be ignored.

Suppose we have a large number of particles distributed arbitrary among the types $\{1,...,d\}$.
More precisely our state space is $\Z^d_+$, the set of sequences of $d$ non-negative integers $n=(n_1,...,n_d)$,
 where each $n_i$ specifies the number of particles in the state $i$. Let $N$ denote the total
 number of particles in state $n$: $N=n_1+...+n_d$. For $i\neq j$ and a state $n$ with $n_i>0$ denote by
 $n^{ij}$ the state obtained from $n$ by removing one particle of type $i$ and adding a particle of type $j$,
 that is $n_i$ and $n_j$ are changed to $n_i-1$ and $n_j +1$ respectively.
The {\it mean-field interacting particle system}\index{mean-field interacting particle system}
 (in continuous time) specified
by the family $\{Q\}$ is defined as the Markov chain on $S$ with the generator
\begin{equation}
\label{eqdefgenmeanfieldchain}
L_t^Nf (n)=\sum_{i,j=1}^d n_i Q_{ij}(t,n/N)[f(n^{ij})-f(n)].
\end{equation}

Probabilistic description of this process is as follows. Starting from any time and current state $n$
one attaches to each particle a $|Q_{ii}|(t,n/N)$-exponential random waiting time (where $i$ is the
type of this particle). If the shortest of the waiting times $\tau$ turns out to be attached to a particle
of type $i$, this particle jumps to a state $j$ according to the distribution $(Q_{ij}/|Q_{ii}|)(t,n/N)$.
Briefly, with this distribution and at rate $|Q_{ii}|(t,n/N)$, any particle of type $i$ can turn (migrate)
to a type $j$. After any such transition the process starts again from the new state $n^{ij}$.
Notice that since the number of particles $N$ is preserved by any jump, this process is in fact a Markov chain
with a finite state space.

Normalizing the states to $x=n/N \in \Si_d\cap \Z^d_+/N$, leads to the generator (also denoted by $L^N_t$,
 with some abuse of notation)
\begin{equation}
\label{eqdefgenmeanfieldchainnorm}
L^N_tf (n/N)=\sum_{i=1}^d \sum_{j=1}^d n_i Q_{ij}(t,n/N)[f(n^{ij}/N)-f(n/N)],
\end{equation}
or equivalently
\begin{equation}
\label{eqdefgenmeanfieldchainnorm1}
L^N_tf (x)=\sum_{i=1}^d \sum_{j=1}^d x_i Q_{ij}(t,x)N[f(x-e_i/N+e_j/N)-f(x)], \quad x\in \Z^d_+/N,
\end{equation}
where $e_1,...,e_d$ denotes the standard basis in $\R^d$. Let us denote by $X^N(t)=X^N_{s,x}(t)$
 the corresponding Markov chain and by $\E=\E_{s,x}$ the expectation with respect to this chain,
 where $x\in \Si_d\cap \Z^d_+/N$ denotes the initial state at time $s<t$.  The transition
operators of this chain will be denoted by $U_N^{s,t}$:
\begin{equation}
\label{eqdefpropagprelim}
U_N^{s,t}f(x)=\E f(X^N_{s,x}(t))=\E_{s,x}f(X^N(t)), \quad s\le t.
\end{equation}
Two versions of the notations in this formula (with $(s,x)$ attached either to $\E$ or to $X^N(t)$)
are both standard in the theory. As any transition operators of a Markov chain, these operators satisfy
the chain rule (or Chapman-Kolmogorov equation)
\[
U_N^{s,r}U_N^{r,t}=U_N^{s,t}, \quad s\le r\le t.
\]
and are said to form a propagator.

In the pressure and resistance setting of \eqref{eqQmapresres} generator
\eqref{eqdefgenmeanfieldchainnorm1} reduces to
\begin{equation}
\label{eqdefgenmeanfieldchainnorm1aa}
L^N_tf (x)=\sum_{i=1}^d \sum_{j=1}^d \ka x_i x_j (R_j(t,x)-R_i(t,x))^+N[f(x-e_i/N+e_j/N)-f(x)].
\end{equation}

Recall that the dual operator $(L_t^N)^*$ to the operator $L_t^N$ is defined from the relation
\[
\sum_{x=n/N\in \Si_d} [(L_t^Nf)(x)g(x)-f(x)((L_t^N)^*g)(x)]=0.
\]
By the shift of the summation index it is straightforward to see that the dual is given by the formula
\begin{equation}
\label{eqgendualfinN}
(L^N_t)^*g (y)=\sum_{i=1}^d \sum_{j\neq i}
[ (y_i+1/N)Q_{ij}(t,y+e_i/N-e_j/N)g(y+e_i/N-e_j/N)-y_iQ_{ij}(t,y) g(y)],
\end{equation}
where, for convenience, it is set that $Q_{ij}(t,x)=0$ for $x\notin \Si_d$.
As is known from the theory of Markov chains, stable distributions $g$
(also referred to as equilibrium probabilities) for the chains, defined by the
generator $L_t^N=L^N$ in the time-homogeneous case, solve the equation $(L^N)^*g=0$,
and for large times $t$ these Markov chains converge to some of these equilibria.

In accordance with \eqref{eqdiscanfconttimeMC}, {\it mean-field interacting systems in discrete time}
\index{mean-field interacting particle system} related to the above discussed mean-field interacting
particle system in continuous time specified by
the family $\{Q\}$ is defined as the Markov chain evolving in discrete time $t=k\tau$, $k\in \N$, with
\begin{equation}
\label{eqintercontdis0}
\tau \le (N\|Q\|)^{-1}
\end{equation}
with the transition probabilities
\begin{equation}
\label{eqintercontdis}
P_{nn^{ij}}^{\tau,t} =P_{nn^{ij},N}^{\tau,t}=\tau n_i Q_{ij}(t,x), \quad i\neq j,
\end{equation}
the probability of remaining in a given state $n$ being
$1-\tau \sum_i n_i |Q_{ii}(t,x)|$.
As it follows from \eqref{eqintercontdis},
\begin{equation}
\label{eqintercontdis1}
\frac{P^{\tau,t}_Nf-f}{\tau}=L_t^Nf
\end{equation}
for all $N$, $t$ and $\tau$. If $Q$ does not depend explicitly on $t$, the dependence on $t$
disappears in \eqref{eqintercontdis}, \eqref{eqintercontdis1}.

As one can expect and as will be shown, the theories of discrete and continuous time
 mean-field interacting systems are very similar. However, the
 discrete-time version is a bit more restrictive for the choice of parameters, as they
 are linked by \eqref{eqintercontdis0}.

Two asymptotic regimes for  mean-field interacting particle system (given by operators $L_t^N$)
are studied in the first part of this book: $N\to \infty$ and $t\to \infty$. In full detail
the limit $\lim_{t\to \infty} \lim_{N\to \infty}$ will be studied. An important question is
whether the order of applying these limits can be interchanged. In other words, what can be said
about $t\to \infty$ limit of the chain with a finite number of particles $N$ (which is described essentially
by the highly multidimensional equation  $(L^N)^*g=0$) if we know the
large-time-behavior of the simpler limiting dynamics $N\to \infty$
(given by \eqref{eqdefgenmeanfieldchainkineq0} below).

\section{Dynamic LLN: smooth coefficients}
\label{LLNsmooth}

As above, the functional norms of functions of two variables $(t,x)$ with and $x\in \Si_d$
will mean the $\sup_t$ of their respective norms as functions of $x$.

Our main interest concerns the asymptotic behavior of these chains as $N \to \infty$.
To this end, let us observe that, for $f\in C^1(\Si_d)$,
\[
\lim_{N \to \infty,\,  n/N \to x}N [f(n^{ij}/N)-f(n/N)]
=\frac{\pa f}{\pa x_j}(x)-\frac{\pa f}{\pa x_i}(x),
\]
 so that
\[
\lim_{N \to \infty,\,  n/N \to x} L^N_tf (n/N)=\La_t f(x),
\]
where
\begin{equation}
\label{eqdefgenmeanfieldchainlim}
\La_t f(x)=\sum_{i=1}^d \sum_{j\neq i} x_i Q_{ij}(t,x)[\frac{\pa f}{\pa x_j}-\frac{\pa f}{\pa x_i}](x)
=\sum_{k=1}^d \sum_{i\neq k} [x_i Q_{ik}(t,x)-x_k Q_{ki}(t,x)]\frac{\pa f}{\pa x_k}(x).
\end{equation}

More precisely, if $f\in C^2(\Si_d)$, then, by the Taylor formula,
\[
 L^N_tf (x)-\La_t f(x)=\frac{1}{2N} \sum_{i=1}^d \sum_{j\neq i} x_i Q_{ij}(t,x)
 [\frac{\pa ^2f}{\pa x_i^2}+\frac{\pa ^2f}{\pa x_j^2}-\frac{\pa ^2f}{\pa x_i \pa x_j}](\theta),
 \]
 for $x=n/N$,  with some $\theta \in \Si_d$, and thus
 \begin{equation}
\label{eqdefgenmeanfieldchainlima}
 \|L^N_tf -\La_t f\|\le \frac{1}{N} \| f^{(2)}\| \, \|Q\|.
\end{equation}

The limiting operator $\La_t f$ is a first-order PDO with characteristics solving the equations
\begin{equation}
\label{eqdefgenmeanfieldchainkineq0}
\dot x_k =\sum_{i \neq k} [x_i Q_{ik}(t,x)-x_k Q_{ki} (t,x)]
=\sum_{i=1}^d x_i Q_{ik}(t,x), \quad k=1,...,d,
\end{equation}
called the {\it kinetic equations}\index{kinetic equation} for the process of interaction described above.
In vector form this system rewrites as
\begin{equation}
\label{eqdefgenmeanfieldchainkineq1}
\dot x=Q^T(t,x) x=x Q(t,x),
\end{equation}
where $Q^T$ is the transpose matrix to $Q$. (In the second notation $xQ(t,x)$ the vector $x$ is understood as
a row vector allowing for the multiplication by $Q$ from the right.)

The corresponding transition operators act on $C(\Si_d)$ as
\begin{equation}
\label{eqdefdetersemigroup}
U^{s,t} f(x)=f(X_{s,x}(t)), \quad s\le t.
\end{equation}

For the case of operator \eqref{eqdefgenmeanfieldchainnorm1aa}
the limiting operator takes the form
\[
\La_t f(x)
=\sum_{i,j=1}^d \ka x_i x_j [R_j(t,x)-R_i(t,x)]^+
\left[\frac{\pa f}{\pa x_j}-\frac{\pa f}{\pa x_j}\right]
\]
\begin{equation}
\label{eqdefgenmeanfieldchainlimaa}
=\sum_{i,j=1}^d \ka x_i x_j [R_j(t,x)-R_i(t,x)]\frac{\pa f}{\pa x_j},
\end{equation}
and the characteristics (or the kinetic equations) become
\begin{equation}
\label{eqdefgenmeanfieldchainlimaaa}
\dot x_j=\sum_{i=1}^d \ka x_i x_j [R_j(t,x)-R_i(t,x)].
\end{equation}

For the sake of clarity we often derive results first for time independent $Q$
and then comment on (usually straightforward) modification required for the general case
(time dependent versions are important only for the second part of the book).

Thus we shall denote
by $X_x(t)$ the characteristics and by $X^N_x(t)$ the Markov chain
 starting at $x$ at time $t=0$. The generators of these processes also
 become time-homogeneous: $L^N_t=L^N$, $\La_t=\La$. The corresponding
transition operators $U^{s,t}$ depend only on the difference $t-s$ and
the operators $U^t=U^{0,t}$, defined as $U^tf(x)=f(X_x(t))$, form a semigroup.
The transition operators for the Markov chain $X_x^N(t)$ are
\[
U^t_Nf(x)=\E f(X^N_x(t))=\E_x f(X^N(t)).
\]

Let us write down explicitly the straightforward estimates of the norms of the r.h.s.
$g(t,x)=xQ(t,x)$ of  \eqref{eqdefgenmeanfieldchainkineq1}, as a function of $x$, in
terms of the norms of $Q$ introduced in \eqref{eqassumcontQ} and \eqref{eqassumsmooQ}:
\begin{equation}
\label{eqQinkin}
\|g\|\le \|Q\|, \quad \|g\|_{Lip}\le \|Q\|_{bLip}, \quad \|g\|_{bLip}\le 2\|Q\|_{bLip},
\quad \|g^{(2)}\|\le 2\|Q\|_{C^2},
\end{equation}
because
\[
\frac{\pa g_k}{\pa x_l}=Q_{lk}+\sum_j x_j\frac{\pa Q_{jk}}{\pa x_l},
\quad
\frac{\pa ^2g_k}{\pa x_l \pa x_m}
=\frac{\pa Q_{lk}}{\pa x_m}+\frac{\pa Q_{mk}}{\pa x_l}
+\sum_j x_j\frac{\pa ^2Q_{jk}}{\pa x_l \pa x_m}.
\]
Under \eqref{eqassumcontQ} system \eqref{eqdefgenmeanfieldchainkineq1} is easily seen
to be well-posed in $\Si_d$ (see Remark below for additional comments and the proof in
 more general infinite-dimensional setting in Section \ref{secprelonl1}), that  is, for
any $x\in \Si_d$ the solution $X_x(t)$ to \eqref{eqdefgenmeanfieldchainkineq1} with $Q$
not depending on $t$ and with the initial condition $x$ at time $t=0$
is well defined and belongs to $\Si_d$ for all times $t>0$ (not necessarily for $t<0$).

\begin{remark}
(i) Unlike general setting \eqref{eqPDS1storlin1}, where the condition $t>0$ was not essential,
here it is essential, because the preservation of the simplex $\Si_d$ holds only in forward time.
It holds, because  $\dot x_k\ge 0$ whenever $x_k=0$ and $x\in \Si_d$, which does not allow a
trajectory to cross the boundary of $\Si_d$. Hence the operators $\Phi^t$ are defined as
operators in $C(\Si_d)$ only for $t>0$. (ii) It is seen from the structure of
\eqref{eqdefgenmeanfieldchainkineq1} that if $x_k\neq 0$, then $(X_x(t))_k\neq 0$ for any
 $t\ge 0$. Hence the boundary of $\Si_d$ is not attainable for this semigroup, but, depending
on $Q$, it can be gluing or not. For instance, if all elements of $Q$ never vanish, then the
points $X_x(t)$ never belong to the boundary of $\Si_d$ for $t>0$, even if the initial point
$x$ does so. In fact, in this case $\dot x_k>0$ whenever $x_k=0$ and $x\in \Si_d$.
\end{remark}

Our first objective now is to show that the Markov chains $X^N_x(t)$ do in fact converge
to the deterministic evolution $X_x(t)$ in the sense that the corresponding transition
operators converge (the so-called weak convergence of Markov processes), the previous arguments
showing only that their generators converge on sufficiently smooth functions. We are also
interested in precise rates of convergence.

The next elementary result concerns the (unrealistic) situation with optimal
 regularity of all objects concerned.

  \begin{theorem}
 \label{thsiplestLLN}
Let all the elements $Q_{ij}(x)$ belong to $C^2(\Si_d)$ and $f\in C^2(\Si_d)$. Then
\begin{equation}
 \label{eq1thsiplestLLN}
 \sup_{x\in \Z^d_+/N} |U^t_Nf(x)-U^tf(x)|
 \le \frac{t\|Q\|}{N} (\|f^{(2)}\|+2t\|f\|_{bLip} \|Q\|_{C^2}) \exp\{3t\|Q\|_{bLip}\},
 \end{equation}
 and, for any $x$ and $n/N$,
\[
 |U^t_Nf(n/N)-U^tf(x)|
 \]
 \begin{equation}
 \label{eq2thsiplestLLN}
 \le \left[\frac{t \|Q\|}{N}  (\|f^{(2)}\|+2t\|f\|_{bLip} \|Q\|_{C^2}) + \|f\|_{bLip}  \|x-n/N\|\right]
  \exp\{3t\|Q\|_{bLip}\}.
 \end{equation}
 Finally, if the initial states $n/N$ converge to a point $x\in \Si_d$, as $N\to \infty$, then
 \begin{equation}
 \label{eq3thsiplestLLN}
 \sup_{0\le t\le T} |U^t_Nf(n/N)-U^tf(x)| \to 0, \quad N\to \infty,
 \end{equation}
 for any $T$ and any $f\in C(\Si_d)$.
 \end{theorem}

\begin{proof}

 To compare the semigroups, we shall use the following standard trick. We write
\begin{equation}
 \label{eqcompprop}
(U^t-U^t_N )f=U^{t-s}_NU^s|_{s=0}^t f
=\int_0^t \frac{d}{ds}U^{t-s}_NU^s f \, ds
=\int_0^t U^{t-s}_N (\La-L^N)U^s f \, ds.
\end{equation}
Let us apply this equation to an $f\in C^2(\Si_d)$.
By \eqref{eq4ODEsmoo} and the second estimate of \eqref{eqQinkin},
\[
 \|(U^t f)^{(2)}\| \le (\|f^{(2)}\|+2t\|f\|_{bLip} \|Q\|_{C^2})\exp\{3t\|Q\|_{bLip}\}.
 \]
 Hence by \eqref{eqdefgenmeanfieldchainlima},
 \[
\|(L^N-\La) U^sf\| \le \frac{\|Q\|}{N}  (\|f^{(2)}\|+2t\|f\|_{bLip} \|Q\|_{C^2}) \exp\{3t\|Q\|_{bLip}\}.
\]
Consequently \eqref{eq1thsiplestLLN} follows from \eqref{eqcompprop} and the contraction property
of $U^t_N$.

Equation \eqref{eq2thsiplestLLN} follows, because, by \eqref{eq2ODEsmoo} and  \eqref{eqQinkin},
\[
|U^tf(x)-U^tf(y)| \le \|U^tf\|_{Lip} \, \|x-y\|\le \exp\{t\|Q\|_{bLip}\} \|f\|_{Lip} \, \|x-y\|.
\]

The last statement is obtained because of the possibility to approximate any function $f\in C(\Si_d)$
by smooth functions.
\end{proof}

By the definition of the propagators  $U^t,U^t_N$, equations
\eqref{eq1thsiplestLLN}-\eqref{eq3thsiplestLLN} can be written
in terms of the averages of the Markov chains $X^N_x$. For instance,
\eqref{eq3thsiplestLLN}  takes the form
\begin{equation}
 \label{eq3athsiplestLLN}
 \sup_{0\le t\le T} |\E f(X^N_{n/N}(t))-f(X_x(t))| \to 0, \quad N\to \infty.
 \end{equation}

For time-dependent $Q$ everything remains the same:

 \begin{theorem}
 \label{thsiplestLLNtd}
Let all the elements $Q_{ij}(t,x)$ belong to $C^2(\Si_d)$ as functions of $x$
and are piecewise continuous as functions of $t$. Let $f\in C^2(\Si_d)$. Then
\begin{equation}
 \label{eq1thsiplestLLNtd}
 \sup_{x\in \Z^d_+/N} |U^{s,t}_Nf(x)-U^{s,t}f(x)|
 \le \frac{(t-s)\|Q\|}{N} (\|f^{(2)}\|+2(t-s)\|f\|_{bLip} \|Q\|_{C^2}) \exp\{3(t-s)\|Q\|_{bLip}\},
 \end{equation}
 and the corresponding analog of \eqref{eq2thsiplestLLN} holds.
 \end{theorem}

The proof is also the same, though instead of \eqref{eqcompprop} one uses its version for propagators:
\begin{equation}
 \label{eqcompproptd}
(U^{s,t}-U^{s,t}_N )f=-U^{s,r}_NU^{r,t}|_{r=s}^t f
=\int_s^t U^{s,r}_N (\La_r-L^N_r)U^{r,t} f \, dr.
\end{equation}

The LLN for discrete-time setting is also analogous. Namely, the following result holds,
where we returned to the time-homogeneous setting.

  \begin{theorem}
 \label{thsiplestLLNdis}
Let all the elements $Q_{ij}(x)$ belong to $C^2(\Si_d)$ and $f\in C^2(\Si_d)$.
Let the discrete-time Markov chain be defined by \eqref{eqintercontdis0}-\eqref{eqintercontdis1}
 Then
\begin{equation}
 \label{eq1thsiplestLLNdis}
  \sup_{n\in \Z^d_+} |(P^{\tau}_N)^kf(n/N)-U^tf(n/N)|
 \le \frac{2t\|Q\|}{N} (\|f^{(2)}\|+2t\|f\|_{bLip} \|Q\|_{C^2}) \exp\{3t\|Q\|_{bLip}\},
 \end{equation}
 for $t=k\tau$, and, for any $x$ and $n/N$,
\[
 |(P^{\tau}_N)^kf(n/N)-U^tf(x)|
 \]
 \begin{equation}
 \label{eq2thsiplestLLNdis}
 \le  \left[\frac{2t \|Q\|}{N}  (\|f^{(2)}\|+2t\|f\|_{bLip} \|Q\|_{C^2}) + \|f\|_{bLip}  \|x-n/N\|\right]
  \exp\{3t\|Q\|_{bLip}\}.
 \end{equation}
 And of course, the analog of convergence \eqref{eq3thsiplestLLN} also holds.
 \end{theorem}

\begin{proof}
We have
\[
((P^{\tau}_N)^k-U^{\tau k})f=(P^{\tau}_N)^{(k-1)}(P^{\tau}_N-U^{\tau})
+(P^{\tau}_N)^{(k-2)}(P^{\tau}_N-U^{\tau})U^{\tau}+\cdots +(P^{\tau}_N-U^{\tau})U^{\tau(k-1)}.
\]
By \eqref{eqdefgenmeanfieldchainlima},  \eqref{eqintercontdis1} and \eqref{eq2PDEgenconv},
\[
\|(P^{\tau}_N-U^{\tau})f\|=\tau \left\|\frac{P^{\tau}_Nf-f}{\tau}-\frac{U^{\tau}f-f}{\tau}\right\|
\]
\[
\le \tau \left(\frac{\|Q\|}{N} \|f^{(2)}\|+\tau \|Q\|^2 \|f^{(2)}\|+\tau  \|Q\|\, \|Q\|_{bLip}\|f\|_{bLip}\right)
\]
\[
\le \frac{\tau}{N} (2\|Q\| \, \|f^{(2)}\|+ \|Q\|_{bLip}\|f\|_{bLip}),
\]
where \eqref{eqintercontdis0} was taken into account.
Consequently, taking into account \eqref{eq4ODEsmoo},
\[
\|(P^{\tau}_N)^kf-U^{\tau k}f\|
\le  \exp\{3t\|Q\|_{bLip}\}
\]
\[
\times \left(\frac{2\tau n}{N} \|Q\| \, \|f^{(2)}\|
+\frac{\tau }{N}\sum_{k=0}^{n-1}(\|Q\|_{bLip}\|f\|_{bLip}+2k\tau \|Q\|_{C^2} \|f\|_{bLip})\right),
\]
yielding \eqref{eq1thsiplestLLNdis}. Estimate \eqref{eq2thsiplestLLNdis} follows like in Theorem \ref{thsiplestLLN}.
\end{proof}

\section{Dynamic LLN: Lipschitz coefficients}
\label{secLLNLip}

Let us move now to a more realistic situation where $Q$ is only assumed to be Lipschitz,
that is, \eqref{eqassumcontQ} holds.

\begin{theorem}
\label{thconv1}
Let the functions $Q(t,x)$ be piecewise continuous in $t$ and belong to $C_{bLip}$ as functions of $x$
with $\|Q\|_{bLip}=\sup_t \|Q(t,.)\|\le \om$ with some $\om$.
Suppose the initial data $x(N)=n/N$ of the Markov chains  $X^N_{s,x(N)}(t)$
converge to a certain $x$ in $\R^d$, as $N\to \infty$. Then these
Markov chains  converge to the deterministic
evolution $X_{s,x}(t)$, in the weak sense:
\begin{equation}
\label{eq1thconv1}
 |\E f (X^N_{s,x(N)}(t))-f(X_{s,x}(t))| \to 0, \,\, \text{as} \,\, N\to \infty,
 \end{equation}
 or in terms of the transition operators
\begin{equation}
\label{eq2thconv1}
 |U^{s,t}_N f(x)-U^{s,t}f(x)| \to 0, \,\, \text{as} \,\, N\to \infty,
 \end{equation}
 for any $f\in C(\Si_d)$, the convergence being uniform in $x$ whenever the convergence $x(N)\to x$ is uniform.

For smooth or Lipschitz $f$, the following rates of convergence are valid:
\[
 |\E f (X^N_{s,x(N)}(t))-f(X_{s,x(N)}(t))|
 \]
 \begin{equation}
\label{eq3thconv1}
 \le C (t-s) \exp\{3(t-s)\|Q\|_{bLip}\}
  \left(\frac{(t-s)^{1/2}}{N^{1/2}}\|Q\|_{bLip}(d+\|Q\|)\|f\|_{bLip}+\frac{\|Q\|}{N}\|f^{(2)}\|\right),
 \end{equation}
\begin{equation}
\label{eq4thconv1}
 |\E f (X^N_{s,x(N)}(t))-f(X_{s,x(N)}(t))|
 \le C \exp\{3(t-s)\|Q\|_{bLip}\} (d+\|Q\|)\|Q\|_{bLip} \frac{(t-s)^{1/2}}{N^{1/2}} \|f\|_{bLip},
 \end{equation}
 \begin{equation}
\label{eq5thconv1}
|f(X_{s,x(N)}(t))-f(X_{s,x}(t))|\le \exp\{(t-s)\|Q\|_{bLip}\} \|f\|_{bLip}\|x(N)-x\|
\end{equation}
with a constant $C$.
\end{theorem}

\begin{remark}
The dependence on $t$ and $d$ is not essential here,
but the latter becomes crucial for dealing with infinite state-spaces, while the former for dealing
with a forward looking principal.
\end{remark}

\begin{remark}
Assuming intermediate regularity of $Q$, that is, assuming $Q\in C^1(\Si_d)$ with all first order derivatives being
H\"older continuous with a fixed index $\al \in (0,1)$ will yield intermediate rates of convergence between $1/N$
and $1/\sqrt N$ above.
\end{remark}

\begin{proof}
To shorten the formulas let us write down a proof for time independent $Q$.

The  Lipshitz continuity \eqref{eq5thconv1} of the solutions is
a consequence of Proposition \ref{ODEsmoo}.
Next, since any function $f\in C(\R^d)$ can be approximated by functions from $C^2(\R^d)$,
the convergence \eqref{eq1thconv1} follows from \eqref{eq3thconv1} and \eqref{eq5thconv1}.
Thus it remains to show \eqref{eq3thconv1} and \eqref{eq4thconv1}.

The main idea is to approximate all Lipschitz continuous functions involved by the smooth ones.
Namely, choosing an arbitrary mollifier $\chi$  (non-negative infinitely smooth even function on $\R$ with
 a compact support and $\int \chi (w) \, dw=1$) and the corresponding mollifier $\phi(y)=\prod \chi (y_j)$
 on $\R^{d-1}$, let us define, for any function $V$ on $\Si_d$,
 its approximation
 \[
 \Phi_{\de}[V](x)=\int_{R^{d-1}}\frac{1}{\de^{d-1}} \phi \left(\frac{y}{\de}\right)V(x-y) \, dy
 =  \int_{\R^{d-1}}\frac{1}{\de^{d-1}} \phi \left(\frac{x-y}{\de}\right)V(y) \, dy.
 \]
 Notice that $\Si_d$ is $(d-1)$-dimensional object, so that any $V$ on it can be considered
 as a function of first $(d-1)$ coordinates of a vector $x\in \Si_d$ (continued to $\R^{d-1}$
 in an arbitrary continuous way).
 It follows that
 \begin{equation}
\label{eq5athconv1}
\|\Phi_{\de}[V]\|_{C^1} =|\Phi_{\de}[V]\|_{bLip}\le  \|V\|_{bLip}
\end{equation}
for any $\de$ and
\[
 |\Phi_{\de}[V](x)-V(x)|\le \int\frac{1}{\de^{d-1}} \phi \left(\frac{y}{\de}\right)|V(x-y)-V(x)| \, dy
 \]
 \begin{equation}
\label{eq6thconv1}
 \le \|V\|_{Lip}  \int_{\R^{d-1}} \frac{1}{\de^{d-1}} \phi \left(\frac{y}{\de}\right) |y|_1\, dy
 \le \de (d-1) \|V\|_{Lip}   \int_{\R} |w| \chi (w) \, dw.
 \end{equation}

\begin{remark}
We care about dimension $d$ in the estimates only for future use (here it is irrelevant). By a different choice
of mollifier $\phi$ one can get rid of $d$ in \eqref{eq6thconv1}, but then it would pop up
in \eqref{eq7thconv1}, which is avoided with our $\phi$.
\end{remark}

Next, the norm $ \|\Phi_{\de}[V]\|_{C^2}$
does not exceed the sum of the norm $ \|\Phi_{\de}[V]\|_{C^1}$
and the supremum of the Lipschitz constants of the functions
\[
\frac{\pa}{\pa x_j} \Phi_{\de}[V](x)=\int\frac{1}{\de^d} \left(\frac{\pa}{\pa x_j} \phi\right) \left(\frac{y}{\de}\right)V(x-y) \, dy.
\]
Hence
\begin{equation}
\label{eq7thconv1}
\|\Phi_{\de}[V]^{(2)}\| \le \|V\|_{bLip} \frac{1}{\de} \int |\chi'(w) | \, dw, \quad
\|\Phi_{\de}[V]\|_{C^2} \le \|V\|_{bLip}\left(1 + \frac{1}{\de} \int |\chi'(w) | \, dw\right).
\end{equation}

 Let $U_{N,\de}^t$ and $U_{\de}^t$ denote the same transition operators as
 above but built with respect to the matrices
\[
\Phi_{\de}[Q](x)=\int\frac{1}{\de^d} \phi \left(\frac{y}{\de}\right)Q(x-y) \, dy
\]
rather than $Q$. Notice that $\Phi_{\de}[Q](x)$ are also $Q$-matrices for any $\de$.

Similarly we denote by $L^{N,\de}$ and $\La^{\de}$ the corresponding generators
and by $X_x^{\de}(t)$ the characteristics  with $\Phi_{\de}[Q]$ used instead of $Q$.

By \eqref{eq6thconv1} and \eqref{eqODErhs},
\[
\|X_t(x)- X_t^{\de}(x)\| \le C \de t d  \|Q\|_{bLip}\exp\{t\|Q\|_{bLip}\}
\]
 and hence
\begin{equation}
\label{th1eq2}
|U^tf(x)-U^t_{\de}f(x)|=|f(X_t(x)-f(X_t^{\de}(x))|\le C \|Q\|_{bLip} \|f\|_{bLip} \de t d \exp\{t\|Q\|_{bLip}\}.
\end{equation}

Moreover, since
\[
\|(L^{N,\de}-L^N)f\| \le \de (d-1)\|f\|_{bLip} \|Q\|_{bLip}  \int_{\R} |w| \chi (w) \, dw,
\]
it follows by \eqref{eqcompprop} applied to propagators $U_N$ and $U_{N,\de}$ that the same estimate
\eqref{th1eq2} holds for the difference $U_{N,\de}^t -U^t_{N}$:
\begin{equation}
\label{th1eq2a}
\|U_{N,\de}^tf -U^t_{N}f\|\le  C \|Q\|_{bLip} \|f\|_{bLip} \de t d  \exp\{t\|Q\|_{bLip}\}.
\end{equation}

By \eqref{eq1thsiplestLLN} and \eqref{eq7thconv1},
\begin{equation}
\label{eqcompsem2rep}
\|U_{N,\de}^tf -U^t_{\de}f\|\le \frac{t \|Q\|}{N} (\|f^{(2)}\|+\frac{Ct}{\de}\|Q\|_{bLip}\|f\|_{bLip}).
\end{equation}

Therefore,
\[
\|U_{N}^tf -U^tf\| \le \|U_{N}^tf -U^t_{N,\de}f\|+\|U_{N,\de}^tf -U^t_{\de}f\|+\|U_{\de}^tf -U^tf\|
\]
\[
\le C t \left(\de d \|Q\|_{bLip} \|f\|_{bLip} +  \frac{\|Q\|}{N} \|f^{(2)}\|
+\frac{t}{N \de}\|Q\|_{bLip}\|Q\|\|f\|_{bLip}\right)\exp\{3t\|Q\|_{bLip}\}
\]
Thus choosing $\de =\sqrt{t/N}$, makes the decay rate of $\de$ and $t/(N\de)$ equal yielding \eqref{eq3thconv1}.

Finally, if $f$ is only Lipschitz, we approximate it by $\tilde f=\Phi_{\tilde \de}[f]$, so that the second derivative of
$\Phi_{\tilde \de}[f]$ is bounded by $\|f\|_{bLip}/\tilde \de$.
By the contraction property of $U^t_N$ and $U^t$,
\[
\|U_N^t (f-\tilde f)\|\le \|f-\tilde f\|\le C d\tilde \de \|f\|_{bLip},
\quad
 \|U^t (f-\tilde f)\|\le \|f-\tilde f\|\le  C d\tilde \de \|f\|_{bLip}.
 \]
Thus the rates of convergence for $f$ become of order
\[
[d\tilde \de + t \de d \|Q\|_{bLip}+\frac{t^2}{N\de}\|Q\|_{bLip}\|Q\|+\frac{t}{N\tilde \de}]\|f\|_{bLip} \exp\{3t\|Q\|_{bLip}\}.
\]
Choosing  $\de =\tilde \de =\sqrt{t/N}$ yields \eqref{eq4thconv1}.
\end{proof}

The corresponding result for the discrete-time setting is again fully analogous.

\begin{theorem}
\label{thconv1dis}
Let the functions $Q(x)$ belong to $C_{bLip}$.
Let the discrete-time Markov chains $X^{N,\tau}_x(t)$ be defined by \eqref{eqintercontdis0}-\eqref{eqintercontdis1}.
Suppose the initial data $x(N)=n/N$ of these Markov chains
converge to a certain $x$ in $\R^d$, as $N\to \infty$ (and thus $\tau \to 0$ by \eqref{eqintercontdis0}). Then these
Markov chains  converge to the deterministic
evolution $X_x(t)$, in the weak sense:
\begin{equation}
\label{eq1thconv1dis}
 |\E f (X^{N,\tau}_{x(N))}(t))-f(X_x(t))| \to 0, \,\, \text{as} \,\, N\to \infty,
 \end{equation}
 for any $f\in C(\Si_d)$, the convergence being uniform in $x$ whenever the convergence $x(N)\to x$ is uniform.

For smooth or Lipschitz $f$, the following rates of convergence are valid:
\[
 |\E f (X^{N,\tau}_{x(N)}(t))-f(X_{x(N)}(t))|
 \]
 \begin{equation}
\label{eq3thconv1dis}
 \le  C t \exp\{3t\|Q\|_{bLip}\}
  \left(\frac{t^{1/2}}{N^{1/2}}(d+\|Q\|_{bLip}^2)\|f\|_{bLip}+\frac{\|Q\|}{N}\|f^{(2)}\|\right),
 \end{equation}
\begin{equation}
\label{eq4thconv1dis}
 |\E f (X^{N,\tau}_{x(N)}(t))-f(X_{x(N)}(t))|
 \le C \exp\{3t\|Q\|_{bLip}\} (d+\|Q\|_{bLip}^2) \frac{t^{1/2}}{N^{1/2}} \|f\|_{bLip},
 \end{equation}
with a constants $C$.
\end{theorem}

\begin{proof} The only modification as compared with the proof of Theorem \ref{thconv1dis}
is the necessity to estimate $(P^{\tau}_{N,\de})^k-(P^{\tau}_N)^k$ instead of
$U^t_{N,\de}-U^t_N$. This is done like in the proof of Theorem \ref{thsiplestLLNdis} using the identity
\[
((P^{\tau}_N)^k-(P^{\tau}_{N,\de})^k)f=(P^{\tau}_N)^{(n-1)}(P^{\tau}_N-P^{\tau}_{N,\de})
\]
\[
+(P^{\tau}_N)^{(n-2)}(P^{\tau}_N-P^{\tau}_{N,\de})P^{\tau}_{N,\de}+\cdots +(P^{\tau}_N-P^{\tau}_{N,\de})(P^{\tau}_{N,\de})^{k-1}.
\]
\end{proof}

\section{Dynamic LLN with major players}
\label{secLLNmajor}

As discussed in Section \ref{secprinplaysim}, we are mostly interested in the presence of  a principal
that may exert pressure on small players on the level described by the parameter $b$ from a bounded convex subset of
a Euclidean space. Mean-field interacting particle system controlled by the principal will be generated by
\eqref{eqdefgenmeanfieldchain} with the coefficients depending on the control parameter $b$ of the principal
(and not depending on time, for simplicity):
\begin{equation}
\label{eqdefgenmeanfieldchainprin}
L^{N,b}f (n)=\sum_{i,j=1}^d n_i Q_{ij}(n/N,b)[f(n^{ij})-f(n)].
\end{equation}

In the simplest setting, which we refer to as 'best response principal',
one can imagine the principal choosing the value of $b^*$  maximizing
 some current profit $B(x,b,N)$ for given $x, N$, that is via \eqref{eqbestrespprin}:
\begin{equation}
\label{eqbestrespprinrep}
b^*(x,N)=argmax B(x,.,N).
\end{equation}
If there exists a limit $b^*(x)=\lim b^*(x,N)$,
 the limiting evolution  \eqref{eqdefgenmeanfieldchainkineq} turns to evolution \eqref{eqkineqprin}:
\begin{equation}
\label{eqkineqprinrep}
\dot x_k =\sum_{i=1}^d x_i Q_{ik}(x, b^*(x)), \quad k=1,...,d,
\end{equation}
or, in particular in pressure and resistance framework, to evolution \eqref{eqpresreslimevol}:
\begin{equation}
\label{eqpresreslimevolrep}
\dot x_j=\sum_i \ka x_i x_j[R_j(x,b^*(x))-R_i(x,b^*(x))], \quad j=1,...,d.
\end{equation}

\begin{remark}
A more strategic thinking principal will be discussed in Chapter \ref{strategicprincipal}.
\end{remark}

The corresponding modification of Theorem  \ref{thconv1} is straightforward yielding the following result
(using time-homogeneous $Q$ for simplicity).
\begin{theorem}
\label{thconv1c}
Assume
\begin{equation}
\label{eq1thconv1c0}
|b^*(x,N)-b^*(x)|\le \ep(N),
\end{equation}
with some $\ep(N)\to 0$, as $N\to \infty$ and some function $b^*(x)$, and let the functions
 $Q(x,b)$ (or, in particular, $R_j(x,b)$ in the pressure and resistance framework), $b^*(x,N)$, $b^*(x)$
 belong to $C_{bLip}$ as a function of their variables with norms uniformly bounded
 by some $\om$. Suppose the initial data $x(N)=n/N$ of the Markov chains  $X^N_{x(N)}(t)$
converge to a certain $x$ in $\R^d$, as $N\to \infty$. Then these Markov chains  converge to
the deterministic evolution $X_x(t)$ solving \eqref{eqkineqprinrep} (or \eqref{eqpresreslimevolrep} respectively):
\begin{equation}
\label{eq1thconv1c}
 |\E f (X^N_{x(N))}(t)-f(X_x(t))| \to 0, \,\, \text{as} \,\, N\to \infty,
 \end{equation}
 for any $f\in C(\Si_d)$, the convergence being uniform in $x$ whenever the convergence $x(N)\to x$ is uniform.
For Lipschitz $f$, estimate \eqref{eq5thconv1} holds and \eqref{eq4thconv1} generalizes to
\begin{equation}
\label{eq4thconv1c}
 |\E f (X^N_{x(N)}(t)-f(X_{x(N)}(t))| \le C(\om,t) \left(\frac{d t^{1/2}}{N^{1/2}}+t\ep(N)\right) \|f\|_{bLip}.
 \end{equation}
\end{theorem}

Similarly, there can be several, say $K$, major players or principals that may exert pressure
on small players. Assume their control parameters are $b_j$ (each chosen from some bounded
closed domain in a Euclidean space) and their profits are given by some Lipschitz continuous functions $B_j(x,b_1, \cdots, b_K,N)$
depending on the distribution $x$ of small players. Recall that a {\it Nash equilibrium}\index{Nash equilibrium} in the game of $K$
players given by these payoffs (for any fixed $x,N$) is a profile of 'no regret' strategies $(b_1^*(x,N), \cdots, b_K^*(x,N))$,
that is, such strategies that unilateral deviation cannot be profitable:
\[
 B_j(x,b_1^*(x,N), \cdots , b_K^*(x,N),N)
 \]
 \begin{equation}
\label{eqdefNash}
 = \max_{b_j} B_j(x,b_1^*(x,N), \cdots, b_{j-1}^*(x,N), b_j, b_{j+1}^*(x,N), \cdots , b_K^*(x,N),N)
 \end{equation}
  for all $j$.

  Assume that for $x\in\Si_d$ there exists a branch of such Nash equilibria
   $(b_1^*(x,N), \cdots, b_K^*(x,N))$ depending continuously on $x$ such that there exists
   a limit
   \begin{equation}
\label{eqNashlimpr}
    (b_1^*(x), \cdots, b_K^*(x))=\lim_{N\to \infty} (b_1^*(x,N), \cdots, b_K^*(x,N)).
    \end{equation}
The dynamics \eqref{eqkineqprinrep} or \eqref{eqpresreslimevolrep}
and Theorem \ref{thconv1c} extend automatically to the case of major players
adhering to these local ($x$ dependent) Nash equilibria. For instance,
\eqref{eqpresreslimevolrep} takes the form
\begin{equation}
\label{eqpresreslimevolrepsev}
\dot x_j=\sum_i \ka x_i x_j[R_j(x,b_1^*(x), \cdots, b_K^*(x))-R_i(x,b_1^*(x), \cdots, b_K^*(x))], \quad j=1,...,d.
\end{equation}

\section{Dynamic LLN with distinguished (tagged) player}
\label{secLLNtag}

Looking for a behavior of some particular distinguished or tagged particle
inside the pool of a large number of indistinguishable ones, is a well known useful tool
in statistical mechanics and experimental biology. Let us extend here the results of Sections
\ref{LLNsmooth} and \ref{secLLNLip} to the case of a tagged agent following
different transition rules than the crowd. This analysis will be used in
Section \ref{secepNashMFG}.

Under the setting of Theorem \ref{thconv1} let us assume that one
 distinguished agent in the group of $N$ players deviates from the general rules moving
 according to the transition $Q$-matrix $Q^{dev}(t,x)$. Then the natural state-space
 for such Markov chain will be $\{1, \cdots, d \}\times \Si_d$, the first coordinate
 $j$ denoting the position of the tagged player. Instead of \eqref{eqdefgenmeanfieldchainnorm1},
 the generator of this Markov chain becomes
 \[
L_t^{N,dev}f(j,x)= \sum_k Q^{dev}_{jk} (t,x) (f(k,x)-f(j,x))
\]
\begin{equation}
\label{eqgenwithtagged}
+\sum_i (x_i-\de^j_i/N) \sum_{k\neq i} Q_{ik}  (t, x)
\left[ f(j, x-e_i/N +e_k/N)-f(j,x)\right].
\end{equation}

Let $U^{s,t}_N$ denote the transition operators of this Markov chains.

For smooth $f$ and as $N\to \infty$ operators \eqref{eqgenwithtagged} converge to the operator
 \[
\La_t^{dev}f(j,x)= \sum_k Q_{jk}^{dev} (t,x) (f(k,x)-f(j,x))
\]
\begin{equation}
\label{eqgenwithtaggedlim}
+\sum_i x_i \sum_{k\neq i} Q_{ik}  (t,x)
\left[ \frac{\pa f}{\pa x_k}-\frac{\pa f}{\pa x_i}\right](j,x),
\end{equation}
with the rates of convergence
\begin{equation}
\label{eqgenwithtaggedlim1}
\|L_t^{N,dev}f-\La_t^{dev}f\|=\sup_{j,x}|(L_t^{N,dev}-\La_t^{dev})f(j,x)|
\le \frac{\|Q\|}{N} (\|f^{(2)}\|+2 \|f\|),
\end{equation}
where
\[
\|f^{(2)}\|=\sup_{j,i,k,x} \left| \frac{\pa f}{\pa x_i \pa x_k}(j,x)\right|.
\]

Operator \eqref{eqgenwithtaggedlim} generates quite specific Markov process on
$\{1, \cdots , d\}\times \Si_d$ (generally speaking, not a chain any more, as
it has a continuous state-space). Its second coordinate $x$ evolves according
to the deterministic kinetic equations $\dot x=Q^T(t,x)x$, independently
on the random first coordinate, which, given $j,x$ at time $s$,  evolves according to the
time-nonhomogeneous Markov chain $J^x_{s,j}(t) \in \{1, \cdots , d\}$ with
the $Q$-matrix
 \[
 Q_{ij}^{dev}(t)= Q_{ij}(t,X_{s,x}(t), u^{dev}_i(t)).
 \]
 Therefore the transition operators $U^{s,t}$ of this process can be written as
 \begin{equation}
 \label{eqgenwithtaggedlim2}
 U^{s,t}f(j,x) =\E f (J^x_{s,j}(t) , X_{s,x}(t)).
 \end{equation}
 For a function $f$ that does not explicitly depends on $x$ this simplifies to
  \begin{equation}
 \label{eqgenwithtaggedlim2a}
 U^{s,t}f(j,x) =\E f (J^x_{s,j}(t)).
 \end{equation}

 \begin{theorem}
 \label{LLNtagged}
 Under setting of Theorem \ref{thconv1} let us assume that one
 distinguished agent in the group of $N$ players deviates from
 the general rules moving according to the transition $Q$-matrix
$Q^{dev}(t,x)$, $t\in [0,T]$, satisfying the same regularity
assumptions as $Q$. Let $f(j,x) =f(j)$ does not explicitly depend
on $x$. Then
\begin{equation}
\label{eq2LLNtagged}
 \|(U^{s,t}-U^{s,t}_N )f\|_{sup}
 \le \frac{(t-s)^{3/2}}{N^{1/2}} C (d, T, \|Q\|_{bLip},\|Q^{dev}\|_{bLip}) \|f\|_{sup}, \quad 0\le s\le t\le T,
 \end{equation}
with a constant $C$ depending on $d, T,\|Q\|_{bLip},\|Q^{dev}\|_{bLip}$.
For smooth $Q$ and $Q^{dev}$,
 \begin{equation}
\label{eq3LLNtagged}
 \|(U^{s,t}-U^{s,t}_N )f\|_{sup}
 \le \frac{(t-s)^2}{N} C (T, \|Q\|_{C^2},\|Q^{dev}\|_{C^2})\|f\|_{sup},
 \end{equation}
with a constant $C$ depending on $T, \|Q\|_{C^2}, \|Q^{dev}\|_{C^2}$.
\end{theorem}

\begin{proof}
Let us start with the case of smooth $Q$ and $Q^{dev}$.
 Using \eqref{eqgenwithtaggedlim1} and the comparison of
 propagators formula \eqref{eqcompproptd} we derive that
\[
\|(U^{s,t}-U^{s,t}_N )f\|_{sup}=sup_{j,x}|(U^{s,t}-U^{s,t}_N )f(j,x)|
\]
\begin{equation}
 \label{eq4LLNtagged}
\le (t-s) \sup_{r\in [s,t]} \|(L_t^{N,dev}-\La_t^{dev})U^{r,t} f \|_{sup}
\le  \frac{t-s}{N} \|Q\| \left(\sup_{r\in [s,t]} \|(U^{r,t} f)^{(2)} \|+2\|f\|\right).
\end{equation}

Thus we need to estimate
\[
\|(U^{r,t} f)^{(2)} \|=\sup_{k,l,j,x} \left|\frac{\pa^2 U^{r,t}f(j,x)}{\pa x_k \pa x_l}\right|,
\]
with $U^{s,t}$ given by \eqref{eqgenwithtaggedlim2a} (and with $f(j,x)=f(j)$ not depending on $x$).

To deal with $U^{s,t}$ it is convenient to fix $s<T$ and $x$ and to consider
the auxiliary propagator $U_{[s,x]}^{r,t}$, $s\le r \le t\le T$, of the Markov
chain $Y_{r,j}(t)$ in $\{1,\cdots, d\}$ with the $Q$-matrix $Q^{dev}(t, X_{s,x}(t))$,
so that
\begin{equation}
 \label{eq5LLNtagged}
 U_{[s,x]}^{r,t}f(j) =\E f (Y_{r,j}(t)),
 \end{equation}
 and
 \begin{equation}
 \label{eq6LLNtagged}
 U^{s,t}f(j) =U_{[s,x]}^{s,t} f (j).
 \end{equation}

Unlike $U^{s,t}$ acting on functions on  $\{1, \cdots , d\}\times \Si_d$,
the propagator $U_[s,x]^{r,t}$ is a propagator of a usual Markov chain and
hence its action satisfies the ODE
 \[
 \frac{d}{dr} U_{[s,x]}^{r,t}f (j)=[Q^{dev}(t, X_{s,x}(t)) U_{[s,x]}^{r,t}f] (j)=
 \sum_k Q^{dev}_{jk}(t, X_{s,x}(t)) (U_{[s,x]}^{r,t}f)(k).
 \]
To find the derivatives with respect to $x$ we can use the standard ODE sensitivity results,
 Proposition \ref{linODEsens} (used in backward time and with the initial condition not
 depending explicitly on the parameter), yielding
 \[
 \sup_{i,k,x}\left|\frac{\pa}{\pa x_i}U_{[s,x]}^{r,t}f (k)\right|
 \le (t-r) \sup_k |f(k)|\sup_{x,i} \left\|\frac{\pa Q^{dev}(t, X_{s,x}(t))}{\pa x_i}\right\| \exp \{2(t-r) \|Q^{dev}\|\},
 \]
 and
\[
\sup_{i,j,k} \left|\frac{\pa ^2}{\pa x_i \pa x_j} U_{[s,x]}^{r,t}f (k)\right|
\le \exp\{ 3(t-r)\|Q^{dev}\| \} \sup_k |f(k)|
\]
\[
\times  \left((t-r) \sup_{i,j} \left\|\frac{\pa ^2 Q^{dev}(t, X_{s,x}(t))}{\pa x_i \pa x_j}\right\|
 + (t-r)^2 \sup_i \left\|\frac{\pa  Q^{dev}(t, X_{s,x}(t))}{\pa x_i}\right\|^2 \right).
\]
Since
\[
\frac{\pa Q^{dev}(t, X_{s,x}(t))}{\pa x_i}
=\frac{\pa Q^{dev}(t, y)}{\pa y}|_{y=X_{s,x}(t)} \frac{\pa X_{s,x}(t))}{\pa x_i}
\]
and similarly for the second derivative, we can use Proposition \ref{ODEsmoo} to estimate the
derivatives of $X_{s,x}(t)$ and thus to obtain
$\|(U^{r,t} f)^{(2)} \|\le C$ with $C$ depending on $T,\|Q\|_{C^2}, \|Q^{dev}\|_{C^2}$
and hence \eqref{eq3LLNtagged} follows by \eqref{eq4LLNtagged}.

When $Q$ and $Q^{dev}$ are only Lipschitz we use the same approximations $\Phi_{\de}[Q]$ and
$\Phi_{\de}[Q^{dev}]$ as in Section \ref{secLLNLip}. And as in Section \ref{secLLNLip} we get the rate
of convergence of order $t d\de+t^2/(\de N)$ yielding \eqref{eq2LLNtagged} by choosing $\de=\sqrt{t/N}$.
\end{proof}

Let us extend the result to functions $f(j,x)$ depending on $x$ explicitly.

\begin{theorem}
\label{LLNtaggedg}
 Under setting of Theorem \ref{thconv1} let us assume that one
 distinguished agent in the group of $N$ players deviates from
 the general rules moving according to the transition $Q$-matrix
$Q^{dev}(t,x)$, $t\in [0,T]$, satisfying the same regularity
assumptions as $Q$. Then
\begin{equation}
\label{eq2LLNtaggedg}
 \|(U^{s,t}-U^{s,t}_N )f\|_{sup}
 \le \frac{(t-s)^{1/2}}{N^{1/2}} C (d, T, \|Q\|_{bLip},\|Q^{dev}\|_{bLip})
 \|f\|_{bLip}, \quad 0\le s\le t\le T,
 \end{equation}
with a constant $C$ depending on $d, T,\|Q\|_{bLip},\|Q^{dev}\|_{bLip}$.
For smooth $Q$ and $Q^{dev}$ and $f$,
 \begin{equation}
\label{eq3LLNtaggedg}
 \|(U^{s,t}-U^{s,t}_N )f\|_{sup}
 \le \frac{(t-s)}{N} C (T, \|Q\|_{C^2},\|Q^{dev}\|_{C^2})\|f\|_{C^2},
 \end{equation}
with a constant $C$ depending on $T, \|Q\|_{C^2}, \|Q^{dev}\|_{C^2}$.
\end{theorem}

\begin{proof} The only difference with the previous case is the
necessity to use Proposition \ref{linODEsens} in full, that is,
with the initial condition also depending on the parameter and,
in case $f$ is not smooth, approximate it in a usual way by smooth functions.
\end{proof}

\section{Rest points of limiting dynamics and Nash equilibria}

Theorem \ref{thconv1c} suggests that eventually the controlled Markov
evolution will settle down near some stable equilibrium points of dynamic
systems \eqref{eqkineqprinrep} or \eqref{eqpresreslimevolrep}.

Let us deal now specifically with system \eqref{eqpresreslimevolrep}.
For a subset $I\subset \{1, \cdots , d\}$, let
\[
\Om_I =\{ x \in \Si_d: x_k =0,  k\in I, \,
\text{and} \, R_j(x,b^*(x))=R_i(x,b^*(x)) \, \text{for} \,  i,j\notin I \}.
\]

\begin{theorem}
\label{th11}
A vector $x$ with non-negative coordinates is a rest point of  \eqref{eqpresreslimevolrep},
that is, it satisfies the system of equations
\begin{equation}
\label{eqdefgenmeanfpoolsing1}
\sum_{i} \ka x_i x_j[R_j(x,b^*(x))-R_i(x,b^*(x))]=0, \quad j=1,...,d,
\end{equation}
if and only if $x\in \Om_I$ for some $I\subset \{1, \cdots , d\}$.
\end{theorem}

\begin{proof}
For any $I$ such that $x_k=0$ for $k\in I$, system  \eqref{eqdefgenmeanfpoolsing1}
reduces to the same system but with coordinates $k\notin I$. Hence it is sufficient to show the result for
the empty $I$. In this situation, system  \eqref{eqdefgenmeanfpoolsing1} reduces to
\begin{equation}
\label{eqdefgenmeanfpoolsing2}
\sum_{i} x_i [R_j(x,b^*(x))-R_i(x,b^*(x))]=0, \quad j=1,...,d.
\end{equation}
Subtracting $j$th and $k$th equations of this system yields
\[
(x_1+\cdots + x_d)  [R_j(x,b^*(x))-R_k(x,b^*(x))]=0,
\]
and thus
\[
R_j(x,b^*(x))=R_k(x,b^*(x)),
\]
as required.
\end{proof}

So far we have deduced the dynamics arising from a certain Markov model of interaction.
As it is known, the internal (not lying on the boundary of the simplex) singular points
of the standard replicator dynamics of evolutionary game theory correspond to the
mixed-strategy Nash equilibria of the initial game with a fixed number of players
(in most examples just two-player game). Therefore, it is natural to ask whether a
similar interpretation can be given to fixed points of Theorem \ref{th11}. Because
of the additional nonlinear mean-field dependence of $R$ on $x$ the interpretation
of $x$ as mixed strategies is not at all clear. However, consider explicitly the
following game $\Ga_N$ of $N+1$ players (that was tacitly borne in mind when
discussing dynamics). When the major player chooses the strategy $b$ and each of
$N$ small players chooses the state $i$, the major player receives the payoff
$B(x,b,N)$ and each player in the state $i$ receives $R_i(x,b)$, $i=1, \cdots, d$
(as above, with $x=n/N$ and $n=(n_1, \cdots , n_d)$ the realized occupation numbers
of all the states). Thus a strategy profile of small players  in this game can be
specified either by a sequence of $N$ numbers (expressing the choice of the state
by each agent), or more succinctly, by the resulting collection of frequencies $x=n/N$.

As usual (see any text in game theory, e.g. \cite{Mazabook} or \cite{PetrZen})
one defines a Nash equilibrium in $\Ga_N$ as a profile of strategies $(x_N,b^*_N)$
such that for any player changing its choice unilaterally would not be beneficial, that is
\[
b^*_N =b^*(x_N,N)= argmax \, B(x_N,b,N)
\]
and for any $i,j\in \{1, \cdots , d\}$
\begin{equation}
\label{eqNashevolprin1}
R_j(x-e_i/N+e_j/N,b_N^*)\le R_i(x,b_N^*).
\end{equation}
A profile is an $\ep$-{\it Nash equilibrium}\index{$\ep$-Nash equilibrium} if these
inequalities hold up to an additive correction term not exceeding $\ep$. It turns out
 that the singular points of \eqref{eqpresreslimevolrep} describe all approximate
Nash equilibria for $\Ga_N$ in the following precise sense:

\begin{theorem}
\label{thfixpointNashl}
Let $R(x,b)$ be Lipschitz continuous in $x$ uniformly in $b$. Let
\[
\hat R=\sup_{i,b} \|R_i(.,b)\|_{Lip},
\]
and, for $I\subset \{1, \cdots , d\}$, let
\[
\hat \Om_I
=\{x\in \Om_I:  R_k(x,b^*(x)) \le R_i(x,b^*(x)) \, \text{for} \, k\in I, i\notin I\}.
\]
Then the following assertions hold.

(i) The limit points
of any sequence $x_N$ such that $(x_N, b^*(x_N, N))$ is a Nash equilibrium for $\Ga_N$
belong to $\hat \Om_I$ for some $I$. In particular, if all $x_N$ are internal points of
$\Si_d$, then any limiting point belongs to $\Om_{\emptyset}$.

(ii) For any $I$ and $x\in \hat \Om_I$ there exists an $2\hat R d/N$-Nash equilibrium
$(x_N,b^*(x_N,N))$ to $\Ga_N$ such  that the difference of any coordinates of $x_N$
and $x$ does not exceed $1/N$ in magnitude.
\end{theorem}

\begin{proof}
(i) Let us consider a sequence of Nash equilibria  $(x_N,b^*(x_N,N))$ such that the
coordinates of all $x_N$ in $I$ vanish. By \eqref{eqNashevolprin1} and the definition of $\hat R$,
\begin{equation}
\label{eqNashevolprin2}
|R_j(x_N,b^*(x_N,N)) - R_i(x_N,b^*(x_N,N))| \le \frac{2}{N}\hat R
\end{equation}
for any $i,j \notin I$ and
\begin{equation}
\label{eqNashevolprin2a}
R_k(x_N,b^*(x_N,N)) \le R_i(x_N,b^*(x_N,N)) + \frac{2}{N}\hat R, \quad k\in I, i\notin I.
\end{equation}
Hence  $x\in \hat \Om_I$ for any limiting point $(x,b)$.

(ii) If $x\in \hat \Om_I$ one can construct its $1/N$-rational approximation,
namely a sequence $x_N\in \Si_d \cap \Z_+^d/N$ such that the difference
of any coordinates of $x_N$ and $x$ does not exceed $1/N$ in magnitude.
For any such $x_N$, the profile $(x_N,b^*(x_N,N))$ is an $2\hat R d/N$-Nash equilibrium for $\Ga_N$.
\end{proof}

Theorem \ref{thfixpointNashl} provides a game-theoretic interpretation of the fixed points of dynamics
\eqref{eqpresreslimevolrep}, which is independent of any myopic hypothesis used to
justify this dynamics. To better illustrate this independence we can easily extend this theorem
to the situations, when the solutions to the kinetic equations are not well defined, namely to the case
 of only continuous $R_j(x)$ (neither Lipschitz not even H\"older). The result can be best expressed
 in terms of the modulus of continuity of the function $R_j$ that we define in the following way:
 $w_j(h;b)=\sup\{|R_j(x,b)-R_j(y,b)|\}$, where $\sup$ is over the pairs of $x,y$ that differ only in
 one coordinate and by amount not exceeding $h$. Straightforward extension of the proof of
 Theorem \ref{thfixpointNashl} yields the following result.

\begin{theorem}
\label{thfixpointNashlcon}
Let $R(x,b)$ be continuous in $x$ uniformly in $b$, so that $\hat R(h)\to 0$ as $h\to 0$, where
\[
\hat R(h)=\sup_{i,b} w_i(h;b).
\]
Then (i) the limit points
of any sequence $x_N$ such that $(x_N, b_N^*(x_N, N))$ is a Nash equilibrium for $\Ga_N$
belong to $\hat \Om_I$ for some $I$; and
(ii) for any $I$ and $x\in \hat \Om_I$ there exists an $2d\hat R(1/N)$-Nash equilibrium
$(x_N,b_N^*(x_N,N))$ to $\Ga_N$ such  that the difference of any coordinates of $x_N$
and $x$ does not exceed $1/N$ in magnitude.
\end{theorem}

Theorem \ref{thfixpointNashl} extends also automatically to the case of several major players and dynamics
\eqref{eqpresreslimevolrepsev}. For example, let us discuss the case of Lipschitz continuous payoffs.
Namely, let us consider the game $\Ga_{N,K}$ of $N+K$ players,
where the major players choose the strategies $b_1, \cdots , b_K$ and
each of $N$ small players chooses the state $i$. The payoffs of the major players are
$B(x,b,N)$ and each player in the state $i$ receives $R_i(x,b)$, $i=1, \cdots, d$.
Assume the existence of a continuous branch of Nash equilibria
$(b_1^*(x,N), \cdots, b_K^*(x,N))$ having limit \eqref{eqNashlimpr}.

\begin{theorem}
\label{thfixpointNashl1}
Let $R(x,b_1, \cdots, b_K)$ be Lipschitz continuous in $x$ uniformly in $b_1, \cdots, b_K$ and
\[
\hat R=\sup_{i,b_1, \cdots, b_K} \|R_i(.,b_1, \cdots, b_K)\|_{Lip}.
\]
For $I\subset \{1, \cdots , d\}$, let us define $\Om_I$ and $\hat \Om_I$ as above
but with  $(b_1^*(x), \cdots, b_K^*(x))$ instead of just $b^*(x)$.
Then

(i) The limit points
of any sequence $x_N$ such that $(x_N, b_1^*(x_N,N), \cdots, b_K^*(x_N,N))$
is a Nash equilibrium for $\Ga_{N,K}$ belong to $\hat \Om_I$ for some $I$.

(ii) For any $I$ and $x\in \hat \Om_I$ there exists an $2\hat R d/N$-Nash equilibrium
\newline $(x_N,b_1^*(x,N), \cdots, b_K^*(x,N))$ for $\Ga_{N,K}$
such  that the difference of any coordinates of $x_N$ and $x$ does not exceed $1/N$ in magnitude.
\end{theorem}

Of course, the set of 'almost equilibria' $\Om$ may be empty or contain many points.
Thus one can naturally pose here the analog of the question, which is well discussed
in the literature on the standard evolutionary dynamics (see \cite{BinSam} and references therein),
namely which equilibria can be chosen  in the long run (the analogs of stochastically
stable equilibria in the sense of \cite{FoYou}) if
small mutations are included in the evolution of the Markov approximation.

A distinguished class of rest points of dynamics \eqref{eqpresreslimevolrep} (and the related Nash equilibria)
represent {\it stable rest points} (or stable sets of rest points), characterized by the property that all points
 starting motion in a neighborhood of such point (or set of points) remain in this neighborhood forever.
The analysis of stable rest points in the key class of examples will be preformed in Section \ref{stableanalin}.
It will be shown there that such stability implies certain 'long term stability', usually for times $t$ of order $N$,
for the approximating dynamics of $N$ players, see Theorem \ref{thstableequiMar}. This theorem is formulated for a class
of examples, but is very general in its nature. In many cases one gets much better results showing that the
equilibrium probabilities for the chains of $N$ players have supports in the $1/N$-neighborhood of the set of
stable rest points  of dynamics \eqref{eqpresreslimevolrep}.

\section{Main class of examples: inspection, corruption, security}

All models of inspection, corruption, counterterrorist measures and cyber-security
of Chapter \ref{chapintkolMaCor} fall in the general class of pressure and resistance games,
where payoffs to small players have the following structure:
\begin{equation}
\label{eqmainexpresres}
R_j(x,b)= w_j-p(b)f_j,
\end{equation}
where $w_j>0$ are profits (or winnings) resulting in applying $j$th strategy and $f_j>0$
are fines that have to be paid with some probabilities $p(b)$ depending on, and increasing
with, the efforts  of the principal measured by the budget parameter $b$.
By ordering the strategies of small players one can assume that
\begin{equation}
\label{eqmainexpresres0}
w_1<w_2<\cdots < w_d, \quad f_1<f_2<\cdots < f_d,
\end{equation}
the latter inequalities expressing the natural assumption that fines (risks) increase when
 one attempts to get higher profits. With the principal having essentially opposite interests
to the interest of the pool of small players, the payoff of the principal can be often expressed
 as the weighted average loss of small players with the budget used subtracted:
\begin{equation}
\label{eqmainexpresres1}
B(x,b)= -b + \ka \sum_j x_j (p(b)f_j-w_j)=-b +\ka (p(b)\bar f -\bar w),
\end{equation}
where $\ka$ is a constant and bars denote the averaging with respect to the distribution $x$.

\begin{remark} Of course, one can think of more general situations with $R_j= w_j-p_j(b)f_j$,
with probabilities depending on $j$ or with $R_j= w_j-p(b_j)f_j$ with different budgets used for
dealing with different strategies.
\end{remark}

Since $p$ is increasing, $p'(b)>0$. Moreover, interpreting $p$ as the probability of finding
certain hidden behavior of small players, it is natural to assume that search can never be perfect,
that is $p(b)\neq 1$, and thus $p:[0,\infty)\to [0,1)$. Assuming that $p$ is fast increasing for
small $b$ and this growth decreases with the increase of $b$ (the analog assumption to the law of
diminishing return in economics) leads to the conditions that $p'(0)=\infty$ and $p(b)$ is a concave
function. Summarising, the property of $p:(0,\infty)\to (0,1)$ that can be naturally assumed for a
rough qualitative analysis are as follows: $p$ is a smooth monotone bijection such
\begin{equation}
\label{eqmainexpresres2}
p''(b)<0 \,\, \text{for all} \,\,  b, \quad p'(0)=+\infty, \,\, p'(\infty)=0.
\end{equation}
Under these assumptions the function $B(x,b)$ from \eqref{eqmainexpresres1} is concave as a function of $b$,
its derivatives monotonically decreases from $\infty$ to $-1$ and hence there exists a unique point of
maximum $b^*=b^*(x)$ such that $-1+\ka \bar f p'(b^*)=0$, so that
\begin{equation}
\label{eqmainexpresres3}
b^*(x)=(p')^{-1}(1/\ka \bar f(x)),
\end{equation}
where $(p')^{-1}: (0,\infty)\to (0,\infty)$ is the (monotonically decreasing) inverse function to $p'$.
Moreover, $b^*(x)$ depends on $x$ only via the average $\bar f=\sum_j x_jf_j$, so that
\begin{equation}
\label{eqmainexpresres4}
b^*(x)=\hat b(\bar f(x)), \quad \hat b(\bar f)=(p')^{-1}(1/\ka \bar f(x)).
 \end{equation}
Clearly the range of possible $b^*(x)=\hat b(\bar f)$ is the interval
\begin{equation}
\label{eqmainexpresres5}
b^*\in [(p')^{-1}(1/\ka f_1),(p')^{-1}(1/\ka f_d)].
\end{equation}

Thus the controlled dynamics \eqref{eqpresreslimevolrep} takes the form
\begin{equation}
\label{eqmainexpresres6}
\dot x_j=x_j [w_j-p(b^*(x))f_j-(\bar w-p(b^*(x))\bar f)], \quad j=1,...,d.
\end{equation}

The rest point of this dynamics can be explicitly calculated and their stability analysed
for general classes of the dependence of fines $f_j$ on the profits $w_j$.
The simplest possibility is the {\it proportional fines} $f_j=\la w_j$ with a constant $\la>0$.

\begin{remark}
Such fines are used in the Russian legislation for punishment arising from tax evasion,
that is, in the context of inspection games.
\end{remark}

The next result classifying the rest points for proportional fines is straightforward.
\begin{prop}
\label{proppropfinerest}
For the case of proportional fines dynamics \eqref{eqmainexpresres6}
turns to the dynamics
\begin{equation}
\label{eqmainexpresres7}
\dot x_j=x_j (w_j-\bar w)(1- p(b^*(x))\la ), \quad j=1,...,d.
\end{equation}
The rest points of this dynamics
are the vertices of the simplex $\Si_d$ and, if
\begin{equation}
\label{eqmainexpresres8}
\frac{1}{\la} \in p [(p')^{-1}(1/\ka f_1),(p')^{-1}(1/\ka f_d)],
\end{equation}
then all points of the hyperplane defined uniquely by the equation $\bar f(x)=f^*$ with
\begin{equation}
\label{eqmainexpresres9}
p(\hat b(f^*))=1/\la,
\end{equation}
also represent rest points.
\end{prop}

\section{Optimal allocation}
\label{secoptimalocandgroup}

So far our small players were indistinguishable. However, in many cases
the small players can belong to different types. These can be inspectees
with various income brackets, the levels of danger or overflow of particular
traffic path, or the classes of computers susceptible to infection.
In this situation the problem for the principal becomes a policy problem,
that is, how to allocate efficiently her limited resources. Our theory
extends to a setting with various types more-or-less straightforwardly.
We shall touch it briefly.

The models of investment policies of Section \ref{projectselectdef} also
 belong to this class of problems.

Let our players, apart from being distinguished by states $i\in \{1,\cdots, d\}$,
can be also classified by their types or classes $\al \in \{1, \cdots, \AC \}$.
The state space of the group becomes $\Z^d_+\times \Z^{\AC}_+$, the set of matrices
$n=(n_{i\al})$, where $n_{i\al}$ is the number of players of type $\al$ in the state
$i$ (for simplicity of notation we identify the state spaces of each type, which is
not at all necessary). One can imagine several scenarios of communications between
classes, two extreme cases being as follows:

(C1) No-communication: the players of different classes can neither communicate
nor observe the distribution of states in other classes, so that the interaction
between types arises exclusively through the principal;

(C2) Full communication: the players can change both their types and states
via pairwise exchange of information, and can observe the total distribution
of types and states.

There are lots of intermediate cases, say, when types form a graph (or a network)
with edges specifying the possible channels of information. Let us deal here only
with cases (C1) and (C2). Starting with (C1), let $N_{\al}$ denote the number of
players in class $\al$ and $n_{\al}$ the vector $\{n_{i\al}\}, i=1, \cdots ,d$.
Let $x_{\al}=n_{\al}/N_{\al}$,
\[
x=(x_{i\al})= (n_{i\al}/N_{\al})\in (\Si_d) ^{\AC},
\]
and $b=(b_1, \cdots, b_{\AC})$ be the vector of the allocation of resources of the
principal, which may depend on $x$. Assuming that the principal uses the optimal policy
\begin{equation}
\label{eqdefgenmeanfcl0}
b^*(x)=argmax \, B(x,b)
\end{equation}
arising from some concave (in the second variable) payoff function $B$ on
$(\Si_d)^{\AC}\times \R^{\AC}$, the generator of the controlled Markov process
becomes

\[
L_{b*,N}f (x)=\sum_{\al=1}^{\AC}N_{\al}\ka_{\al} \sum_{i,j} x_{i\al} x_{j\al}
\]
\begin{equation}
\label{eqdefgenmeanfcl1}
\times
[R_j^{\al}(x_{\al},b^*(x))-R_i^{\al}(x_{\al},b^*(x))]^+
[f(x-e_i^{\al}/N_{\al}+e_j^{\al}/N_{\al})-f(x)],
\end{equation}
where $e_i^{\al}$ is now the standard basis in $\R^d\times \R^{\AC}$.
Passing to the limit as $N\to \infty$ under the assumption that
\[
\lim_{N\to \infty} N_{\al}/N=\om_{\al}
\]
with some constants $\om_{\al}$ we obtain a generalization of
\eqref{eqpresreslimevolrep} in the form
\begin{equation}
\label{eqdefgenmeanfcl2}
\dot x_{j\al}=\ka_{\al}\om_{\al} \sum_i x_{i\al} x_{j\al}
[R^{\al}_j(x_{\al},b^*(x))-R^{\al}_i(x_{\al},b^*(x))],
\end{equation}
for $j=1,...,d$ and $\al=1, \cdots , \AC$,
coupled with \eqref{eqdefgenmeanfcl0}.

In case (C2), $x=(x_{i\al})\in \Si_{d\al}$, the generator becomes
\[
L_{b*,N}f (x)=\sum_{\al, \be=1}^{\AC}N\ka \sum_{i,j} x_{i\be} x_{j\al}
\]
\begin{equation}
\label{eqdefgenmeanfcl1a}
\times
[R_j^{\al}(x,b^*(x))-R_i^{\be}(x,b^*(x))]^+[f(x-e_i^{\be}/N_{\al}+e_j^{\al}/N_{\al})-f(x)],
\end{equation}
and the limiting system of differential equations
\begin{equation}
\label{eqdefgenmeanfcl3}
\dot x_{j\al}=\ka \sum_{i, \be} x_{i\be} x_{j\al}[R^{\al}_j(x,b^*(x))-R^{\be}_i(x,b^*(x))].
\end{equation}

\section{Stability of rest points and its consequences}
\label{stableanalin}

As was noted above, an important class of rest points represent stable points.

Let us show that the hyperplane of fixed points \eqref{eqmainexpresres9} is
{\it stable} under dynamics  \eqref{eqmainexpresres7}, that is, if the dynamics
starts in a sufficiently small neighborhood of this hyperplane, then it would remain
there forever. We shall do it by the method of Lyapunov functions showing that the function
\[
V(x)=(1-\la p(b^*(x)))^2
\]
is a {\it Lyapunov function}\index{Lyapunov function} for dynamics \eqref{eqmainexpresres7} meaning
that $V(x)=0$ only on plane \eqref{eqmainexpresres9} and $\dot V(x)\le 0$
everywhere whenever $x$ evolves according to \eqref{eqmainexpresres7}.

\begin{prop}
\label{propstableequi}
If $x$ evolves according to \eqref{eqmainexpresres7}, then $\dot V(x)< 0$
for all $x$ that are not vertices of $\Si_d$ and do not belong to plane \eqref{eqmainexpresres9}
implying that the neighborhoods $\{x:V(x)< v\}$ are invariant under \eqref{eqmainexpresres7} for any $v$.
\end{prop}

\begin{proof}
We have
\[
\frac{d}{dt} V(x)= \frac{\pa V}{\pa x} \dot x
=-2\la (1-\la p(b^*(x)))^2 p'(\hat b(\bar f))\hat b'(\bar f)\sum_j  x_j f_j (w_j-\bar w)
\]
\[
=-2\la^2 (1-\la p(b^*(x)))^2 p'(\hat b(\bar f))\hat b'(\bar f)(\sum_j x_jw_j^2-\bar w^2)
\]
\[
=-2\la^2 V(x) p'(\hat b(\bar f))\hat b'(\bar f)(\sum_j x_jw_j^2-\bar w^2)<0,
\]
since $\sum_j x_jw_j^2-\bar w^2$ is the variance of the random variable $w$
taking values $w_j$ with probabilities $x_j$, which is always non-negative.
\end{proof}

The statement of this proposition can be essentially improved yielding full portrait
of dynamics \eqref{eqmainexpresres7}.

\begin{theorem}
\label{thstableequi}
(i) If the plane of rest points \eqref{eqmainexpresres9} exists, that is condition
\eqref{eqmainexpresres8} holds, then this plane is the global attractor for dynamics \eqref{eqmainexpresres7}
outside vertices of $\Si_d$: for any $x$, which is not a vertex of $\Si_d$, $X_x(t)$ approaches this plane as $t\to \infty$.
Moreover, the simplex $\Si_d$ is decomposed in two invariant sets with $1-p(b^*(x))\la>0$ and $1-p(b^*(x))\la<0$.

(ii) If
\[
1/\la \ge p[(p')^{-1}(1/\ka f_d)],
\]
then the pure strategy of maximal activity $j=d$ is the global attractor: for any $x$, which is not a vertex of $\Si_d$,
$X_x(t)$ approaches the point $(0,\cdots, 0,1)$, as $t\to \infty$.

(iii) If
\[
1/\la \le p[(p')^{-1}(1/\ka f_1)],
\]
then the pure strategy of minimal activity $j=1$ is the global attractor: for any $x$, which is not a vertex of $\Si_d$,
$X_x(t)$ approaches the point $(1,0,\cdots, 0)$, as $t\to \infty$.
\end{theorem}

\begin{proof}
(i) Since for any $x$, which is not a vertex of $\Si_d$, $X_x(t)$ can also never become a vertex,
$(d/dt) V(X_x(t))<0$ on the whole trajectory. Hence there exists a limit of $V(X_x(t))$,
as $t\to \infty$. A straightforward argument by contradiction
shows that this limit should be zero.

(ii) In this case $1-p(b^*(x))\la>0$ in the whole $\Si_d$
and the point of minimum of $V$ in $\Si_d$ is $(0,\cdots, 0,1)$.

(iii) In this case $1-p(b^*(x))\la<0$ in the whole $\Si_d$
and the point of minimum of $V$ in $\Si_d$ is $(1,0,\cdots, 0)$.
\end{proof}

Let us describe a consequence of stability of hyperplane \eqref{eqmainexpresres9}
to the behavior of the dynamics $U_N^t$ of a finite number of players.

\begin{theorem}
\label{thstableequiMar}
If $V(x)\le v$, with high probability the points $X^N_x(t)$ stay in the neighborhood $\{V(y)\le rv\}$
for large $r$ and times $t$ much less than $N$. Namely,
\begin{equation}
\label{eqthstableequiMar}
\P (V(X^N_x(t))>rv)\le \frac{1}{r}\left(1 +\frac{\ep}{v} \| V^{(2)}\| \, \|Q\|\right)
\end{equation}
for $t\le \ep N$.
\end{theorem}

\begin{proof}
By \eqref{eqdefgenmeanfieldchainlima},
\[
 \|L^N_tV -\La_t V\|\le \frac{1}{N} \| V^{(2)}\| \, \|Q\|.
\]
Hence, by Proposition \ref{propstableequi},
\[
L^N_tV(x)\le \frac{1}{N} \| V^{(2)}\| \, \|Q\|
\]
for all $x$. Consequently,
\[
\frac{d}{dt}U_N^tV(x)=L_N^tV(x)\le \frac{1}{N} \| V^{(2)}\| \, \|Q\|
\]
for all $x$ and therefore
\[
U_N^tV(x)\le V(x)+\frac{t}{N} \| V^{(2)}\| \, \|Q\|.
\]
Consequently, if $t\le \ep N$ for some small $\ep$ and $V(x)\le v$, then
\[
\E V(X^N_x(t))\le v+\ep \| V^{(2)}\| \, \|Q\|,
\]
implying \eqref{eqthstableequiMar} by Markov's inequality.
\end{proof}

This theorem is very rough and has almost straightforward extension to general stable rest points of dynamics
\eqref{eqkineqprinrep}. For the particular case of  dynamics \eqref{eqmainexpresres7} the full description
of large time behavior of the Markov chains $X^N_x(t)$ is available.

\begin{theorem}
\label{thstableequiMarprfi}
(i) Under condition of Theorem \ref{thstableequi} (ii) or (iii) the points $(0,\cdots, 0,1)$ or $(1,0,\cdots, 0)$, respectively,
are the absorbing points for Markov chains $X^N_x(t)$ for any $N$: starting from any profile of strategies, the trajectory
almost surely reaches this point in finite time and remains there forever.
(ii) Under condition of Theorem \ref{thstableequi} (i), the $2/N$-neighborhood of plane \eqref{eqmainexpresres9}
is absorbing for the Markov chains $X^N_x(t)$ for any $N$: starting from any profile of strategies, the trajectory
almost surely enters this neighborhood in a finite time and remains there forever. It follows, in particular, that
these Markov chains have stationary distributions supported on the states belonging to the $2/N$-neighborhoods
of plane \eqref{eqmainexpresres9}.
\end{theorem}

\begin{proof}
 The generator \eqref{eqdefgenmeanfieldchainnorm1} of the Markov chain $X^N_x(t)$ for
  transitions \eqref{eqQmapresres}, \eqref{eqmainexpresres} is

\begin{equation}
\label{eqgenprfiMa}
L^N_tf (x)=N \sum_{i=1}^d \sum_{j=1}^d x_i x_j[w_j-w_i -p(b^*(x))(f_j-f_i)]^+[f(x-e_i/N+e_j/N)-f(x)], \quad x\in \Z^d_+/N,
\end{equation}
which, for the case of proportional fines, takes the form
\begin{equation}
\label{eqgenprfiMa1}
L^N_tf (x)=N \sum_{i=1}^d \sum_{j=1}^d x_i x_j[(w_j-w_i)(1 -p(b^*(x))\la)]^+[f(x-e_i/N+e_j/N)-f(x)].
\end{equation}

Therefore, when $1 -p(b^*(x))\la<0$, only transitions decreasing activity, $i\to j<i$ are allowed.
Vice versa,  when $1 -p(b^*(x))\la>0$, only transitions increasing activity, $i\to j>i$ are allowed.
This implies statement (i), as in this case the sign of $1 -p(b^*(x))\la$ is constant on the whole $\Si_d$.

In case (ii) the transitions from points on the 'upper set' and 'lower set',
\[
\Si_d^+=\{x: 1 -p(b^*(x))\la<0\},\quad \Si_d^-=\{x: 1 -p(b^*(x))\la>0\},
\]
decrease or increase the activity, respectively. As long as these points are outside $2/N$-neighborhood
of plane \eqref{eqmainexpresres9} any single transition cannot jump from the upper to the lower set or vise versa,
but it decreases the distance of the states to plane \eqref{eqmainexpresres9}. Thus the points continue to move
until they enter this neighborhood.
\end{proof}

The analysis of system \eqref{eqmainexpresres6} for non-proportional fines is given in \cite{KatKolYan}.
In cases of convex and concave dependence of fines on the profits, all fixed points turn out
to be isolated and supported by only two pure strategies.

These results lead to clear practical conclusions
showing how the manipulations with the structure of fines (often imposed by the principal)
can yield the desired distribution of the rest points of the related dynamics and hence the equilibria
of the corresponding games. The choice of the fine structure can be looked at as a kind of {\it mechanism design}
of the principal. For instance, in the case of proportional fines considered above, an appropriate choice of
the coefficient of proportionality can lead to the whole pool of small players to act on the maximal or on the
minimal level of activity.

\section{Stability for two-state models}
\label{sectwostatemod}

Of course, one expects the simplest evolutions to occur for two-state models.
In fact, under reasonable assumptions the behavior of both limiting and approximating dynamics
can be fully sorted out in this case.

Assume the state space of small players consists of only two strategies.
Then the rates are specified
just by two numbers $Q_{12}$ and $Q_{21}$, describing transitions from the first state to the second and vice versa,
and the overall distribution by one number, $x=n/N$, the fraction of players using the first strategy.
When dealing with time independent transitions, the limiting dynamics \eqref{eqdefgenmeanfieldchainkineq0}
reduces to the following single equation:

\begin{equation}
\label{eqtwostatekin}
\dot x =(1-x)Q_{21}(x)-x Q_{12} (x).
\end{equation}

Assume now that the interval $[0,1]$ is decomposed into a finite number of regions with the directions
of preferred transitions alternating between them. Namely, let
\[
0=a_0 <a_1 < \cdots < a_{k-1}<a_k=1, \quad I_k=(a_{k-1}, a_k),
\]
and let $Q_{21}>0$, $Q_{12}=0$ in $I_k$ with even $k$, and $Q_{21}=0$, $Q_{12}>0$ in $I_k$ with odd $k$.
Assume also that $Q_{12}(x)$ and $Q_{21}(x)$ are Lipschitz continuous and vanish on the boundary points $x=0$ and $x=1$.

The minority game of Section \ref{minorgamedef} fits to these assumptions with just two intervals:
 $I_0=(0,1/2)$, $I_1=(1/2, 1)$.

The fixed points of the dynamics are the points of 'changing interests': $x=a_j$, $j=0, \cdots, k$.
It is clear that the points $a_k$ with even (respectively odd) $k$ are stable (respectively unstable) fixed points
of dynamics \eqref{eqtwostatekin}. Therefore, if dynamics starts at $x\in I_k$ with even (resp. odd) $k$, the solution $X_x(t)$
of \eqref{eqtwostatekin} will tend to the left point $a_k$ (resp. right point $a_{k+1}$) of $I_k$, as $t\to \infty$.

What is interesting is that in this case the same behavior can be seen to hold for the approximating systems of
$N$ agents evolving according to the corresponding Markov $X^N_x(t)$ chain with the generator

\begin{equation}
\label{eqdefgenmeanfieldchainnorm1ag}
L^N_tf (x)=x Q_{12}(x)N[f(x-1/N)-f(x)]+(1-x) Q_{21}(x)N[f(x+1/N)-f(x)].
\end{equation}

\begin{prop}
\label{proptwostatemodMark}
Under the above assumptions and for any sufficiently large $N$, the Markov chain $X^N_x(t)$ starting
 at a point $x=n/N \in (a_k,a_{k+1})$ moves to the left (resp. right) if $k$ is even (resp. odd), reaches
an $1/N$- neighborhood of $a_k$ (resp. $a_{k+1}$) in finite time and remains in this neighborhood for ever after.
\end{prop}

\begin{proof} This is straightforward.
\end{proof}

\section{Discontinuous transition rates}
\label{secdiscontrates}

An interesting topic to mention is the study of nonlinear Markov chains
with discontinuous rates $Q_{ij}(x)$.  Specifically important is the case
when the whole simplex $\Si_d$ is decomposed into a finite union of closed
domains, $\Si_d=\cup_{j=1}^k \bar X_j$ such that the interiors $X_j$ of
$\bar X_j$ do not intersect and $Q_{ij}(x)$ are Lipcshitz continuous in
each $X_j$ with the continuous extension to $\bar X_j$. Such system would
naturally arise in MFG forward-backward systems (see Chapter
\ref{chapthreestate}) with a finite set of controls.

We shall not develop here the full theory of discontinuous rates, but we shall set
 the ground for it by looking at the specific case of two-state models of the previous
section. Namely, using the notations of that section, let us assume now that $Q_{21}$
and $Q_{12}$ are not continuous through the switching points $a_k$. As the simplest
assumption let us choose these rates to be constant:
\begin{equation}
\label{eqtwostatedisc}
Q_{21}(x)=
\left\{
\begin{aligned}
& 1, \quad x\in I_k, \, k \, \text{even} \\
& 0, \quad x\in I_k, \, k \, \text{odd},
\end{aligned}
\right.
\quad \quad
Q_{12}(x)=
\left\{
\begin{aligned}
& 1, \quad x\in I_k, \, k \, \text{odd} \\
& 0, \quad x\in I_k, \, k \, \text{even}.
\end{aligned}
\right.
\end{equation}

The kinetic equations \eqref{eqtwostatekin} have to be understood now in the sense
of Filippov, see \cite{Filip60}, \cite{Filip88}, which means in this case that instead
of the equation we look at the kinetic differential inclusion

\begin{equation}
\label{eqtwostatekininc}
\dot x \in F(x), \quad F(x)=\left\{
\begin{aligned}
& (1-x)Q_{21}(x)-x Q_{12} (x), \quad x \notin\{a_0, \cdots, a_k\} \\
& [-a_j, 1-a_j], \quad x=a_j,
\end{aligned}
\right.
\end{equation}
so that away from points $a_j$ the inclusion coincides with the equation \eqref{eqtwostatekin}.
The solution of this inclusion is not unique only if started from the points $a_k$
with odd $k$ (i.e. from the unstable equilibria). Discarding these initial points we see that
starting from $x_0=a_k$ with even $k$ the unique solution to \eqref{eqtwostatekininc} remains in $a_k$.
Starting from $x_0 \in I_k$ with even $k$, the unique solution to  \eqref{eqtwostatekininc} moves according to
$\dot x=1-x$ until it reaches $a_k$, where it remains for ever, that is,
\[
x(t)=
\left\{
\begin{aligned}
& e^{-t}(x_0-1)+1, \quad t\le t_0=\ln (1-x_0)-\ln (1-a_k), \\
& a_k, t\ge t_0.
\end{aligned}
\right.
\]
Similarly, starting from  $x_0 \in I_k$ with odd $k$, the unique
solution to  \eqref{eqtwostatekininc} moves according to $\dot x=-x$
until it reaches $a_{k-1}$, where it remains for ever.

Comparing this behavior of the kinetic inclusion with the behavior of
the approximating Markov chain $X^N_x(t)$ (it was described in Section
\ref{sectwostatemod} and remains the same for the present discontinuous
case), we obtain the following.

\begin{prop}
\label{proptwostatemoddisc}
Under the above assumptions, the Markov chain $X^N_x(t)$ starting
at a point $x=k/N$, which does not equal any $a_k$ with odd $k$
converges weakly to the unique solution of the differential inclusion
\eqref{eqtwostatekininc} starting from the same point. Moreover,
this converges is even uniform in all times.
\end{prop}

\section{Nonexponential waiting times and fractional LLN}
\label{secfracLLN}

The key point in the probabilistic description of both Markov and nonlinear Markov chains
was the assumption that the waiting time between transitions is an exponential random variable.
Here we shall touch upon another extension of the results of Sections \ref{LLNsmooth} and
\ref{secLLNLip}, when more general waiting times are allowed. The characteristic feature
of the exponential waiting times is its memoryless property: the distribution of the remaining
 time does not depend on the time one has waited so far. Thus refuting exponential
 assumptions leads necessarily to some effects of memory and to non-Markovian evolutions.
Such evolutions are usually described by fractional in time differential equations,
which have drawn lots of attention in modern scientific literature. Let us obtain the dynamic LLN
for interacting multi-agent systems for the case of non-exponential waiting times with the
power tail distributions. As one can expect this LLN will not be deterministic anymore, and
our exposition in this section will be more demanding mathematically than in other parts of the book.

Let us say that a positive random variable $\tau$ has the probability law $\P$ on $[0,\infty)$
with a {\it power tail of index} $\al$ if
\[
\P(\tau > t) \sim \frac{1}{t^{\al}}
\]
for large $t$, that is the ration of the l.h.s. and the r.h.s tends to $1$, as $t \to \infty$.
For any $\al \in (0,1)$ such random variable has infinite expectation. Nevertheless, the index
$\al$ characterized in some sense the length of the waiting time (and thus the rate at which
the sequence of independent identically distributed random variables of this kind evolves), since
the waiting time increases with the decrease of $\al$.

As the exponential tails, the power tails are well suited for taking minima. Namely, if
$\tau_j$, $j=1, \cdots ,d$, are independent variables with a power tail of indices $\al_i$,
then $\tau=\min(\tau_1, \cdots, \tau_d)$ is clearly a variable with a power tail of index
$\al=\al_1+ \cdots + \al_d$.

Assume for simplicity that the
family $\{Q (x)\}$ does not depend on time.
Then in full analogy with the case of exponential times of Sections \ref{secmeanfieldch1}
or  \ref{secmeanfieldint} let us assume that the waiting time of the agents of type $i$ to change the strategy
has the power tail with the index $\al_i(x)=\al x_i |Q_{ii}|$ with some fixed $\al >0$. Consequently,
the minimal waiting time of all agents will have the probability law $P_x(dy)$ with a tail of the index
\[
\al A(x)=\sum \al_i(x)=\al \sum_i x_i |Q_{ii}(x)|.
\]

Our process with power tail waiting times can thus be described probabilistically
as follows. Starting from any time and current state $n$, or $x=n/N$, we wait a
random waiting time $\tau$, which has a power tail with the index $A(x)$. After
this time the choice of the pair $(i,j)$ such that the transition of one agent
from state $i$ to state $j$ occurs is carried out with probability \eqref{deftransprobmfMC}:
 \begin{equation}
\label{deftransprobmfMCrep}
\P(i\to j)= \frac{n_iQ_{ij}(n/N)}{\sum_k n_k |Q_{kk}(n/N)|}=\frac{x_iQ_{ij}(x)}{A(x)}.
\end{equation}
After such jump the process continues from the new state $x-e_i/N+e_j/N$ in the same way
as previously from $x$.

The easiest way to see, what kind of LLN one can expect under this setting and appropriate scaling, is to lift
the non-Markovian evolution on the space of sequences $x=(x_1, \cdots, x_d)\in \Si_d\cap \Z^d_+/N$ to the
discrete time Markov chain by considering the total waiting time $s$ as an additional space variable.
Namely, let us consider the Markov chain $(X^{N,\tau}_{x,s},S^{N,\tau}_{x,s})(k\tau)$ on $(\Si_d\cap \Z^d_+/N) \times \R_+$
with the jumps occurring at discrete times $k\tau$, $k\in \N$, such that the process at a state $(x,s)$ at time $\tau k$
jumps to $(x-e_i/N+e_j/N,s+\tau^{1/\al A(x)}y)$ with the probability distribution
\begin{equation}
\label{transprobmfMCtot}
P_x(dy)\P(i\to j) = P_x(dy) \frac{\tau x_iQ_{ij}(x)}{A(x)}.
\end{equation}

The key to the reasonable scaling is to choose $\tau =1/N$. We shall denote the
corresponding Markov chain  $(X^{\tau}_{x,s},S^{\tau}_{x,s})(k\tau)$. Its transition operator is
\begin{equation}
\label{transprobmfMCtot2}
U^{\tau}f(x,s)=\int P_x(dy)\sum_i \sum_{j\neq i} \frac{x_iQ_{ij}(x)}{A(x)}f\left(x-\frac{e_i}{N}+\frac{e_j}{N},s+\tau^{1/\al A(x)}y\right).
\end{equation}

What we are interested in is not exactly this chain, but the value of its first coordinate $X^{\tau}_{x,s}$
at the time $k\tau$, when the total waiting time $S^{\tau}_{x,s}(k\tau)$ (which is our real time, unlike the artificial time $s$) reaches $t$,
that is, at the time
\[
k\tau =T_{x,s}^{\tau}(t)=\inf \{ m\tau: S^{\tau}_{x,s} (m\tau) \ge t\},
\]
so that $T_{x,s}^{\tau}$ is the inverse process to $S^{\tau}_{x,s}$.
Thus the {\it scaled mean-field interacting system of agents with a power tail waiting time between jumps}
is the (non-Markovian) process
 \begin{equation}
\label{scaledmfchainwithtail}
\tilde X^{\tau}_{x,s}(t)= X^{\tau}_{x,s}(T_{x,s}^{\tau}(t)).
\end{equation}

Let us see what happens in the limit $\tau \to 0$, or equivalently $N=1/\tau \to \infty$.
Since the transitions in $\Si_d\cap \Z^d_+/N$ do not depend on the waiting times,
the first coordinate of the chain  $(X^{\tau}_{x,s},S^{\tau}_{x,s})(k\tau)$
is itself a discrete time Markov chain, which converges (by Theorems
\ref{thsiplestLLNdis} and \ref{thconv1dis}) to the deterministic process
described by the solutions $X^A_{x,s}(t)$ to the kinetic equations $\dot x=xQ(x)/A(x)$
with the initial condition $x$ at time $s$.  Thus we can expect that
the limit of the process \eqref{scaledmfchainwithtail}, which is the LLN we are looking for,
is the process obtained from the characteristics $X^A_{x,0}(t)$ evaluated at some random time $t$.

Let us see more precisely what happens with the whole chain $(X^{\tau}_{x,s},S^{\tau}_{x,s})(k\tau)$.
It is well known (see e.g. Theorem 8.1.1 of \cite{Ko11}) that if the operator
 \begin{equation}
\label{scalingMchains}
\La^Af= \lim_{\tau\to 0} \frac{1}{\tau}(U^{\tau}f-f)
\end{equation}
is well defined and generates a Feller process, then this Feller process represents the weak limit
of the scaled chain with the transitions $[U^{\tau}]^{[t/\tau]}$, where $[t/\tau]$ denotes the integer
 part of the number $t/\tau$. Thus we need to calculate \eqref{scalingMchains}. We have
\[
\frac{1}{\tau}(U^{\tau}f-f)
=\frac{1}{\tau}\int P_x(dy)\sum_i \sum_{j\neq i} \frac{x_iQ_{ij}(x)}{A(x)}
\left[f\left(x-\frac{e_i}{N}+\frac{e_j}{N},s+\tau^{1/\al A(x)}y\right)-f(x,s)\right]
\]
\[
=\frac{1}{\tau}\int P_x(dy)
\left[f(x,s+\tau^{1/\al A(x)}y)-f(x,s)\right]
\]
 \begin{equation}
\label{scalingMarlift}
+\frac{1}{\tau}\sum_i \sum_{j\neq i} \frac{x_iQ_{ij}(x)}{A(x)}
\left[f\left(x-\frac{e_i}{N}+\frac{e_j}{N},s\right)-f(x,s)\right]+R,
\end{equation}
where the error term equals
\[
R=\frac{1}{\tau}\int P_x(dy)\sum_i \sum_{j\neq i} \frac{x_iQ_{ij}(x)}{A(x)}
\]
\[
\times
\left[\left(f\left(x-\frac{e_i}{N}+\frac{e_j}{N},s+\tau^{1/\al A(x)}y\right)-f\left(x,s+\tau^{1/\al A(x)}y\right)\right)
-\left(f\left(x-\frac{e_i}{N}+\frac{e_j}{N},s\right)-f(x,s)\right)\right]
\]
\[
=\frac{1}{\tau}\int P_x(dy)\sum_i \sum_{j\neq i} \frac{x_iQ_{ij}(x)}{A(x)}
\left[g(x,s+\tau^{1/\al A(x)}y)-g(x,s)\right],
\]
with
\[
g(x,s)=f\left(x-\frac{e_i}{N}+\frac{e_j}{N},s\right)-f(x,s).
\]

From the calculations of Section \ref{LLNsmooth} we know that the second term in  \eqref{scalingMarlift} converges to
$\La f(x,s)/A(x)$, with
\[
\La f(x,s)=\sum_{i=1}^d \sum_{j\neq i} x_i Q_{ij}(x)[\frac{\pa f}{\pa x_j}-\frac{\pa f}{\pa x_i}](x,s)
=\sum_{k=1}^d \sum_{i\neq k} [x_i Q_{ik}(x)-x_k Q_{ki}(x)]\frac{\pa f}{\pa x_k}(x,s),
\]
whenever $f$ is continuously differentiable in $x$.

By Theorem 8.3.1 of \cite{Ko11} (see also \cite{Ko07a}), the first term in   \eqref{scalingMarlift} converges to
\[
\al A(x) \int_0^{\infty} \frac {f(x, s+y)-f(x,s)}{y^{1+\al A(x)}} dy,
\]
whenever $f$ is continuously differentiable in $s$.

To estimate the term $R$ we note that if $f \in C^1(\Si_d \times \R_+)$, then $g(x)$ is uniformly bounded by $1/N$
and $|\pa g/\pa s|$ is uniformly bounded. Hence applying again Theorem 8.3.1 of \cite{Ko11}  to $R$ we obtain that
$R/\tau \to 0$ as $\tau \to 0$. Hence it follows that, for $f \in C^1(\Si_d \times \R_+)$,
 \begin{equation}
\label{scalingMarlift1}
\La^Af(x,s)= \lim_{\tau\to 0} \frac{1}{\tau}(U^{\tau}f-f)(x,s)
=\al A(x) \int_0^{\infty} \frac {f(x, s+y)-f(x,s)}{y^{1+\al A(x)}} dy +\La f(x,s)/A(x).
\end{equation}

Thus we see again that the first coordinate of the chain  $(X^{\tau}_{x,s},S^{\tau}_{x,s})(k\tau)$
converges to the deterministic process
described by the solutions $X^A_{x,s}(t)$ (starting in $x$ at time $s$) to the kinetic equations $\dot x=xQ(x)/A(x)$. The second
coordinate of $(X^{\tau}_{x,s},S^{\tau}_{x,s})(k\tau)$ depends on the first coordinate. For $s=0$ it
 converges to the stable like process with the time dependent family of generators
 \begin{equation}
\label{scalingMarlift2}
\La^A_Sg(t)=\al A(X^A_{x,0}(t)) \int_0^{\infty} \frac {g(t+y)-g(t)}{y^{1+\al A(X^A_{x,0}(t))}} dy.
\end{equation}

Deriving the convergence of the processes from the convergence of the generators
as in Section \ref{LLNsmooth} and using \eqref{scaledmfchainwithtail} we arrive at the following result.

 \begin{theorem}
\label{thLLNfrac}
If $Q(x)$ is a Lipschitz continuous function, then the process of the mean-field interacting
system of agents $\tilde X^{\tau}_{x,s}(t)$ given by \eqref{scaledmfchainwithtail} converges in distribution to the process
 \begin{equation}
\label{eq1thLLNfrac}
\tilde X_{x,s}(t)= X^A_{x,s}(T_x^t),
\end{equation}
where $T_x^t$ is the random time when the stable-like process generated by \eqref{scalingMarlift2}
and started at $s$ reaches the time $t$.
Moreover, from the probabilistic
interpretation of the generalized Caputo-type derivative of mixed orders $A(x)$ (see \cite{Ko15}) it follows that the evolution
of averages $f(x,s)=\E f(\tilde X_{x,s}(t))$ satisfies the following generalized fractional differential equation
 \begin{equation}
\label{eq2thLLNfrac}
D_{t-*}^{A(x)}f(x,s)
= \frac{1}{A(x)}\sum_{k=1}^d \sum_{i\neq k} [x_i Q_{ik}(x)-x_k Q_{ki}(x)]\frac{\pa f}{\pa x_k}(x,s),
\quad s \in [0,t],
\end{equation}
with the terminal condition $f(x,t)=f(x)$,
where the left fractional derivative acting on the variable $s\le t$ of $f(x,s)$ is defined as
 \begin{equation}
\label{eq3thLLNfrac}
D_{t-*}^{A(x)}g(s) = \al A(x) \int_0^{t-s} \frac {g(s+y)-g(s)}{y^{1+\al A(x)}} dy
+\al A(x) (g(t)-g(s)) \int_{t-s}^{\infty}\frac {dy}{y^{1+\al A(x)}} dy.
\end{equation}
\end{theorem}

As usual with the nondeterministic LLN, the limiting process depends essentially on the scaling
and the choice of approximations. Instead of modeling the total waiting time by a random variable
with the tail of order $t^{-\al A(x)}$, we can equally well model this time by a random variable
with the tail of order $A^{-1}(x) t^{-\al }$, which would yield similar results with the mixed
fractional derivative of a simpler form
\[
\tilde D_{t-*}^{A(x)}g(s) = \frac{\al}{A(x)} \int_0^{t-s} \frac {g(s+y)-g(s)}{y^{1+\al}} dy
+\frac{\al (g(t)-g(s))}{A(x)} \int_{t-s}^{\infty} \frac {dy}{y^{1+\al}} dy
\]
\begin{equation}
\label{eq4thLLNfrac}
= \frac{\al}{A(x)} \int_0^{t-s} \frac {g(s+y)-g(s)}{y^{1+\al}} dy
+\frac{g(t)-g(s)}{A(x)(t-s)^{\al}}.
\end{equation}
However, any scaling of the jump times would lead to the process obtained from the characteristics
(solutions of $\dot x=xQ(x)$) via certain random time change.

The situation changes drastically if the family $Q(t,x)$ is time dependent. Then the partial decoupling
(possibility to consider the first coordinate $(X^{\tau}_{x,s}$ independent of the second one) does not occur,
and formula \eqref{eq1thLLNfrac} does not hold. Nevertheless, the fractional equation \eqref{eq2thLLNfrac}
extends directly to this case turning to the equation
 \begin{equation}
\label{eq5thLLNfrac}
D_{t-*}^{A(x)}f(x,s)
= \frac{1}{A(x)}\sum_{k=1}^d \sum_{i\neq k} [x_i Q_{ik}(s,x)-x_k Q_{ki}(s, x)]\frac{\pa f}{\pa x_k}(x,s),
\quad s \in [0,t].
\end{equation}
Unlike   \eqref{eq1thLLNfrac}, its solution can be represented probabilistically either via
the technique of Dynkin's martingale (see \cite{Ko15}) or via chronological Feynmann-Kac formula
(see \cite{Ko17a}).

\chapter{Dynamic control of major players}
\label{strategicprincipal}

Here we start exploiting another setting for the major player behavior. We shall assume
that the major player has some planning horizon with both running and (in case of a finite
horizon) terminal costs. For instance, running costs can reflect real spending and terminal
cost some global objective, like reducing the overall crime level by a specified amount.
This setting will lead us to the class of problem that can be called Markov decision
(or control) processes (for the principal) on the evolutionary background (of
permanently varying profiles of small players). We shall obtain the corresponding LLN limit,
both for discrete and continuous time. For discrete time the LLN limit turns to the deterministic
multi-step control problem  in case of one major player and to the deterministic multi-step
game between major players in case of several such players. In continuous time modeling
the LLN limit turns to the deterministic continuous time dynamic control problem in case
of one major player and to the deterministic differential game in case of several major players.
We analyze the problem both with finite and infinite horizons. The latter case is developed both
for the payoffs with discounting and without it. The last version leads naturally to the so-called
turnpike behavior of both limiting and prelimiting (finite $N$) evolutions. The theory of this chapter
also has an extension arising from non-exponential waiting time and leading to the control fractional
dynamics in the spirit of Section \ref{secfracLLN}, but we do not touch this extension here (see
however \cite{KolVer17}).

\section{Multi-step decision making of the principal}
\label{secmultistep1}

As above, we shall work with the case of a finite-state-space of small players,
so that the state space of the group is given by vectors $x=(n_1, \cdots, n_d)/N$
 from the lattice $\Z^d_+/N$.

Starting with a discrete time case,
we denote by $X_N(t,x,b)$ the Markov chain generated by \eqref{eqdefgenmeanfieldchainprin}
 with a fixed $b$ taken from certain convex compact subset of a Euclidean space, that is by the operator
\begin{equation}
\label{eqdefgenmeanfieldchainprinrep1}
L^{b,N}f (x)=N\sum_i x_i Q_{ij}(x,b)\left[f\left(x-\frac{e_i}{N}+\frac{e_j}{N}\right)-f(x)\right],
\end{equation}
and starting in $x\in \Z^d_+/N$ at the initial time $t=0$.

Let us assume that the principal is updating her strategy
 in discrete times $\{k\tau\}$, $k=0,1, \cdots ..., n-1$, with some fixed $\tau>0$, $n\in \N$, aiming
 at finding a strategy $\pi=\{b_0,b_1,\cdots , b_{n-1}\}$ maximizing the reward
\begin{equation}
\label{eqMarkdecevolback1}
V_n^{\pi,N}(x(N)) =\E_{N,x(N)}\left[\tau B(x_0,b_0)+\cdots +\tau B(x_{n-1},b_{n-1})+V(x_n)\right],
\end{equation}
where $B$ and $V$ are given functions (the running and the terminal payoff), $x(N)\in \Z^d_+/N$ also given,
\[
x_k=X_N(\tau,x_{k-1},b_{k-1}), \quad k=1,2, \cdots ,
\]
and $b_k=b_k(x_k)$ are measurable functions of the current state $x=x_k$
($\E_{N,x(N)}$ denotes the expectation specified by such process).

By the basic dynamic programming (see e.g. \cite{HeLa} or \cite{KolMal1}) the maximal rewards
$V_n^N(x(N))=\sup_{\pi} V_n^{\pi,N}(x(N))$
at different times $k$ are linked by the optimality equation
$V_k^N=S[N] V_{k-1}^N$,
where  the Shapley operator $S[N]$ (sometimes referred to as the Bellman operator)
is defined by the equation
\begin{equation}
\label{eqMarkdecevolback1a}
S[N]V(x)=\sup_b \left[\tau B(x,b)+\E V(X_N(\tau,x,b))\right],
\end{equation}
so that $V_n$
can be obtained by the $n$th iteration of the Shapley operator:
\begin{equation}
\label{eqMarkdecevolback1b}
V_n^N=S[N] V_{n-1}^N=S^n[N]V.
\end{equation}

We are again interested in the law of large numbers limit $N\to \infty$, where we expect the limiting
problem for the principal to be the maximization of the reward
\begin{equation}
\label{eqMarkdecevolback3}
V_n^{\pi}(x_0) =\tau B(x_0,b_0)+\cdots +\tau B(x_{n-1},b_{n-1})+V_0(x_n),
\end{equation}
where
\begin{equation}
\label{eqMarkdecevolback5}
x_0=\lim_{\N\to \infty} x(N)
\end{equation}
(which is supposed to exist) and
\begin{equation}
\label{eqMarkdecevolback6}
x_k=X(\tau, x_{k-1}, b_{k-1}), \quad k=1,2, \cdots,
\end{equation}
with $X(t,x,b)$ denoting the solution to the characteristic system (or kinetic equations)
\begin{equation}
\label{eqMarkdecevolback6a}
\dot x_k =\sum_{i=1}^d x_i Q_{ik}(x, b(x)), \quad k=1,...,d,
\end{equation}
with the initial condition $x$ at time $t=0$, or in the pressure and resistance framework,
\begin{equation}
\label{eqMarkdecevolback7}
\dot x_j=\sum_{i} \ka x_i x_j[R_j(x,b)-R_i(x,b)], \quad j=1,...,d.
\end{equation}

Again by dynamic programming, the maximal reward in this problem
$V_n(x)=\sup_{\pi} V_n^{\pi} (x)$, $\pi=\{b_k\}$,
is obtained by the iterations of the corresponding Shapley operator, $V_n =S^n V_0$, with
\begin{equation}
\label{eqMarkdecevolback8}
SV(x)=\sup_b \left[\tau B(x,b)+V(X(\tau,x,b))\right].
\end{equation}

Especially for the application to the continuous time models it is important to have estimates
of convergence uniform in $n=t/\tau$ for bounded total time $t=n\tau$.

As a preliminary step let us prove a rather standard fact about the propagation of continuity by the operator $S$.

\begin{prop}
\label{propcontdisccont}
Let $V$, $B$, $Q$ be bounded continuous functions, which are Lipschitz continuous in $x$, with
\begin{equation}
\label{eq0propcontdisccont}
\ka_V=\|V\|_{Lip}, \quad \ka_B=\sup_b \|B(.,b)\|_{Lip}<\infty,
\quad \om=\sup_b \|Q(.,b)\|_{bLip}<\infty.
\end{equation}
Then $S^nV\in C_{bLip}$ for all $n$ and
\begin{equation}
\label{eq1propcontdisccont}
\|S^nV\|\le t\|B\|+\|V\|, \quad \|S^nV\|_{Lip} \le (t\ka_B +\ka_V)e^{t \om }
\end{equation}
for $t=n\tau$.
\end{prop}

\begin{proof}
First equation in \eqref{eq1propcontdisccont} follows from the definition of $SV$. Next,
by \eqref{eqQinkin} and \eqref{eqODErhs},
\[
|X(\tau,x,b)-X(\tau,y,b)|\le |x-y| e^{t \om },
\]
and therefore $|SV(x)-SV(y)|$ does not exceed
\[
\sup_b [\tau B(x,b)+V(X(\tau,x,b))-\tau B(y,b)-V(X(\tau,y,b))]
\le |x-y| (\tau \ka_B +\ka_V e^{t \om }).
\]
Similarly,
\[
|S^2V(x)-S^2V(y)| \le \sup_b [\tau B(x,b)+SV(X(\tau,x,b))-\tau B(y,b)-SV(X(\tau,y,b))]
\]
\[
\le |x-y| (\tau \ka_B +\|SV\|_{Lip}e^{t \om })
\le |x-y| (\tau \ka_B +(\tau \ka_B +\ka_V  e^{t \om })e^{t \om }).
\]
By induction we obtain that
\[
\|S^nV\|_{Lip} \le \tau \ka_B (1+e^{\tau \om}+\cdots +e^{\tau \om(n-1)})+\ka_V e^{\tau \om n}
\le n\tau \ka_B e^{\tau \om n}+  \ka_V e^{\tau \om n},
\]
implying the second estimate in \eqref{eq1propcontdisccont}.
\end{proof}

\begin{theorem}
\label{thMarkDeconEvolBack1}
(i) Assume \eqref{eq0propcontdisccont} and \eqref{eqMarkdecevolback5} hold.
Then, for any $\tau \in (0,1]$, $n\in N$ and $t=\tau n$,
\begin{equation}
\label{eq1thMarkDeconEvolBack1}
\|S^n[N]V-S^nV\|\le C(d,\om) (t\ka_B +\ka_V)e^{t\om} (t \sqrt{1/(\tau N)}+|x(N)-x|)
\end{equation}
with a constant $C(d,\om)$ depending on $d$ and $\om$.
In particular, for $\tau=N^{-\ep}$ with $\ep \in (0,1)$, this turns  to
\begin{equation}
\label{eq1athMarkDeconEvolBack1}
\|S^n[N]V-S^nV\|\le C(d,\om) (t\ka_B +\ka_V)e^{t\om} (t N^{-(1-\ep)/2}+|x(N)-x|).
\end{equation}

(ii) If there exists a Lipshitz continuous optimal policy $\pi=\{b_k\}$, $k=1, \cdots , n$,
 for the limiting optimization problem,
then $\pi$ is approximately optimal for the $N$-agent problem, in the sense that for any $\ep>0$ there exists
$N_0$ such that, for all $N>N_0$,
\[
|V_n^N(x(N)) -V_n^{N,\pi}(x(N))|\le \ep.
\]
\end{theorem}

\begin{proof}
(i) Let $L=\sup_{k\le n}\|S^nV\|_{Lip}$. By \eqref{eq1propcontdisccont} it is bounded by
$(t\ka_B +\ka_V)e^{t \om }$. By \eqref{eq4thconv1},
\[
|S[N]V(x)-SV(x)|\le sup_b | \E V(X^N_x(\tau))-V(X_x(\tau))| \le C(d,\om)L\sqrt{\tau/N},
\]
where we used the condition $\tau \le 1$ to estimate $e^{3\tau \om } \le C(\om)$.
Next,
\[
 |S^2[N]V(x)-S^2V(x)| \le sup_b |\E S[N]V(X^N_x(\tau))-SV(X_x(\tau))|
 \]
 \[
 \le sup_b \E | S[N]V(X^N_x(\tau))- SV(X^N_x(\tau))|+sup_b |\E SV(X^N_x(\tau))-SV(X_x(\tau))|
\]
\[
\le  C(d,\om) L \sqrt{\tau/N}+C(d,\om) L \sqrt{\tau/N} \le 2C(d,\om) L \sqrt{\tau/N}.
\]
It follows by induction that
\begin{equation}
\label{eq2thMarkDeconEvolBack1}
 \|S^n[N]V-S^nV\|\le C(d,\om) n L \sqrt{\tau/N} = C(d,\om) t L \sqrt{1/(\tau N)},
\end{equation}
yielding \eqref{eq1thMarkDeconEvolBack1}.

(ii) One shows as above that for any Lipschitz continuous policy $\pi$,
the corresponding value functions $V^{\pi,N}$ converge. Combined with (i),
this yields Statement (ii).
\end{proof}

\section{Infinite horizon: discounted payoff}
\label{secdiscinfhor}

The standard optimization problem of infinite horizon planning related to the finite horizon problem
 of optimising \eqref{eqMarkdecevolback1} is the problem of maximizing the discounted sum
\begin{equation}
\label{eqMarkinfhor1}
\Pi^{\pi,N}(x(N)) =\E_{N,x(N)}\sum_{k=0}^{\infty} \tau \be^k B(x_k,b_k),
\end{equation}
with a $\be \in (0,1)$, where, as above,
\[
x_k=X_N(\tau,x_{k-1},b_{k-1}), \quad k=1,2, \cdots ,
\]
$\pi=\{b_k\}$ and $b_k=b_k(x_k)$ are measurable functions depending on the current state $x=x_k$.

In the law of large numbers limit, $N\to \infty$, we expect the limiting
problem for the principal to be the maximization of the reward
\begin{equation}
\label{eqMarkdecevolback4}
\Pi^{\pi}(x) =\sum_{k=0}^{\infty} \be^k \tau B(x_k,b_k)
\end{equation}
with $x_k=X(\tau,x_{k-1}, b_{k-1})$.

Notice firstly that
the solution to the finite-time discounting problem of the maximization of the payoff
 \begin{equation}
\label{eq1Markdecevolback2}
V_n^{\pi,N}(x(N)) =\E_{N,x(N)}\left[\tau B(x_0,b_0)+\cdots +\be^{n-1}\tau B(x_{n-1},b_{n-1})+\be^nV(x_n)\right],
\end{equation}
is given by the iterations
\begin{equation}
\label{eq2Markdecevolback2}
V_n^{N}=S_{\be}[N] V_{n-1}^N=S^n_{\be}[N]V
\end{equation}
of the corresponding discounted Shapley operator
\begin{equation}
\label{eq3Markdecevolback2}
S_{\be}[N]V(x)=\sup_b \left[\tau B(x,b)+\be \E V(X_N(\tau,x,b))\right].
\end{equation}
Similarly the solution to the corresponding limiting discounted problem
 \begin{equation}
\label{eq4Markdecevolback2}
V_n(x) =\max_{\pi}\left[\tau B(x_0,b_0)+\cdots +\be^{n-1}\tau B(x_{n-1},b_{n-1})+\be^nV(x_n)\right],
\end{equation}
is given by the iterations
\begin{equation}
\label{eq5Markdecevolback2}
V_n=S_{\be} V_{n-1}^N=S^n_{\be}V
\end{equation}
of the corresponding discounted Shapley operator
\begin{equation}
\label{eq6Markdecevolback2}
S_{\be}V(x)=\sup_b \left[\tau B(x,b)+\be V(X(\tau,x,b))\right].
\end{equation}

As a preliminary step let us recall a standard fact about the finite-step approximations to the optimal $\Pi^{\pi}$.

\begin{prop}
\label{propfiniteapprdiscdet}
Assume
\[
 \om=\sup_b \|Q(.,b)\|_{bLip}<\infty.
 \]
Let $B$ and $V$ be bounded continuous functions.
Then the sequence $S^n_{\be}V(x)$ converges, as $n\to \infty$, to the discounted infinite horizon optimal reward
\[
\Pi(x)=\sup_{\pi}\Pi^{\pi}(x),
\]
and the sequence $S^n_{\be}[N]V(x(N))$ converges, as $n\to \infty$, to the discounted infinite horizon optimal reward
\[
\Pi^{N}(x(N))=\sup_{\pi}\Pi^{\pi,N}(x(N)).
\]
\end{prop}

\begin{proof}
Since
\[
\|S^n_{\be}V-S^n_{\be}\tilde V\|\le \be^n \|V-\tilde V\|,
\]
\[
\|S^n_{\be}[N]V-S^n_{\be}[N]\tilde V\|\le \be^n \|V-\tilde V\|,
\]
it follows that, if the sequences $S^n_{\be}V$ or $S^n_{\be}[N]V$ converge for some $V$,
then these iterations converge to the same limit for any bounded $V$.
But for $V=0$ we see directly that
\[
\|S^n_{\be}V-\Pi\|\le \be^n \|B\| \frac{2}{1-\be},
\]
\[
\|S^n_{\be}[N]V-\Pi^N\|\le \be^n \|B\| \frac{2}{1-\be}.
\]
\end{proof}

\begin{theorem}
\label{thMarkDeconEvolBack2}
 Assume \eqref{eq0propcontdisccont} and \eqref{eqMarkdecevolback5} hold and let
\[
\be e^{\tau \om}\le \be_0 <1.
\]
Then the discounted optimal rewards
\[
\Pi^{N}(x(N))=\sup_{\pi}\Pi^{\pi,N}(x(N))
\]
converge, as $N\to \infty$ and $x(N)\to x$, to the discounted best reward
\[
\Pi(x)=\sup_{\pi}\Pi^{\pi}(x).
\]
\end{theorem}

\begin{proof}
Following the arguments of Proposition \ref{propcontdisccont}, we obtain
\[
\|S^n_{\be}V\| \le  \tau \|B\| \frac{1}{1-\be_0}+\be_0^n \|V\|,
\]
\[
\|S^n_{\be}V\|_{Lip} \le \tau \ka_B (1+\be e^{\tau \om}+\cdots +\be^{n-1}e^{\tau \om(n-1)})+\be^n\ka_V e^{\tau \om n}
\le \tau \ka_B \frac{1}{1-\be_0}+\be_0^n \ka_V,
\]
that is, unlike optimization without discounting, these norms are bound uniformly in the number of steps used.

Estimating the differences $\|S^n_{\be}[N]V-S^n_{\be}V\|$ as in the proof of Theorem \ref{thMarkDeconEvolBack1} yields
\[
\|S^n_{\be}[N]V-S^n_{\be}V\|\le C\frac{\sqrt{\tau}}{\sqrt N} (\be \|S^n_{\be}V\|_{bLip}+\cdots +\be^n \|S_{\be}V\|_{bLip})
\]
\[
\le C\frac{\sqrt{\tau}}{\sqrt N} \left(\tau \|B\|_{bLip} \frac{1}{(1-\be_0)^2}+\be_0^{n+1}\|V\|_{bLip}\right).
\]
Since  $S^n_{\be}V(x(N))$ converges to $\Pi(x)$, as $n\to \infty$, it follows that
$S^n_{\be}[N]V(x)$ converges to $\Pi(x)$, as $n\to \infty$ and $N\to \infty$.
\end{proof}

\begin{remark}
The optimal payoffs $\Pi(x)$ and $\Pi^N(x)$ are the fixed points of the Shapley operator:
$S_{\be}\Pi=\Pi$, $S_{\be}[N]\Pi^N=\Pi$. This fact can be used as a basis for another proof
of Theorem \ref{thMarkDeconEvolBack2}.
\end{remark}

\section{Continuous time modeling}
\label{secconttimemajor}

Here we initiate the analysis of the optimization problem for a forward-looking principal.
Namely, let the state space of the group being again given by vectors $x=(n_1, \cdots, n_d)/N$ from the lattice $\Z^d_+/N$,
but the efforts (budget) $b$ of the major player are chosen continuously in time aiming at optimizing
the payoff
\[
\int_t^T B(x(s),b(s)) \, ds +S_T(x(T))
\]
where $B,S_T$ are some continuous functions uniformly Lipschitz in all their variables. The optimal payoff of the major player is thus
\begin{equation}
\label{eqHJBmaj0}
S_N(t,x(N))=\sup_{b(.)}\E^N_{x(N)}  \left\{ \int_t^T B(x(s),b(s)) \, ds +S_T(x(T))\right\},
\end{equation}
where $\E^N_{x}$ is the expectation of the corresponding Markov process starting at the position $(x)$ at time $t$,
and $\tilde U$ is some class of controls (say, piecewise constant).
We are now in the standard Markov decision setting of a controlled Markov process generated by
the operator $L_{b,N}$ from \eqref{eqdefgenmeanfieldchainprinrep1}

As was shown above, the operators $L_{b,N}$ tend to a simple first order PDO, so that the limiting
optimization problem of the major player turns out to be the problem of finding
\begin{equation}
\label{eqlimoptimmajorconttime}
S(t,x)=\sup_{b(.)} \left\{ \int_t^T B(x(s),b(s)) ds +S_T(x(T))\right\},
\end{equation}
where $x(s)$ solve the system of equations
\begin{equation}
\label{eqlimoptimmajorconttimea}
\dot x_j=\sum_i x_i Q_{ij}(x,b), \quad j=1,...,d.
\end{equation}
The well-posedness of this system is a straightforward extension of the well-posedness of equations
\eqref{eqdefgenmeanfieldchainkineq1}.

It is well known (see e.g. \cite{KolMas} or any textbook on deterministic optimal control)
that the optimal payoff $S(t,x)$ of \eqref{eqlimoptimmajorconttime} represents the unique
generalized solution (so-called viscosity solution) to the HJB equation
\begin{equation}
\label{eqHJBa}
\frac{\pa S}{\pa t}+\sup_b \left[B(x,b) +\left(\frac{\pa S}{\pa x},  x Q(x,b)\right) \right] =0,
\end{equation}
with the initial (or terminal) condition $S(T,x)=S_T(x)$.

Instead of proving the convergence $ S_N(t,x(N)) \to S(t,x)$, we shall concentrate on a more
practical issue comparing the corresponding discrete time approximations, as these approximations are
usually exploited for practical calculations of $S_N$ or $S$.

The discrete-time approximation to the limiting problem of finding \eqref{eqlimoptimmajorconttime}
is the problem of finding
\begin{equation}
\label{eqMarkdecevolcontdisc1}
V_{t,n}(x)=\sup_{\pi} V_{t,n}^{\pi}(x) =\sup_{\pi}\left[\tau B(x_0,b_0)+\cdots +\tau B(x_{n-1},b_{n-1})+V(x_n)\right],
\end{equation}
where $\tau =(T-t)/n$, $x_0=x$, $V(x)=S_T(x)$ and
\begin{equation}
\label{eqMarkdecevolcontdisc2}
x_k=X(\tau, x_{k-1}, b_{k-1}), \quad k=1,2, \cdots,
\end{equation}
with $X(t,x,b)$ solving equation \eqref{eqMarkdecevolback6a}
with the initial condition $x$ at time $t=0$. It is known (see e. g. Theorem 3.4 of \cite{KolMas}) that the discrete approximations
$V_n(x)$ approach the optimal solution $S(T-t,x)$ given by \eqref{eqlimoptimmajorconttime} and solving the Cauchy problem
for the HJB \eqref{eqHJBa}.

The discrete-time approximation to the initial optimization problem
is the problem of finding
\[
V_{t,n}^N(x_0)=\sup_{\pi} V_{t,n}^{\pi,N}(x_0)
\]
\begin{equation}
\label{eqMarkdecevolcontdisc3}
 =\sup_{\pi}\E_{N,x(N)}
\left[\tau B(x_0,b_0)+\cdots +\tau B(x_{n-1},b_{n-1})+V_0(x_n)\right],
\end{equation}
where $x_k=X_N(\tau,x_{k-1},b_{k-1})$ with $X_N(t,x,b)$ denoting the Markov process with generator
\eqref{eqdefgenmeanfieldchainprinrep1} and with the strategies $\pi=\{b_k\}$ as in Section \ref{secmultistep1}.

\begin{remark}
It is also known, see e. g. Theorem 4.1 of \cite{FleSo}, that $V_n^N(x)$ with $V_0=S_T$ approach
the optimal solutions $S_N(T-t,x)$ given by \eqref{eqHJBmaj0} and solving certain HJB equation.
\end{remark}

\begin{theorem}
\label{thMarkDeconEvolcontdisc}
Let $B,S_T$ be uniformly Lipschitz in all their variables. Then,
for any $x$, $t\in [0,T]$ and $\tau=N^{-\ep}$ with $\ep \in (0,1)$,
\begin{equation}
\label{eq1thMarkDeconEvolcontdisc}
|V_{t,n}^N(x)-V_{t,n}(x)|\le C(d,\om,T) (\ka_B +\ka_V)(N^{-(1-\ep)/2}+|x(N)-x|).
\end{equation}
And consequently $V_{t,n}^N(x)$ converge, as $N\to \infty$ (and $n=(T-t)/\tau$, $\tau =N^{-\ep}$),
to the optimal solution $S(T-t,x)$ given by \eqref{eqlimoptimmajorconttime} and solving the
Cauchy problem for the HJB \eqref{eqHJBa}.
\end{theorem}

\begin{proof}
This is a direct consequence of Theorem \ref{thMarkDeconEvolBack1} (i).
\end{proof}

One can also imagine the situation that changing the budget bears some costs,
so that instantaneous adjustments of policies become unfeasible,
in which case the efforts (budget) $b$ of the major player can evolves depending
 on some more flexible control parameter $u\in U\in \R^r$, for instance, according to the equation
$\dot b=u$. In this case the payoff of the major player can be given, as usual, by the function
\[
\int_t^T J(x(s),b(s), u(s)) \, ds +S_T(x(T),b(T))
\]
where $J,S_T$ are some Lipschitz continuous functions.
The optimal payoff of the major player is then
\begin{equation}
\label{eqHJBmaj00}
S_N(t,x(N),b)=\sup_{u(.)}\E^N_{x(N),b}  \left\{ \int_t^T J(x(s),b(s),u(s)) \, ds +S_T(x(T),(b(T))\right\},
\end{equation}
where $\E^N_{x,b}$ is the expectation of the corresponding Markov process starting at the position $(x,b)$ at time $t$,
and the corresponding limiting optimal payoff  is
\begin{equation}
\label{eqlimoptimmajorconttime0}
S(t,x,b)=\sup_{u(.)} \left\{ \int_t^T J(x(s),b(s),u(s)) ds +S_T(x(T),(b(T))\right\},
\end{equation}
where $(x(s),(b(s))$ (depending on $u(.)$) solve the system of equations
\[
\dot b=u, \quad \dot x_j=\sum_i x_i Q_{ij}(x,b), \quad j=1,...,d.
\]
The well-posedness of this system is a straightforward extension of the well-posedness of equations
\eqref{eqdefgenmeanfieldchainkineq1}. Everything extends directly to this case,
where the limiting HJB equation becomes
\begin{equation}
\label{eqHJBa0}
\frac{\pa S}{\pa t}+\sup_u \left[J(x,b,u) +\left(\frac{\pa S}{\pa x},  x Q(x,b)\right) + u \frac{\pa S}{\pa b}\right] =0,
\end{equation}
with the initial (or terminal) condition $S(T,x,b)=S_T(x,b)$.

Similar analysis can be performed for the continuous-time discounted payoff setting, the
limiting payoff being then given by a solution to the corresponding stationary HJB equation.

\begin{remark}
Performing the same analysis for the case of non-exponential jumps of Section \ref{secfracLLN}
would lead to the limit that would not be controlled by the HJB equation
\eqref{eqHJBa}, but its fractional counterpart, the {\it fractional HJB}, derived in general setting setting in
\cite{KolVer14} and \cite{KolVer17}, which in the present case gets the form
 \begin{equation}
\label{eqHJBafr}
D_{T-*}^{A(x)}S(x,s)=\sup_b \left[B(x,b)+\left(\frac{\pa S}{\pa x},  x Q(x,b)\right) \right],
\quad s \in [0,T].
\end{equation}
\end{remark}

\section{Several major players}
\label{sectwoplayerzeromult}

Let us extend the above result to the case of several players competing on the background of
the pool of small players.

As was pointed out in Chapter \ref{chapintkolMaCor}, in the games with several principals one
can naturally distinguished two cases:  the pool of small players is common for the major
agents (think about advertising models, or several states fighting together a terrorist group)
or is specific to each major agent (think about generals controlling armies, big banks
controlling their subsidaries, etc). Of course mathematically, the second case can be embedded
in the first case (with an appropriate increase of dimension). However, the second case
is quite special, because in the deterministic limit (we are mostly concern with) it turns
to the differential games with the so-called {\it separated dynamics}, where each player
controls her own position independently of other players (like in the classical games of
pursuit and evasion). And this case is much better developed in the theory of games.

Let us start with the case of two major players, I and II,  playing a
zero-sum game. We shall first discuss the general case of a common pool of small players
 and then point out the specific features of 'separated dynamics'.

Let $b^1$ and $b^2$ denote the control parameters of players I
and II respectively, each taken from a compact subset of a Euclidean space. Let player I tries
to maximise the reward $B(x,b^1,b^2,N)$ and player II tries to minimise it.

The corresponding modification of Section \ref{secmultistep1} leads to the study
of the Marokv chains $X_N(t,x,b^1, b^2)$, which are the chains generated
by operators \eqref{eqdefgenmeanfieldchainprinrep1} with $b$ substituted by the pair $(b^1,b^2)$.

Starting again with the discrete setting we assume that players I and II update their strategies
 in discrete times $\{k\tau\}$, $k=0,1, \cdots ..., n-1$.
 The rules of the game are distinguished by the order of the moves of players I or II.

 In the case of the {\it upper game}\index{upper game} we assume that player I always makes the
 first move, which directly becomes known to player II. This leads to the problem of finding
\begin{equation}
\label{sectwoplayerzeromult1}
V_{n,up}^N(x(N))
=\inf_{\pi^1}\sup_{\pi^2} \E_{N,x(N)}
\left[\tau B(x_0,b_0^1, b_0^2)+\cdots +\tau B(x_{n-1},b^1_{n-1},b^2_{n-1})+V(x_n)\right],
\end{equation}
with $x_0=x(N)\in \Z^d_+/N$. Here strategies $\pi^1$ are rules that assign the controls $b_k^1$
 of player I on the basis of the history of the chain up to steps $k$, that is, on the basis of
 known $x_0, \cdots , x_k$, $b^1_0, \cdots , b^1_{k-1}$, $b^1_0, \cdots , b^1_{k-1}$, strategies
$\pi^2$ are rules that assign the controls $b_k^2$ of player II on the basis of the history of
the chain up to steps $k$ plus the choice $b^1_k$, that is, on the basis of known $x_0, \cdots , x_k$,
$b^1_0, \cdots , b^1_{k-1}, b^k$, $b^1_0, \cdots , b^1_{k-1}$, and
\[
x_k=X_N(\tau,x_{k-1},b^1_{k-1}, b^2_{k-1}), \quad k=1,2, \cdots.
\]
Similarly, in the case of the {\it lower game}\index{lower game} we assume that player II always makes the
 first move, which directly becomes known to player I. This leads to the problem of finding
\begin{equation}
\label{sectwoplayerzeromult2}
V_{n,low}^N(x(N))
=\sup_{\pi^2} \inf_{\pi^1} \E_{N,x(N)}
\left[\tau B(x_0,b_0^1, b_0^2)+\cdots +\tau B(x_{n-1},b^1_{n-1},b^2_{n-1})+V(x_n)\right].
\end{equation}
Here $\pi^2$ are strategies assigning controls $b_k^2$ of player II on the basis of the history
of the chain up to steps $k$ and  $\pi^1$ are strategies that assign controls $b_k^1$ of player I
 on the basis of the history of the chain up to steps $k$ plus the choice $b^2_k$.

The values $V_{n,up}^N$ and $V_{n,low}^N$ are called the {\it upper} and {\it low values}
of the multistage stochastic game with payoff $B(x,b^1,b^2,N)$ and Markov dynamics $X_N(\tau,x,b^1, b^2)$.

It is a standard fact (see e.g. \cite{PetrZen} or \cite{KolMal1}) that $V_{n,up}\ge V_{n,low}$
and each of these values can be calculated as the iterations of certain Shapley operators:
\begin{equation}
\label{sectwoplayerzeromult3}
V_{n,low}^N=S_{low}[N] V_{n-1,low}^N=S_{low}^n[N]V,
\quad V_{n,up}^N=S_{up}[N] V_{n-1,up}^N=S_{up}^n[N]V,
\end{equation}
where
\begin{equation}
\label{sectwoplayerzeromult4}
S_{low}[N]V(x)=\sup_{b^2} \inf_{b^1} \left[\tau B(x,b_1,b_2,N)+\E V(X_N(\tau,x,b_1,b_2))\right],
\end{equation}
\begin{equation}
\label{sectwoplayerzeromult5}
S_{up}[N]V(x)=\inf_{b^1} \sup_{b^2}  \left[\tau B(x,b_1,b_2,N)+\E V(X_N(\tau,x,b_1,b_2))\right],
\end{equation}

We are again interested in the law of large numbers limit, $N\to \infty$, where we expect the limiting
values $V_{n,up}$ and $V_{n,low}$ to be given by the corresponding problems with the limiting deterministic dynamics:

\begin{equation}
\label{sectwoplayerzeromult6}
V_{n,up}(x)
=\inf_{\pi^1}\sup_{\pi^2}
\left[\tau B(x_0,b_0^1, b_0^2)+\cdots +\tau B(x_{n-1},b^1_{n-1},b^2_{n-1})+V(x_n)\right],
\end{equation}
\begin{equation}
\label{sectwoplayerzeromult7}
V_{n,low}(x)
=\sup_{\pi^2} \inf_{\pi^1}
\left[\tau B(x_0,b_0^1, b_0^2)+\cdots +\tau B(x_{n-1},b^1_{n-1},b^2_{n-1})+V(x_n)\right],
\end{equation}
with $x_0=x$, the same rules for strategies as above, but with the dynamics
\[
x_k=X(\tau,x_{k-1},b^1_{k-1}, b^2_{k-1}), \quad k=1,2, \cdots,
\]
where $X(t,x,b^1,b^2)$ denotes the solution of the kinetic equation
 \begin{equation}
\label{sectwoplayerzeromult8}
\dot x_k =\sum_{i=1}^d x_i Q_{ik}(x, b^1,b^2), \quad k=1,...,d,
\end{equation}
with the initial data $x$ at time zero.

As in the case of finite $N$ games, the values $V_{n,up}$ and $V_{n,low}$
are given by the iterations of the corresponding Shapley operators:
\begin{equation}
\label{sectwoplayerzeromult9}
V_{n,low}=S_{low} V_{n-1,low}=S_{low}^nV,
\quad V_{n,up}=S_{up} V_{n-1,up}=S_{up}^nV,
\end{equation}
where
\begin{equation}
\label{sectwoplayerzeromult10}
S_{low}V(x)=\sup_{b^2} \inf_{b^1} \left[\tau B(x,b_1,b_2)+V(X(\tau,x,b_1,b_2))\right],
\end{equation}
\begin{equation}
\label{sectwoplayerzeromult11}
S_{low}V(x)=\inf_{b^1} \sup_{b^2}  \left[\tau B(x,b_1,b_2)+V(X(\tau,x,b_1,b_2))\right],
\end{equation}

The next analog of Theorem \ref{thMarkDeconEvolBack1} is straightforward.
\begin{theorem}
\label{thMarkDeconEvolBack1ze}
(i) Assume
\begin{equation}
\label{eqthMarkDeconEvolBack1ze}
\ka_V=\|V\|_{Lip}, \quad \ka_B=\sup_{b^1,b^2} \|B(.,b^1,b^2)\|_{Lip}<\infty,
\quad \om=\sup_{b^1,b^2} \|Q(.,b^1,b^2)\|_{bLip}<\infty.
\end{equation}
and \eqref{eqMarkdecevolback5}.
Then, for any $\tau \in (0,1]$, $n\in N$ and $t=\tau n$,
\[
\max(\|S_{low}^n[N]V-S_{low}^nV\|, \|S_{up}^n[N]V-S_{up}^nV\|)
\]
\begin{equation}
\label{eq1thMarkDeconEvolBack1ze}
\le C(d,\om) (t\ka_B +\ka_V)e^{t\om} (t \sqrt{1/(\tau N)}+|x(N)-x|)
\end{equation}
with a constant $C(d,\om)$ depending on $d$ and $\om$.
In particular, for $\tau=N^{-\ep}$ with $\ep \in (0,1)$, this turns  to
\[
\max(\|S_{low}^n[N]V-S_{low}^nV\|, \|S_{up}^n[N]V-S_{up}^nV\|)
\]
\begin{equation}
\label{eq1athMarkDeconEvolBack1ze}
\le C(d,\om) (t\ka_B +\ka_V)e^{t\om} (t N^{-(1-\ep)/2}+|x(N)-x|).
\end{equation}

(ii) If there exists Lipshitz continuous optimal policies $\pi^1, \pi^2$ for the low or upper game (the policies giving
minmax in \eqref{sectwoplayerzeromult6} or maxmin in \eqref{sectwoplayerzeromult7} respectively,
then $\pi^1, \pi^2$ are approximately optimal for the $N$-agent problem.
\end{theorem}

In continuous-time-setting the efforts $b^1,b^2$ of the major players are chosen continuously, with player I
aiming at maximizing and player II at minimizing
the payoff
 \begin{equation}
\label{eqHJBauppay}
\int_t^T B(x(s),b^1(s),b^2(s)) \, ds +S_T(x(T))
\end{equation}
where $J,S_T$ are some continuous functions uniformly Lipschitz in all their variables.

The discrete-time approximation to this game is the problem of finding
the upper and lower values \eqref{sectwoplayerzeromult1}, \eqref{sectwoplayerzeromult2}
with the limiting values \eqref{sectwoplayerzeromult6}, \eqref{sectwoplayerzeromult7}.

It is known (see e.g. \cite{FleSo}, \cite{PetrZen}) that under the assumptions of Lipschitz continuity of all functions involved,
the limit of values \eqref{sectwoplayerzeromult6} and \eqref{sectwoplayerzeromult7} as
$n \to \infty$, $\tau=t/n$, are the generalized (so-called {\it viscosity}) solutions
to the upper and lower HJB-Isaacs equations, respectively:
 \begin{equation}
\label{eqHJBaup}
\frac{\pa S}{\pa t}+\inf_{b^1}\sup_{b^2} \left[B(x,b^1,b^2) +\left(\frac{\pa S}{\pa x},  x Q(x,b^1,b^2)\right) \right] =0,
\end{equation}
 \begin{equation}
\label{eqHJBalow}
\frac{\pa S}{\pa t}+\sup_{b^2}\inf_{b^1} \left[B(x,b^1,b^2) +\left(\frac{\pa S}{\pa x},  x Q(x,b^1,b^2)\right) \right] =0,
\end{equation}
with the terminal condition $S_T$.

In many cases (see e.g. \cite{FleSo}, \cite{PetrZen}) one can show that solutions to \eqref{eqHJBaup} and \eqref{eqHJBalow} coincide,
the common value $V(t,x)$ being known as the {\it value}\index{value of differential game} of the zero-sum differential game with the dynamics
  \begin{equation}
\label{eq0thMarkDeconEvolcontdiscze}
\dot x_k =\sum_{i=1}^d x_i Q_{ik}(x, b^1,b^2), \quad k=1,...,d,
\end{equation}
and payoff function \eqref{eqHJBauppay}. Such differential games are said to {\it have a value}.

The following result is a straightforward extension of Theorem \ref{thMarkDeconEvolcontdisc}.
\begin{theorem}
\label{thMarkDeconEvolcontdiscze}
Assume that the game with dynamics \eqref{eq0thMarkDeconEvolcontdiscze}
and payoff \eqref{eqHJBauppay} has a value $V(t,x)$ and that all functions
involved as parameters are Lipschitz continuous Then,
for any $x$, $t\in [0,T]$ and $\tau=N^{-\ep}$ with $\ep \in (0,1)$,
both the upper and lower values \eqref{sectwoplayerzeromult1},
\eqref{sectwoplayerzeromult2} converge to $V(t,x)$, as $N\to \infty$.
\end{theorem}

Let us describe in more detail the setting of two players controlling different pools of minor players.
Suppose we are given two families of
$Q$-matrices $\{Q(t,u,x)=(Q_{ij})(u, x)\}$ and $\{P(t,v,x)=(P_{ij})(v, x)\}$, $i,j=1, \cdots d$, depending
on $x \in \Si_d$ and parameters $u$ and $v$ from two subsets $U$ and $V$ of Euclidean spaces.
Any given bounded measurable curves $u(t),v(t)$, $t\in [0,T]$, defines a Markov chain on
\[
(\Si_d \cap \Z^d_+/N)\times (\Si_d \cap \Z^d_+/M),
\]
 specified by the generator
 \[
L^{N,M}_{t,u(t),v(t)} f (x,y)=N\sum_{i=1}^d \sum_{j=1}^d x_i Q_{ij}(t,u(t),x)[f(x-e_i/N+e_j/N,y)-f(x,y)]
\]
 \begin{equation}
\label{eqdefgenmeanfieldchaininter2norm1}
+M\sum_{i=1}^d \sum_{j=1}^d y_i P_{ij}(t,v(t),y)[f(x, y-e_i/M+e_j/M)-f(x,y)].
 \end{equation}
For $f\in C^1(\Si_d\times \Si_d)$,
 \[
\lim_{N,M \to \infty} L_{t,u(t),v(t)}^{N,M}f (x, y)
=\La_{t,u(t),v(t)} f(x,y),
\]
where
 \[
\La_{t,u(t),v(t)} f(x,y)
 =\sum_{k=1}^d \sum_{i\neq k} [x_i Q_{ik}(t,u(t),x)-x_k Q_{ki}(t,u(t), x)]\frac{\pa f}{\pa x_k}(x,y)
 \]
 \begin{equation}
\label{eqdefgenmeanfieldchaininter2lim}
+\sum_{k=1}^d \sum_{i\neq k} [y_i P_{ik}(t,v(t),y)-y_k P_{ki}(t,v(t),y)]\frac{\pa f}{\pa y_k}(x,y).
\end{equation}
The corresponding controlled characteristics are governed by the equations
\begin{equation}
\label{eqdefgenmeanfieldchaininter2kineq}
\dot x_k
=\sum_{i \neq k} [x_i Q_{ik}(t,u(t),x)-x_k Q_{ki} (t,u(t),x)]=\sum_{i=1}^d x_i Q_{ik}(t,u(t),x), \quad k=1,...,d,
\end{equation}
\begin{equation}
\label{eqdefgenmeanfieldchaininter2kineq1}
\dot y_k
=\sum_{i \neq k} [y_i P_{ik}(t,v(t),y)-y_k P_{ki} (t,v(t),y)]=\sum_{i=1}^d y_i P_{ik}(t,v(t),y), \quad k=1,...,d.
\end{equation}

For a given $T>0$, let us denote by $\Ga(T,\tau)^N$ the discrete-time stochastic game with the dynamics specified by
the generator \eqref{eqdefgenmeanfieldchaininter2norm1}, with the objective of the player $I$ (controlling $Q$ via $u$) to maximize the payoff
$\E V_T(x(T),y(T))$  for a given function $V_T$ (terminal payoff), and with the objective of player $II$ (controlling $P$ via $v$) to minimize this payoff (zero-sum game),
 while the decision for choosing $u$ and $v$ are taken at time $\tau k$, $k\in \N$.
As previously we want to approximate it by the deterministic zero-sum differential game
$\Ga(T)$, defined by dynamics \eqref{eqdefgenmeanfieldchaininter2kineq}, \eqref{eqdefgenmeanfieldchaininter2kineq1}
and the payoff of player $I$ given by $V_T(X_{t,x}(T),Y_{t,y}(T)$.

For this game with the separated dynamics it is known that the value $V(t,x,y)$ exists (see e.g. \cite{FleSo} or \cite{Mal00}).
Hence Theorem \ref{thMarkDeconEvolcontdiscze} applies, so that the corresponding upper and lower discrete-timer values
converge to $V(t,x)$ as as $N\to \infty$, where $V$ is the generalized (viscosity or minimax) solution (see \cite{FleSo} or \cite{Subbotin})
of the HJB-Isaacs equation
 \begin{equation}
\label{eqHJBIsasep}
\frac{\pa S}{\pa t}+\sup_u \left(\frac{\pa S}{\pa x},  x Q(x,u)\right) + \inf_v\left(\frac{\pa S}{\pa y},  y P(y,v)\right) =0.
\end{equation}

One can similarly explore the corresponding competitive version of setting
\eqref{eqHJBmaj00}, \eqref{eqlimoptimmajorconttime0} that takes into account
 the costs of changing the budgets of the major players.

The situation with many major players is quite similar. Though the results on the Nash equilibria in
several-player differential games are less developed (see however \cite{Mal00}, \cite{PetrZen})
the general conclusion from the above discussion is as follows: {\it whenever the discrete approximations
to certain differential game have a common limit, the same common limit have the discrete approximations
to the corresponding stochastic game of the major players and} $N$ {\it minor players, as} $N\to \infty$,
 $\tau=N^{-\ep}$, $\ep \in (0,1)$.

 \begin{remark}
 \label{remongapinme}
 Original paper \cite{Ko12} dealing with Markov chains generated by \eqref{eqdefgenmeanfieldchaininter2norm1}
 as a gap in the argument showing directly the convergence of the limiting $\tau \to 0$ dynamics of fixed number of players
 to the limiting differential game. This gap is corrected by showing the convergence of the discrete approximations
 $\tau\to 0$ with $\tau \sim N^{-\ep} \sim M^{-\ep}$.
 \end{remark}

\section{Turnpike theory for nonlinear Markov games}
\label{secturnpikenonlinM}

In Section \ref{secdiscinfhor} we discussed the infinite horizon problems with discounting payoff.
Here we shall touch upon the analysis of long time optimization problems without discounting,
working in the framework of Section \ref{secconttimemajor}. This analyses also allows one to link
 the best response principal modeling of Chapter \ref{chbestres} with the present setting of a
 forward looking principal.

For simplicity, let us assume, as in Section \ref{secLLNmajor}, that for any $x$ there exists the unique
point of maximum $b^*(x)$ of the function $B(x,b)$.

Our analysis will be based on the following assumption (A) on the existence of an optimal steady
state for the kinetic dynamics:

(A) There exists a unique $x^*$ such that
  \begin{equation}
\label{eqturnpass}
0=B(x^*,b^*(x^*)) =\max_x B(x, b^*(x))= \max_{x,b} B(x,b).
\end{equation}
Moreover, $(x^*,b^*(x^*))$ is an asymptotically stable rest point of the kinetic equations, that is
 $x^*Q(x^*,b^*(x^*))=0$ and  there exists a neighborhood $U$ of $x^*$ such that for all $x\in U$ the solution
 $X_x(t)$ of the best response dynamics \eqref{eqkineqprinrep}, that is, of the equation $\dot x=xQ(x,b^*(x))$
 starting at $x$, converges to $x^*$ exponentially:
 \begin{equation}
\label{eqturnpass1}
\|X_x(t)-x^*\|\le e^{-\la t}\|x-x^*\|.
\end{equation}

Let us also assume a power type nondegeneracy of the maximum point $x^*$, that is, the following condition:

(B) For $x$ in $U$, $J$ behaves like a power function:
\begin{equation}
\label{eqturnpass2}
J(x, b^*(x)) \sim -\tilde J |x-x^*|^{\al}
\end{equation}
for some constants $\tilde J,\al >0$.

Under these assumptions the rest point solution $X(t)=x^*$ of the best
response dynamics \eqref{eqkineqprinrep} represents a {\it turnpike}\index{turnpike}
for the optimization problem \eqref{eqlimoptimmajorconttime}-\eqref{eqlimoptimmajorconttimea}
 in the following sense.

\begin{theorem}
\label{thturnpikeODE}
Under assumptions (A) and (B) let $x\in U$, $S_T$ be bounded,  and $\tilde X_x(s)$ be the optimal trajectory for
the optimization problem \eqref{eqlimoptimmajorconttime}-\eqref{eqlimoptimmajorconttimea}.
Then, for any sufficiently small $\ep$ and any time horizon $T-t$, the time $\tau$ this trajectory spends away from the $\ep$-neighborhood of $x^*$
satisfies the following bound:
\begin{equation}
\label{eq1thturnpikeODE}
\tau \le \frac{1}{\ep^{\al}} \left[|x-x^*|^{\al} \la^{-1} +(\sup S_T-S(x^*))\tilde J^{-1}\right].
\end{equation}
Thus for large $T-t$, an optimal trajectory spends most of the time near $x^*$.
Moreover, the payoff $V_t$ for the optimization problem \eqref{eqlimoptimmajorconttime}-\eqref{eqlimoptimmajorconttimea}
has the following bounds
\begin{equation}
\label{eq1athturnpikeODE}
C |x-x^*|^{\al}+  S(x^*)\le   V_t\le \sup S_T
\end{equation}
with a constant $C$.
\end{theorem}

\begin{proof}
The cost on the trajectory  $X_x(s)$ of the best response dynamics \eqref{eqkineqprinrep}
starting in $x$ at time $t$ is of order
\begin{equation}
\label{eq2thturnpikeODE}
S_T(x^*)-\int_t^T e^{-\la (s-t)} \tilde J |x-x^*|^{\al}\, ds
=S_T(x^*)- \frac{\tilde J}{\la} |x-x^*|^{\al}(1-e^{-(T-t)}).
\end{equation}
If a trajectory spends time $\tau$ away from the $\ep$-neighborhood of $x^*$, then
the cost on this trajectory does not exceed
\[
 -\tau \tilde J \ep^{\al}+\sup S_T.
 \]
 If it is an optimal trajectory it should perform better than $X_x(s)$. Comparing with
\eqref{eq2thturnpikeODE} yields \eqref{eq1thturnpikeODE}.  The lower bound in \eqref{eq1athturnpikeODE}
is again obtained from the trajectory $X_x(s)$.
\end{proof}

Combining Theorems \ref{thturnpikeODE} and \ref{thMarkDeconEvolcontdisc} we obtain the following result.

\begin{theorem}
\label{thturnpikeODE1}
Under the assumptions of Theorem \ref{thturnpikeODE} the optimal cost $V^N_{t,n}$ for the approximate Markov chain
(from Theorem \ref{thMarkDeconEvolcontdisc}) are approximated by the cost on the limiting best response dynamics $X_x(s)$.
\end{theorem}

For a general introduction to the turnpike theory, related background and bibliography we can refer to paper \cite{KoWei12}, see also monograph \cite{Zas06} for the case of deterministic control.


\chapter{Models of growth under pressure}
\label{secmodgrowthpres}

The results of this chapter extend the results of Chapters \ref{chbestres}
and \ref{strategicprincipal} to the case of a countable state-space of small players,
and moreover, to the case of processes that allow for the change of the number of particles
(thus going beyond the simple migrations that we played with so far), where physical particles
correspond in this setting to the coalitions (stable groups) of agents. This extension is
carried out in order to include important models of evolutionary coalition building, merging
and splitting (banks, subsidiaries, etc), strategically enhanced preferential attachment
and many other. The mathematics of this chapter is more demanding than in the rest of our
presentation, and its results are not used in other parts of the book. It is based on some elements
of infinite -dimensional analysis, the analysis of functions on the Banach space of sequences
$l_1$ and of the ODEs in this space. We start with brief description of the tools used.

\section{Preliminaries: ODEs in $l_1$}
\label{secprelonl1}

We recall some elementary results on the positivity preserving ODEs in $l_1^+$
starting with reminding some basic notations.
As usual, we denote by $\R^{\infty}=\{(x_1, x_2, \cdots )\}$ the space of sequences of real numbers
and by $l_1$ the Banach space of sumable sequences
\[
l_1=\{(x_1, x_2, \cdots ): \|x\|=\sum_j |x_j|<\infty\}.
\]
By $\R^{\infty}_+$ and $l_1^+$ we denote the subsets of theses space with non-negative coordinates $x_j$.

All notations for basic classes of regular functions extend automatically to this infinite-dimensional setting.
For instance, for a convex closed subset $Z$ of $l_1$, $C(Z)$ denotes the space of bounded continuous functions
equipped with the sup-norm: $\|f\|=\sup_x |f(x)|$. For these functions the Lipschitz constant is
 \begin{equation}
\label{eqdefLipnorml1l1}
\|f\|_{Lip} = \sup_{x\neq y} \frac{|f(x)-f(y)|}{\|x-y\|}=\sup_j \sup \frac{|f(x)-f(y)|}{|x_j-y_j|},
\end{equation}
where the last $\sup$ is the supremum over the pairs $x,y$ that differ only in its $j$th coordinate.
By $C_{bLip}(Z)$ we denote the space of bounded Lipschitz functions
with the norm $\|f\|_{bLip}=\|f\|+\|f\|_{Lip}$.
By $C^k(Z)$ we denote the space of functions on $Z$ with continuous uniformly bounded partial derivatives equipped with the norm
\[
\|f\|_{C^k(Z)}=\|f\|+\sum_{j=1}^k \| f^{(j)}(Z)\|,
\]
where $\|f^{(j)}(Z)\|$ is the supremum of the magnitudes of all partial derivatives of $f$ of order $j$.
As in case of functions on $\R^d$, $\|f\|_{C^1}=\|f\|_{bLip}$.
The space  $C(Z,l_1)$ of bounded continuous $l_1$-valued functions $f=(f_i):Z\to l_1$,
will be also usually denoted $C(Z)$, their norm being
$\|f\|=\sup_{x\in Z} \|f(x)\|$.

We shall denote by $(g,f)$ the inner product $(g,f)=\sum_{k=1}^{\infty} g_kf_k$ for the elements $g,f \in \R^{\infty}$
such that $\sum_{k=1}^{\infty} |g_k| \, |f_k| <\infty$.

As subsets $Z$ of $l_1$ we shall mostly use the sets
\[
\MC(L)=\{y\in \R^{\infty}: \sum_j L_j|y_j| < \infty \}, \quad \MC_{\le \la}(L)=\{y\in \MC(L): \sum_j L_j|y_j| \le \la \},
\]
\[
\MC^+(L)=\MC(L)\cap \R^{\infty}_+, \quad \MC^+_{\le \la}(L)=\MC_{\le \la}(L)\cap \R^{\infty}_+.
\]
for a non-decreasing sequence $L\in \R^{\infty}_+$ and  $\la >0$.
For our purposes we will need only two examples $L$: $L(j)=1$, for which $\MC(L)=l_1$, and $L(j)=j$.
If $L_j\to \infty$, as $j\to \infty$, the sets $\MC^+_{\le \la}(L)$ are easily seen to be compact in $l_1$.


The function $g:\MC^+(L)\to \MC(L)$ for some $L\in \R^{\infty}_+$ is called
{\it conditionally positive} if $g_j(x) \ge 0$ whenever $x_j=0$.
We will work with the $l_1$-valued ODEs
\begin{equation}
\label{eqeqinRinfty}
\dot x = g(x) \Longleftrightarrow \{ \dot x_j =g_j(x) \, \text{for all} \, j =1, \cdots \},
\end{equation}
with such r.h.s. $g$. Since $g$ may be defined only on $\MC^+(L)$, we are thus reduced to the analysis
of solutions that belong to $\MC^+(L)$ for all times. The vector $L$ (or the corresponding function on $\N$)
is said to be a {\it Lyapunov vector} (or a {\it Lyapunov function})
for equation \eqref{eqeqinRinfty} (or for its r.h.s. function $g$), if the {\it Lyapunov condition}
\begin{equation}
\label{eqeqinRinfLyiap}
(L,g(x))\le a (L,x)+b, \quad x\in \MC^+(L),
\end{equation}
holds with some constants $a,b$.

The Lyapunov function $L$ is called {\it subcritical} (resp. {\it critical}) for $g$,
or $g$ is said to be $L$- {\it subcritical} (resp. $L$- {\it critical}),
if $(L,g(x)) \le 0$ (resp. $(L,g(x))=0$) for all $x\in \MC^+(L)$. In the last case $L$
is also referred to as the {\it conservation law}.

\begin{theorem}
\label{ODERinfty}
Let $g:\MC^+(L) \to \MC(L)$ be conditionally positive and Lipschitz continuous
on any set $\MC^+_{\le \la}(L)$ with some Lipschitz constant $\ka(\la)$, and $L$ be the Lyapunov vector for $g$.
Then for any $x\in \MC^+(L)$
there exists a unique global solution $X(t,x)\in \MC^+(L)$ (defined for all $t\ge 0$) of equation \eqref{eqeqinRinfty}
in $l_1$ with the initial condition $x$. Moreover, if $x\in \MC^+_{\le \la}(L)$ and $a\neq 0$, then
\begin{equation}
\label{eq1ODERinfty}
X(t,x)\in \MC^+_{\le \la(t)}(L),
\quad \la(t)=e^{at}\left(\la+\frac{b}{a}\right) - \frac{b}{a}=e^{at}\la +(e^{at}-1)\frac{b}{a}.
\end{equation}
If $a=0$, then the same holds with $\la(t)= \la +bt$, and if $f$ is $L$-critical, then with $\la(t)=\la$.
Finally, $X(t,x)$ is Lipschitz as the function of $x$:
  \begin{equation}
\label{eq2ODERinfty}
\|X(t,x)-X(t,y)\| \le e^{\ka(\la(t))}\|x-y\|.
\end{equation}
\end{theorem}

\begin{remark}
Intuitively, this result is clear. In fact, by conditional positivity, the vector field $g(x)$
on any boundary point of $\R^{\infty}_+$ is directed inside or tangent to the boundary, thus not allowing
a solution to leave it. On the other hand, by the Lyapunov condition,
\[
(L,X(t,x)) \le (L,x)+a\int_0^t (L, X(s,x)) \, ds +bt
\]
implying \eqref{eq1ODERinfty} by Gronwall's lemma.
However, a rigorous proof is not fully straightforward, because already the existence of a solution
is not clear: an attempt to construct it via a usual fixed-point argument reducing it to the integral equation
$x(t)=\int_0^tg(x(s))\,ds$, or by the standard Euler or Peano approximations encounters a problem, since all these approximations
may not preserve positivity (and thus may jump out of the domain where $g$ is defined).
\end{remark}

\begin{proof}
Assuming $a\neq 0$ for definiteness, let
\[
M^x_{a,b}(t)=\left\{y\in \R^{\infty}_+: (L,y)\le e^{at}\left((L,x)+\frac{b}{a}\right) - \frac{b}{a} \right\}.
\]
Fixing $T$, let us define the space $C_{a,b}(T)$ of continuous functions $y: [0,T] \mapsto \R_+^{\infty}$ such that
$y([0,t]) \in M^x_{a,b}(t)$ for all $t\in [0,T]$.

Let $g_L=g_L(x)$ be the maximum of the Lipschitz constants of all $g_j$ on $M^x_{a,b}(T)$.
Let us pick up a constant $K=K(x)\ge g_L$. By conditional positivity, $g_j(y) \ge -Ky_j$ in $M^x_{a,b}(T)$,
because
\[
\|g_j(y)-g_j(y_1, \cdots, y_{j-1},0,y_{j+1}, \cdots )\|\le Ky_j
\,\, \text{and} \,\, g_j(y_1, \cdots, y_{j-1},0,y_{j+1}, \cdots )\ge 0.
\]
Hence, by rewriting equation \eqref{eqeqinRinfty} equivalently as
\begin{equation}
\label{eqeqinRnneglin}
\dot y = (g(y)+Ky)-Ky \, \Longleftrightarrow \, \{  \dot y_j = (g_j(y)+Ky_j)-Ky_j, \quad j =1, \cdots , n \},
\end{equation}
we ensure that the 'nonlinear part' $g_j(y)+Ky_j$ of the r.h.s. is always non-negative.

We modify the usual approximation scheme (see remark above) for ODEs
by defining the map $\Phi_x$ from $C_{a,b}(T)$ to itself in the following way:
for a $y\in C([0,T], \R_n^+)$ let $\Phi_x(y)$ be the solution of the equation
\[
\frac{d}{dt} [\Phi_x(y)](t) = g(y(t))+Ky(t) -K[\Phi_x(y.)](t),
\]
with the initial data $[\Phi_x(y)](0)=x$. It is a linear equation with the unique explicit solution,
which can be taken as an alternative definition  of $\Phi_x$:
\[
[\Phi_x(y)](t)=e^{-Kt}x+\int_0^te^{-K(t-s)}[g(y(s))+Ky(s)] \, ds.
\]
Clearly, $\Phi_x$ preserves positivity and fixed points of $\Phi$ are positive solutions to \eqref{eqeqinRnneglin} with the initial data $x$.

Let us check that $\Phi$ takes $C_{a,b}(T)$ to itself. If $y\in C_{a,b}(T)$,
\[
(L, [\Phi_x(y](t)))=e^{-Kt} (L,x)+\int_0^t e^{-K(t-s)}[(a+K)(L,y(s))+b] \, ds
\]
\[
\le e^{-Kt} (L,x)+\int_0^t e^{-K(t-s)}[(a+K) \left( e^{as}\left((L,x)+\frac{b}{a}\right) - \frac{b}{a}\right) +b] \, ds
 \]
 \[
=(L,x)e^{-Kt}+e^{-Kt}(e^{(K+a)t}-1) ((L,x)+b/a)-(b/a)e^{-Kt} (e^{Kt}-1)
\]
\[
 =(L,x)e^{at}+\frac{b}{a} (e^{at}-1).
  \]
Notice that, due to this bound, the iterations of $\Phi$ remain in $C_{a,b}(T)$
and hence it is justified to use the Lipshitz constant $K$ for $g$.

Next, the map $\Phi(t)$ is a contraction for small $t$, because
 \[
\|[\Phi_{x_1} (y^1_.)](t)-[\Phi_{x_2} (y^2_.)](t)\|
\le (K+g_L) \int_0^t \|y^1_s-y^2_s\| \, ds +\|x_1-x_2\|
\]
and the proof of the existence and uniqueness of a fixed point is completed by the usual application of the Banach fixed point principle.

Finally,
\[
\|X(t,x)-X(t,y)\|\le \int_0^t \ka(\la(t))\|X(s,x)-X(s,y)\| \, ds +\|x-y\|,
\]
and \eqref{eq2ODERinfty} follows by usual Gronwall's lemma.
\end{proof}

Once the solution $X(t,x)$ are constructed the linear operators $T^t$,
\begin{equation}
\label{eqdefflowonkineq}
T^tf(x)=f(X(t,x)), \quad t\ge 0,
\end{equation}
become well defined contractions in $C(\MC^+(L))$ forming a semigroup. In case $a=b=0$,
the operators $U^t$ form a semigroup of contractions also in $C(\MC_{\le \la}^+(L))$ for any $\la$.

\begin{theorem}
\label{lemonsenseBanach1}
Under the assumptions of Theorem \ref{ODERinfty} assume additionally
that $g$ is twice continuously differentiable on $\MC^+(L)$ so that
for any $\la$
\begin{equation}
\label{eq1lemonsenseBanach1}
 \|g^{(2)}(\MC^+_{\le \la}(L))\| \le D_2(\la)
\end{equation}
with some continuous function $D_2$.
Then the solutions to $X(t,x)$ from Theorem \ref{ODERinfty} are twice continuously differentiable with respect to initial data and
\begin{equation}
\label{eq1lemonsenseBanach10}
\|X(t,.)^{(1)}(\MC^+_{\le \la}(L))\|\le e^{t\ka(\la(t))},
\end{equation}
\begin{equation}
\label{eq2lemonsenseBanach1}
 \|X(t,.)^{(2)}(\MC^+_{\le \la}(L))\|=\sup_{j,i}\sup_{x\in \MC^+_{\le \la}(L)}  \|\frac{\pa ^2X(t,x)}{\pa x_i \pa x_j}\|
\le t D_2(\la(t)) e^{3t\ka(\la(t))}.
\end{equation}
Moreover,
\begin{equation}
\label{eq3lemonsenseBanach1}
\|(T^tf)^{(1)}(\MC^+_{\le \la}(L))\| =\sup_j \sup_{x\in \MC^+_{\le \la}(L)} \left| \frac{\pa}{\pa x_j}f(X(t,x))\right|
\le \|f^{(1)}(\MC^+_{\le \la}(L))\| e^{t\ka(\la(t))},
 \end{equation}
  \[
\|(T^tf)^{(2)}(\MC^+_{\le \la}(L))\| =\sup_{j,i} \sup_{x \in \MC^+_{\le \la}(L)} \left| \frac{\pa^2}{\pa x_j \pa x_i}f(X(t,x))\right|
\]
 \begin{equation}
\label{eq4lemonsenseBanach1}
\le \|f^{(2)}(\MC^+_{\le \la (t)}(L))\|   e^{2t\ka(\la(t))}+t\|f^{(1)}(\MC^+_{\le \la(t)}(L))\| D_2(\la(t))  e^{3t\ka(\la(t))}
\end{equation}
\end{theorem}

\begin{proof}
Differentiating the equation $\dot x=g(x)$ with respect to initial conditions yields
\begin{equation}
\label{eq5alemonsenseBanach1}
\frac{d}{dt} \frac{\pa X(t,x)}{\pa x_j}=\sum_k \frac{\pa g}{\pa x_k} (X(t,x)) \frac{\pa X_k(t,x)}{\pa x_j},
\end{equation}
implying \eqref{eq1lemonsenseBanach10}.
Differentiating the equation $\dot x=f(x)$ twice yields
\begin{equation}
\label{eq6blemonsenseBanach1}
\frac{d}{dt} \frac{\pa ^2X(t,x)}{\pa x_j \pa x_i}=\sum_k \frac{\pa g}{\pa x_k} (X(t,x)) \frac{\pa ^2X_k(t,x)}{\pa x_j \pa x_i}
+\sum_{k,l} \frac{\pa ^2g}{\pa x_k \pa x_l} (X(t,x))\frac{\pa X_k(t,x)}{\pa x_j}\frac{\pa X_l(t,x)}{\pa x_i}.
\end{equation}
Solving this linear equation with the initial condition  $\pa ^2X(0,x)/\pa x_j \pa x_i=0$ yields \eqref{eq2lemonsenseBanach1}.

Differentiating \eqref{eqdefflowonkineq} yields
\begin{equation}
\label{eq7lemonsenseBanach1}
\frac{\pa}{\pa x_j}(T^tf)(x)=\sum_k \frac{\pa f}{\pa x_k}(X(t,x))\frac{\pa X_k(t,x)}{\pa x_j},
 \end{equation}
 implying \eqref{eq3lemonsenseBanach1}. Differentiating second time yields
 \begin{equation}
\label{eq8lemonsenseBanach1}
\frac{\pa^2}{\pa x_j \pa x_i}(T^tf)(x)=\sum_k \frac{\pa f}{\pa x_k}(X(t,x))\frac{\pa ^2X_k(t,x)}{\pa x_j\pa x_i}
+\sum_{k,l} \frac{\pa ^2f}{\pa x_k \pa x_l} (X(t,x))\frac{\pa X_k(t,x)}{\pa x_j}\frac{\pa X_l(t,x)}{\pa x_j},
 \end{equation}
implying \eqref{eq4lemonsenseBanach1}.
\end{proof}

\begin{remark}
The proof above is not complete. All estimates are proved based on the
assumption that the required derivatives exist. However, the corresponding
justification argument are standard in the theory of ODEs and thus are not
reproduced here, see e.g. \cite{Konewbook}.
\end{remark}

\section{Mean-field interacting systems in $l^1$}
\label{secmeanfieldinl1}

In this chapter  we extend the main results of Chapters \ref{chbestres}
and \ref{strategicprincipal} in two directions, namely, by working with
a countable (rather than finite or compact) state-space and unbounded
rates, and with more general interactions allowing in particular for
a change in the number of particles.

Therefore we take the set of natural numbers $\{1,2, \cdots \}$ as
the state space of each small player, the set of finite Borel measures
on it being the Banach space $l^1$ of sumable real sequences $x=(x_1, x_2, \cdots ) $.

Thus the state space of the total multitude of
small players will be formed by the set $\Z^{fin}_+$ of sequences of integers
$n=(n_1, n_2, \cdots )$ with only finite number of non-vanishing ones,
with $n_k$ denoting the number of players in the state $k$, the total number
of small players being $N=\sum_k n_k$.
As we are going to extend the analysis to processes not preserving the number of particles,
we shall work now with a more general scaling
of the states, namely with the sequences
\[
x=(x_1,x_2, \cdots ...)=hn =h(n_1,n_2, \cdots ...)\in h\Z_+^{fin}
\]
with the parameter $h=1/N_0$, the inverse number to
the total number of players $\sum_k n_k$ at the initial moment of observation.
The necessity to distinguish initial moment is crucial here, as this number changes over time.
Working with the scaling related to the current number of particles $N$ may lead, of course,
to different evolutions.

The general processes of birth, death, mutations and binary interactions that
 can occur under the effort $b$ of the principal are Markov chains on $h\Z_+^{fin}$
 specified by the generators of the following type
\[
L_{b,h}F(x)=\frac{1}{h}\sum_j \be_j(x,b)[F(x+he_j)-F(x)]
+\frac{1}{h}\sum_j \al_j(x,b)[F(x-he_j)-F(x)]
\]
\[
+\frac{1}{h}\sum_{i,j} \al_{ij}^1(x,b)[F(x-he_i+he_j)-F(x)]
+\frac{1}{h}\sum_{i,(j_1,j_2)} \al_{i(j_1j_2)}^1(x,b)[F(x-he_i+he_{j_1}+he_{j_2})-F(x)]
\]
\[
+\frac{1}{h}\sum_{(i_1,i_2),j} \al^2_{(i_1i_2)j}(x,b)[F(x-he_{i_1}-he_{i_2}+he_j)-F(x)]
\]
\begin{equation}
\label{eqgenbininterpres}
+\frac{1}{h}\sum_{(i_1,i_2)}\sum_{(j_1,j_2)}
\al^2_{(i_1i_2)(j_1j_2)}(x,b)[F(x-he_{i_1}-he_{i_2}+he_{j_1}+he_{j_2})-F(x)],
\end{equation}
where brackets $(i,j)$ denote the pairs of states. Here the terms with $\be_j$
and $\al_j$ describe the spontaneous injection (birth) and death of agents,
the terms with $\al^1_{ij}$ describe the mutation or migration of single agents,
the terms with $\al^1_{i(lj)}$ describe the fragmentation or splitting, the terms
with $\al^2$ describe the binary interactions, which can result in either merging
of two agents producing an agent with another strategy or their simultaneous
migration to any other pair of strategies. All terms include possible mean-field
interactions. Say, our model \eqref{eqdefgenmeanfieldchainnorm1} was an example of migration.

Let $L$ be a strictly positive non-decreasing function on $\N$. We shall refer
to such functions as Lyapunov functions. We say that the generator $L_{b,h}$
with $\be_j=0$ and the corresponding process do not increase $L$ if for any allowed
transition the total value of $L$ cannot increase, that is if $\al_{ij}^1 \neq 0$,
then $L(j)\le L(i)$, if  $\al_{i(j_1,j_2)}^1 \neq 0$, then $L(j_1)+L(j_2)\le L(i)$,
if $\al_{(i_1i_2)j}^2 \neq 0$, then $L(j)\le L(i_1)+L(i_2)$, if
$\al_{(i_1i_2)(j_1j_2)}^2 \neq 0$, then $L(j_1)+L(j_2)\le L(i_1)+L(i_2)$.
If this is the case, the chains generated by $L_{b,h}$ always remain in
$\MC^+_{\le \la}(L)$, if they were started there. Hence, if $L_j\to \infty$,
as $j\to \infty$, the sets $h\Z_+^{fin}\cap \MC^+_{\le \la}(L)$ are finite for
any $h$ and $\la$, and $L_{b,h}$ generates well-defined Markov chains
$X_{b,h}(t,x)$ in any of these sets.

A generator $L_{b,h}$ is called $L$-subcritical if $L_{b,h}(L) \le 0$. Of course,
if $L_{b,h}$ does not increase $L$, then it is $L$-subcritical. Though the condition
to not increase $L$ seems to be restrictive, many concrete models satisfy it, for
instance, the celebrated merging-splitting (Smoluchovskii) process considered below.
On the other hand, models with spontaneous injections may increase $L$, so that one
is confined to work with the weaker property of sub-criticality. We shall first
analyse the case of chains that do not increase $L$, and then will consider
subcritical evolutions and evolutions allowing for a mild growth of $L$.

By Taylor-expanding $F$ in \eqref{eqgenbininterpres} one sees that if $F$ is
sufficiently smooth, the sequence $L_{b,h}F$ converges to
\[
\La_bF(x)=\sum_j (\be_j(x,b)-\al_j(x,b))\frac{\pa F}{\pa x_i}
+\sum_{i,j} \al_{ij}^1(x,b)[\frac{\pa F}{\pa x_j}-\frac{\pa F}{\pa x_i}]
\]
\[
+\sum_{i,(j_1,j_2)} \al_{i(j_1j_2)}^1(x,b)[\frac{\pa F}{\pa x_{j_1}} +\frac{\pa F}{\pa x_{j_2}}-\frac{\pa F}{\pa x_i}]
+\sum_{(i_1,i_2)}\sum_j \al^2_{(i_1i_2)j}(x,b)[\frac{\pa F}{\pa x_j}-\frac{\pa F}{\pa x_{i_1}}-\frac{\pa F}{\pa x_{i_2}}]
\]
\begin{equation}
\label{eqgenbininterpreslim}
+\sum_{(i_1,i_2)}\sum_{(j_1,j_2)} \al^2_{(i_1i_2)(j_1j_2)}(x,b)
[\frac{\pa F}{\pa x_{j_1}} +\frac{\pa F}{\pa x_{j_2}}-\frac{\pa F}{\pa x_{i_1}}-\frac{\pa F}{\pa x_{i_2}}].
\end{equation}
Moreover,
\begin{equation}
\label{eqgenbininterpreslim1}
\|(L_{b,h}-\La_b)F\|_{C(\MC^+_{\le \la}(L))}
 \le 8h \ka (\la) \|F\|_{C^2(\MC^+_{\le \la}(L))},
\end{equation}
with $\ka(L,\la)$ being the $\sup_b$ of the norms
\begin{equation}
\label{eqgenbininterpreslim1a}
\left\|\sum_i (\al_i+\be_i) +\sum_{i,j} \al_{ij}^1+\sum_{i,(j_1,j_2)} \al^1_{i(j_1j_2)}
+\sum_{(i_1,i_2),j} \al^2_{(i_1i_2)j}
+\sum_{(i_1,i_2),(j_1,j_2)} \al^2_{(i_1i_2)(j_1j_2)}\right\|_{C(\MC^+_{\le \la}(L))}.
\end{equation}

By regrouping the terms of $\La_b$, it can be rewritten in the form of the general first order operator
\begin{equation}
\label{eqgenbininterpreslim2}
\La_b F(x)=\sum_j g_j(x,b)\frac{\pa F}{\pa x_j},
\end{equation}
where $g=(g_i)$ with
\[
g_i=\be_i -\al_i+\sum_k (\al^1_{ki}-\al^1_{ik})
+\sum_k [\al^1_{k(ii)}+\sum_{j\neq i} (\al^1_{k(ij)}+\al^1_{k(ji)}]-\sum_{(j_1,j_2)} \al^1_{i(j_1j_2)}
\]
\[
+\sum_{(j_1,j_2)} \al^2_{(j_1j_2)i}-\sum_k [\al^2_{(ii)k}+\sum_{j\neq i} (\al^2_{(ij)k}+\al^2_{(ji)k})]
\]
\[
+\sum_{(j_1,j_2)}[\al^2_{(j_1j_2)(ii)}+\sum_{j\neq i} (\al^2_{(j_1j_2)(ji)}+\al^2_{(j_1j_2)(ij)})]
\]
\[
-\sum_{(j_1,j_2)}[\al^2_{(ii)(j_1j_2)}+\sum_{j\neq i} (\al^2_{(ij)(j_1j_2)}+\al^2_{(ji)(j_1j_2)})].
\]

Its characteristics solving the ODE $\dot x =g(x)$ can be expected to describe the limiting behavior
of the Markov chains $X_{b,h}(x,t)$ for $h\to 0$.

As in finite-dimensional case we can build the semigroups of transition operators $U^t_hF(x)= \E F(X_h(t,x))$,
$U^tF(x)= \E F(X(t,x))$ and we are going to search for conditions ensuring that $U^t_h \to U^t$ as $h\to 0$.

Let us stress again that the present models are much more general than the models of Chapter \ref{chbestres}.
The present $g$ cannot always be written in the form $g=xQ(x)$ with a $Q$-matrix $Q$ describing the evolution
of one player, because the possibility of birth and death makes the tracking of a particular player throughout
the game impossible.

\section{Convergence results for evolutions in $l^1$}
\label{secconvcount}

As in  Chapter \ref{chbestres} let us start with the case of smooth coefficients.

\begin{theorem}
\label{th1inl1}
Assume the operators $L_{b,h}$ are $L$ non-increasing for
a (positive nondecreasing) Lyapunov function $L$ on $\Z$,
and the function $g: \MC^+_{\le \la}(L) \to l^1$ belongs to $C^2(\MC^+_{\le \la}(L))$.
Then the Markov chains $X_h(t,x(h))$ with $x(h)\in  \MC^+_{\le \la}(L)$
such that $x(h)\to x$, as $h\to 0$, converge in distribution to the deterministic
evolution $X(t,x)$ and moreover
\[
\|U^t_hF-U^tF\|=\sup_x|\E F (X_h(t,x))-F(X(t,x))|
\]
\begin{equation}
\label{eq1th1inl1}
\le 8th\ka (L,\la) (\|F^{(2)}\| +t\|F^{(1)}\| \, \|g^{(2)}\|) e^{3t\|g\|_{Lip}},
\end{equation}
where all norms are understood as the $\sup_b$ of the norms of functions on $\MC^+_{\le \la}(L)\subset l^1$.
\end{theorem}

\begin{proof}
The proof is the same as the proof of Theorem \ref{thsiplestLLN}.

All function-norms below are the norms of functions defined on $\MC^+_{\le \la}(L)$.
By \eqref{eq3lemonsenseBanach1},
\begin{equation}
\label{eq2th1inl1}
\|(U^tF)^{(2)}\| \le (\|F^{(2)}\| +t\|F^{(1)}\| \, \|g^{(2)}\|) e^{3t\|g\|_{Lip}},
 \end{equation}
 because $\la(t)=\la$ (this is precisely the simplifying assumption that all transitions do not increase $L$).
 Hence by \eqref{eqgenbininterpreslim1},
 \[
 \|(L_{b,h}-\La_b)U^tF\| \le 8h\ka (L,\la) (\|F^{(2)}\| +t\|F^{(1)}\| \, \|g^{(2)}\|) e^{3t\|g\|_{Lip}},
 \]
 implying \eqref{eq1th1inl1}.
\end{proof}

Let us turn to the case of Lipschitz continuous coefficients.

\begin{theorem}
\label{th2inl1}
Assume the operators $L_{b,h}$ are $L$ non-increasing for
a Lyapunov function $L$ on $\Z$ such that $L(j)/j^{\ga} \to 1$ with $\ga>0$, as $j\to \infty$.
Let $g: \MC^+_{\le \la}(L) \to l^1$ belongs to $(\MC^+_{\le \la}(L))_{bLip}$.
Then the Markov chains $X_h(t,x(h))$ with $x(h)\in  \MC^+_{\le \la}(L)$
such that $x(h)\to x$, as $h\to 0$, converge in distribution to the deterministic
evolution $X(t,x)$ and moreover
\begin{equation}
\label{eq1th2inl1}
|\E F (X_h(t,x))-F(X(t,x))| \le  C \|F\|_{Lip} (t\|g\|_{bLip}+\ka(L,\la)) (ht)^{\ga/(1+2\ga)}
\end{equation}
with a constant $C$. In particular, if $L(j)=j$, as is often the case in applications,
\begin{equation}
\label{eq2th2inl1}
|\E F (X_h(t,x))-F(X(t,x))| \le  C \|F\|_{Lip} (t\|g\|_{bLip}+\ka(L,\la)) (ht)^{1/3}.
\end{equation}
\end{theorem}

\begin{proof}
The proof is similar to the proof of Theorem \ref{thconv1}. The new difficulty arises from the fact that
the number of pure states (which is the maximal $j$ such that the sequence $(\de_i^j)\in l^1$ belongs to
$\MC^+_{\le \la}(L)\cap h\Z^{\infty}_+$) increases with decrease of $h$, while in Theorem \ref{thconv1} the number
of pure states $d$ was fixed. Thus the smoothing will be combined with finite-dimensional approximations.

Let $P_j$ denotes the projection on the first $j$ coordinates in $\R^{\infty}$, that is $[P_j(x)]_k=x_k$
for $k\le j$ and $[P_j(x)]_k=0$ otherwise. For $x\in \MC_{\le \la}(L)$,
\[
\|P_j(x)-x\|_{l^1}\le \frac{\la}{L(j)}.
\]
For functions on $x\in \MC_{\le \la}(L)$ we shall use the smooth approximation $\Phi_{\de,j}(F)=\Phi_{\de}(F\circ P_j)$
with $\Phi_{\de}$ as in Theorem \ref{thconv1}. More precisely,
\[
 \Phi_{\de,j}[V](x)=\int_{R^j}\frac{1}{\de^j} \prod_{k=1}^j \chi \left(\frac{y_k}{\de}\right)V(x-y) \, \prod_{k=1}^j dy_k
 =  \int_{\R^j}\frac{1}{\de^j} \prod_{k=1}^j \chi \left(\frac{x_k-y_k}{\de}\right)V(y) \, \prod_{k=1}^j dy_k,
 \]
 where $\chi$ is a mollifier, a non-negative infinitely smooth even function on $\R$ with
 a compact support and $\int \chi (w) \, dw=1$. Thus
 \[
  \Phi_{\de,j}[V](x)=\int_{R^j}\frac{1}{\de^j} \prod_{k=1}^j \chi \left(\frac{y_k}{\de}\right)V(P_j(x-y)) \, \prod_{k=1}^j dy_k.
  \]
Estimates \eqref{eq5athconv1} and \eqref{eq7thconv1} of Theorem \ref{thconv1} remain the same:
\[
\|\Phi_{\de,j}[V]\|_{C^1} =|\Phi_{\de}[V]\|_{bLip}\le  \|V\|_{bLip},
\]
and
\begin{equation}
\label{eq4th2inl1}
\|\Phi_{\de,j}[V]^{(2)}\| \le \|V\|_{bLip} \frac{1}{\de} \int |\chi'(w) | \, dw
\end{equation}
for any $\de$, and \eqref{eq6thconv1} generalizes to
 \begin{equation}
\label{eq3th2inl1}
 \|\Phi_{\de,j}[V]-V\| \le C \|V\|_{Lip} (j\de + \la/L(j)).
 \end{equation}

 Consequently, arguing as in Theorem \ref{thconv1} we obtain (using similar notations)
 \[
\|U_{h,\de,j}^tf -U^t_hf\|\le  Ct \|Q\|_{bLip} \|f\|_{bLip} (\de j + \la/L(j)) \exp\{t\|g\|_{Lip}\},
\]
 \[
\|U_{\de,j}^tf -U^tf\|\le  Ct \|Q\|_{bLip} \|f\|_{bLip} (\de j + \la/L(j)) \exp\{t\|g\|_{Lip}\},
\]
\[
\|U_{h,\de,j}^tf -U^t_{\de,j}f\|\le 8 t h\ka (L,\la) (\|F^{(2)}\|+t\|F\|_{Lip}\|g^{(2)}\|) e^{3t\|g\|_{Lip}}.
\]
And therefore
\[
\|U_h^tF -U^tF\| \le 8th\ka \|F^{(2)}\| e^{3t\|g\|_{Lip}}+
\]
\[
 C t  \|F\|_{Lip} \|g\|_{bLip}\left[\frac{ht}{\de}\ka(L,\la) +j\de + \frac{\la}{L(j)}\right] e^{3t\|g\|_{Lip}}.
\]

To get rid of $F^{(2)}$ we also approximate $F$ by the function $\Phi_{\de,j}[F]$, so that
\[
\|U_h^tF -U^tF\| \le
 C  \|F\|_{Lip} (t\|g\|_{bLip}+\ka(L,\la)) \left[\frac{ht}{\de}\ka(L,\la) +j\de + \frac{\la}{L(j)}\right] e^{3t\|g\|_{Lip}}.
\]
Choosing
\[
j=(ht)^{-1/(1+2\ga)}, \de= (ht)^{(1+\ga)/(1+2\ga)},
\]
makes all three terms in the square bracket of the same order, yielding \eqref{eq1th2inl1}.
\end{proof}

Let us now write down the analog of the multi-step optimization result of Theorem
\ref{thMarkDeconEvolBack1}. Namely, assume the principal is updating her strategy
 in discrete times $\{k\tau\}$, $k=0,1, \cdots ..., n-1$, with some fixed $\tau>0$, $n\in \N$ aiming
 at finding a strategy $\pi$ maximizing the reward \eqref{eqMarkdecevolback1}, but now with
 $x_0=x(h)\in h\Z^{fin}\cap \MC^+_{\le \la}(L)$ and the transitions
$x_k=X_h(\tau,x_{k-1},b_{k-1})$ defined by the Markov chain generated by  \eqref{eqgenbininterpres}.
The Shapley operators $S{N}$ and $S$ are defined analogously to Section \ref{secmultistep1},
so that their iterations define the optimal payoffs for the random evolution with fixed $h$
and for the limiting deterministic evolution.

\begin{theorem}
\label{th3inl1}
Under the assumption of Theorem \ref{th2inl1} let $x_0=x(h)\in h\Z^{fin}\cap \MC^+_{\le la}(L)$
and $x(h)\to x$ in $l^1$.
Then, for any $\tau \in (0,1]$, $n\in N$ and $t=\tau n$,
\begin{equation}
\label{eq1th3inl1}
\|S^n[N]V(x(h))-S^nV(x)\|\le C(t,\|g\|_{bLip},\|V\|_{bLip}) (h^{\ga/(1+2\ga)} \tau^{-(1+\ga)/(1+2\ga)} +|x(h)-x|),
\end{equation}
with a constant $C$ depending on $(t,\|g\|_{bLip},\|V\|_{bLip})$.
The r.h.s. tends to zero if $h\to 0$ faster than $\tau^{1+1/\ga}$.
\end{theorem}

All other results of Chapter \ref{strategicprincipal} extend automatically (with appropriate modifications)
to the case of Markov chains generated by \eqref{eqgenbininterpres}.

\section{Evolutionary coalition building under pressure}
\label{seccoalbuildpres}

As a direct application of Theorems \ref{th2inl1}, \ref{th3inl1}, let us
discuss the model of evolutionary coalition building.
Namely, so far we talked about small players that occasionally and randomly
exchange information in small groups (mostly in randomly formed pairs) resulting in copying
the most successful strategy by the members of the group. Another natural reaction of the society
of small players to the pressure exerted by the principal can be executed by forming stable groups that
can confront this pressure in a more effective manner (but possibly imposing certain obligatory regulations
for the members of the group). Analysis of such possibility leads one naturally to models of mean-field-enhanced coagulation
processes under external pressure.
Coagulation-fragmentation processes are well studied in statistical physics, see e. g. \cite{Nor00}.
In particular, general mass-exchange processes, that in our social environment become general coalition forming processes
preserving the total number of participants, were analyzed in \cite{Ko04} and
\cite{Ko06} with their law of large number limits for discrete and general
state spaces. Here we add to this analysis a strategic framework for a major player fitting the model
to the more general framework of the previous section.
Instead of coagulation and fragmentation of particles we shall use here the terms merging and splitting or breakage
of the coalitions of agents.

For simplicity, we ignore here any other behavioral distinctions
(assuming no strategic space for an individual player)
concentrating only on the process of forming coalitions.
Thus the state space of the total multitude of
small players will be formed by the set $\Z^{fin}_+$ of sequences of integers
$n=(n_1, n_2, \cdots ...)$ with only finite number of non-vanishing ones,
with $n_k$ denoting the number of coalition of size $k$, the total number
of small players being $N=\sum_k kn_k$ and the total number of coalitions
(a single player is considered to represent a coalition of size $1$) being $\sum_k n_k$.
 Also for simplicity we reduce attention
to binary merging and breakage only, extension to arbitrary regrouping processes
from \cite{Ko04} (preserving the number of players) is more-or-less straightforward.

 As previously, we will look for the evolution of appropriately scaled states, namely the sequences
\[
x=(x_1,x_2, \cdots ...)=hn =h(n_1,n_2, \cdots ...)\in h\Z_+^{fin}
\]
with certain parameter $h>0$, which can be taken, for instance, as the inverse number to
the total number of coalitions $\sum_k n_k$ at the initial moment of observation.

If any randomly chosen pair of coalitions of sizes $j$ and $k$ can merge with the rates $C_{kj}(x,b)$,
which may depend on the whole composition $x$ and the control parameter $b$ of the major player,
and any randomly chosen coalition of size $j$ can split (break, fragment) into two groups of sizes $k<j$ and $j-k$
with rate $F_{jk}(x,b)$, the limiting deterministic evolution of the state is known to be described by the
system of the so-called Smoluchovski equations
\begin{equation}
\label{eqSmoleqcoagfrag}
\dot x_k=g_k(x)=\sum_{j<k} C_{j,k-j}(x,b) x_j x_{k-j}-2\sum_j C_{kj}(x,b) x_j x_k
+2\sum_{j>k} F_{jk}(x,b) x_j -\sum_{j<k} F_{kj}(x,b) x_k.
\end{equation}
In addition to the well known setting with constant $C_{jk}$ and $F_{jk}$ (see e. g. \cite{BaCa}) we added here the
mean field dependence of these coefficients (dependence on $x$) and the dependence on the control parameter $b$.

As one easily checks, equations \eqref{eqSmoleqcoagfrag} can be written in the equivalent weak form
\begin{equation}
\label{eqSmoleqcoagfragweak}
\frac{d}{dt} \sum_j \phi_jx_j=\sum_{j,k}(\phi_{j+k}-\phi_j-\phi_k) C_{jk}(x,b) x_j x_k
+\sum_j \sum_{k<j} (\phi_{j-k}+\phi_k-\phi_j) F_{jk}(x,b) x_j,
\end{equation}
which should hold for a suitable class of test functions $\phi$. For instance, under the assumption of bounded coefficients
(see \eqref{eqSmoleqcoagfragboundco} below), the class of test functions is the class of all functions from
$l^{\infty}=\{\phi: \sup_j |\phi_j|<\infty\}$.
This implies, in particular, that the corresponding semigroups \eqref{eqdefflowonkineq}
on the space of continuous functions,
that is $U^tG(x)=G(X(t,x))$, have the generator
\[
\La_b G(x)=\sum_k g_k(x) \frac{\pa G}{\pa x_k}(x)
=\sum_{j,k}\left(\frac{\pa G}{\pa x_{k+j}}-\frac{\pa G}{\pa x_j}-\frac{\pa G}{\pa x_k}\right) C_{jk}(x,b) x_j x_k
\]
\begin{equation}
\label{eqencoagfrag}
+\sum_j \sum_{k<j} \left(\frac{\pa G}{\pa x_{j-k}}-\frac{\pa G}{\pa x_j}+\frac{\pa G}{\pa x_k}\right) F_{jk}(x,b) x_j
\end{equation}
of type \eqref{eqgenbininterpreslim} with
\[
\al_{i(j_1j_2)}^1=F_{ij_1}(x)x_i \quad \text{for} \quad j_1+j_2=i;
\quad \al^2_{(i_1i_2)j}=C_{i_1i_2}(x)x_{i_1}x_{i_2} \quad \text{for} \quad j=i_1+i_2
\]
and all other coefficients vanishing. Thus the Smoluchovski equations are particular representatives
of the characteristics of the first order PDE \eqref{eqgenbininterpreslim} or \eqref{eqgenbininterpreslim2}.

Let $R_j(x,b)$ be the payoff for the member of a coalition of size $j$. In our strategic setting,
the rates $C_{jk}(x,b)$ and $F_{jk}(x,b)$ should depend on the differences of these rewards before and after merging or splitting.
For instance, the simplest choices can be
\begin{equation}
\label{eqcoagstrateg1}
C_{kj}(x,b)=a_{j+k,k}  (R_{k+j}- R_k)^+
+a_{j+k,j}(R_{k+j}- R_j)^+,
\end{equation}
with some constants $a_{lk}\ge 0$
reflecting the assumption that merging may occur whenever it is beneficial for all members
concerned but weighted according to the size of the coalitions involved. Similarly
\begin{equation}
\label{eqfragstrateg1}
F_{kj}(x,b)=\tilde a_{kj} (R_j-R_k)^++\tilde a_{k,k-j} (R_{k-j}-R_k)^+.
\end{equation}

A Markov approximation to dynamics \eqref{eqSmoleqcoagfrag} is constructed in the standard way,
which is analogous to the constructions of approximating Markov chains described in the previous section
(for coagulation - fragmentation processes this Markov approximation is often referred
 to as the Markus-Lushnikov process, see e.g. \cite{Nor00}), namely, by attaching exponential
clocks to any pair of coalitions that can merge with rates $C_{kj}$ and to any coalition
that can split with rates $F_{kj}$. This leads to the Markov chain $X_h(t,x,b)$ on $h\Z^{fin}_+$
describing the process of {\it coalition formation} with the generator
\[
L^{coal}_{b,h}G (x)=\frac{1}{h}\sum_{i,j} C_{ij}(x,b) x_i x_j [G(x-he_i-he_j+he_{i+j})-G(x)]
\]
\begin{equation}
\label{eqSmolgencoagfrag}
+\frac{1}{h}\sum_i \sum_{j<i} F_{ij}(x,b) x_i [G(x-he_i+he_j+he_{i-j})-G(x)],
\end{equation}
of type \eqref{eqgenbininterpres} with the same identification of coefficients as above.

We shall propose here only the simplest result in this direction assuming that
the intensities of individual transition are uniformly bounded and uniformly Lipschitz,
 that is
\begin{equation}
\label{eqSmoleqcoagfragboundco}
C=\sup_{j,k} C_{jk}(x,b) <\infty, \quad F =\sup_j \sum_{k<j} F_{kj}(x,b) <\infty,
\end{equation}
\begin{equation}
\label{eqSmoleqcoagfragboundco2}
\begin{aligned}
& C(1)=\sup_{b,j,k} \|C_{jk} (.,b)\|_{C_{bLip}(\MC^+_{\le \la}(L))} <\infty, \\
& F(1) =\sup_{b,j} \sum_{k<j} \|F_{kj}(.,b)\|_{C_{bLip}(\MC^+_{\le \la}(L))} <\infty,
\end{aligned}
\end{equation}
with the Lyapunov function $L(j)=j$.
Notice however that the overall intensities are still unbounded (quadratic),
so that we are still far beyond the assumptions of Chapter  \ref{chbestres}.

For the function $L(j)=j$ the Markov chains $X_h(t,x,b)$ do not increase $L$.
Moreover,  \eqref{eqSmoleqcoagfragboundco} implies
\begin{equation}
\label{eqSmoleqcoagfragboundco1}
\begin{aligned}
& \|g\| \le 3C\la^2+3F\la, \\
& \|g\|_{bLip} \le 6C\la+3F+3(C(1)\la+F(1))\la,
\end{aligned}
\end{equation}
where we use the convention of the previous section that all functional norms are understood as $\sup_b$
of the norms of functions on $\MC^+_{\le \la}(L)$ with some fixed $\la$.
Hence the following result holds.

\begin{theorem}
\label{thMarkDisccontEvolucoal}
For a model of strategically enhanced coalition building subject to
\eqref{eqSmoleqcoagfragboundco} and \eqref{eqSmoleqcoagfragboundco2}
the conditions of Theorems \ref{th2inl1}, \ref{th3inl1} are satisfied for operators
\eqref{eqencoagfrag} and \eqref{eqSmolgencoagfrag} that represent particular cases of operators
\eqref{eqgenbininterpreslim} and \eqref{eqgenbininterpres}. Hence the corresponding Markov chains
$X_h(t,x,b)$ converge to the deterministic limit governed by the equation $\dot x=f(x,b)$
and the corresponding multi-step Markov decision problem converge to its deterministic limit.
\end{theorem}

\section{Strategically enhanced preferential attachment on evolutionary background}
\label{secbirthdeath}

A natural and useful extension of the theory presented above can be obtained by the
inclusion in our pressure-resistance evolutionary-type game the well known model
of linear growth with preferential attachment (Yule, Simon and others,
see  \cite{SimRoy} for review) turning the latter into a strategically enhanced preferential attachment model
that includes evolutionary-type interactions between agents and a major player having tools
to control (interfere into) this interaction.

We shall work with the general framework of Theorem \ref{th2inl1}, having in mind that the basic
examples of the approximating Markov chains $X_h(t,x(h))$ can arise from the merging and splitting coalition model
of the previous section (with generator \eqref{eqSmolgencoagfrag})
or from mean-field interacting particle systems of Chapters \ref{chapintkolMaCor} and \ref{chbestres} (for instance,
\eqref{eqQmapresres}), where the number of possible states becomes infinite, which can be looked at as
describing  the evolution of coalitions resulting from individual migrations from one coalition to another.
Denoting by $j$ the size of coalitions (rather than the type of an agent, as in \eqref{eqQmapresres}),
the generator of this Markov chain becomes an infinite-dimensional version of \eqref{eqdefgenmeanfieldchainnorm1aa}:
\begin{equation}
\label{eqdefgenmeanfpool2modrep}
L^{migr}_{b,h}G (x)=\frac{1}{h}\sum_{i,j} \ka x_i x_j
[R_j(x,b)-R_i(x,b)]^+[G\left(x-h e_i+h e_j\right)-G(x)],
\end{equation}
where $R_j(x,b)$ is the payoff to a member of a coalition of size $j=1,2, \cdots$. The Markov chain with generator
\eqref{eqdefgenmeanfpool2modrep} describes the process where agents  can move from one coalition to another
choosing the size of the coalition that is more beneficial under the control $b$ of the principal.

The most studied form of preferential attachment evolves by the discrete time injections
of agents (see \cite{BaAl99}, \cite{DeMo13}, \cite{SimRoy} and references therein).
Along these lines, we can assume that with time intervals $\tau$ a new agent enters the system
in such a way that with some probability $p(x,b)$ (which, unlike the standard model, can now depend
on the distribution $x$ and the control parameter $b$ of the principal) she does not enter
any of the existing coalitions (thus forming a new coalition of size $1$), and with probability $1-p(x,b)$
she joins one of the coalitions, the probability to join a coalition being proportional to its size
(this reflects the notion of preferential attachment coined in \cite{BaAl99}).

A continuous time version of these evolutions can be modeled by a Markov process, where the injection occurs with
some intensity $\al (x,b)$ (that can be influenced by the principal subject to certain costs).
In other words, it can be included by adding to generator
\eqref{eqdefgenmeanfpool2modrep} or \eqref{eqSmolgencoagfrag}
the additional term of the type
\[
L_{b,h}^{att}G(x)=\frac{p(x,b) \al (x,b)}{h}[G(x+he_1)-G(x)]
\]
\[
+ \frac{(1-p(x,b))\al (x,b)}{h}\sum_{k=1}^{\infty} k x_k[G(x-h e_k+h e_{k+1})-G(x)],
\]
or denoting $a(x,b)=p(x,b)\al(x,b)$, $P(x,b)=(1-p(x,b))\al (x,b)$,
\begin{equation}
\label{eqgenatt}
L_{b,h}^{att}G(x)=\frac{a(x,b)}{h}[G(x+he_1)-G(x)]
+ \frac{P(x,b)}{h}\sum_{k=1}^{\infty} k x_k[G(x-h e_k+h e_{k+1})-G(x)].
\end{equation}
Division by $h=1/N_0$ in this equation makes the rate of growth of the number of agents
comparable to the total number of agents. If it were of smaller scale, it would not influence
significantly the limiting LLN evolution.

The limiting generator obtained form \eqref{eqgenatt} as $h\to 0$ is the first order operator
\begin{equation}
\label{eqprefattgen}
\La_{b,h}^{att}G(x)=a(x,b)\frac{\pa G}{\pa x_1}
+ P(x,b)\sum_{k=1}^{\infty} k x_k\left[\frac{\pa G}{\pa x_{k+1}}-\frac{\pa G}{\pa x_k}\right],
\end{equation}
with the characteristics or the kinetic equations
\begin{equation}
\label{eqprefattgenchar}
\dot x_j=\de^j_1 a(x,b)+ P(x,b)[(j-1) x_{j-1}-jx_j].
\end{equation}

Combining this with the processes given by \eqref{eqdefgenmeanfpool2modrep} or \eqref{eqSmolgencoagfrag}
yields the model of {\it coalition building with the influx of new agents} or the
{\it strategically enhanced preferential attachment model on the evolutionary background}.
To shorten formulas let us ignore the contribution of individual migration \eqref{eqdefgenmeanfpool2modrep}
and concentrate on the combination of \eqref{eqSmolgencoagfrag} and \eqref{eqprefattgen}, that is, on the chain
$X_h(t,x(h))$ generated by the operator
\begin{equation}
\label{eqprefattmetrgsplit}
L_{b,h} =L_{b,h}^{att}+L_{b,h}^{coal},
\end{equation}
with the limiting generator  $\La_{b,h} =\La_{b,h}^{att}+\La_{b,h}^{coal}$.
The corresponding characteristics represent the controlled
(via discrete or continuous-time choice of parameter $b$ by the principal)
 infinite-dimensional ODEs combining \eqref{eqprefattgenchar} and
\eqref{eqSmoleqcoagfrag}:
\[
\dot x_k=g_k(x)=\de^j_1 a(x,b)+ P(x,b)[(j-1) x_{j-1}-jx_j]
\]
\begin{equation}
\label{eqprefattmetrgsplitchar}
+\sum_{j<k} C_{j,k-j}(x,b) x_j x_{k-j}-2\sum_j C_{kj}(x,b) x_j x_k
+2\sum_{j>k} F_{jk}(x,b) x_j -\sum_{j<k} F_{kj}(x,b) x_k.
\end{equation}

The proof that this system does in fact represent the dynamic LLN for the Markov chains defined by the generator
\eqref{eqprefattmetrgsplit} (i.e. the proof of the convergence of these Markov chains as $h\to 0$) requires advanced tools from
the theory of infinite-dimensional ODEs. It will be discussed in the next Section under very general assumption of unbounded rates,
namely under the assumption of the so called {\it additive bounds for rates}:
\begin{equation}
\label{eqaddboundrates}
a(x,b)+P(x,b) \le c,
\quad C_{kj}(x,b)\le c(k+j),
\quad \sum_{j<k} F_{kj}(x,b)\le cj
\end{equation}
with some constant $c$.

An important issue which we will not touch upon here is to understand the
controllability of the limiting (now in the sense $t\to \infty$) stationary solutions,
which may lead to the possibility for developing tools for influencing the power tails of distributions (Zipf's law)
appearing in many situations of practical interest, as well as the proliferation or extinction
of certain desirable (or undesirable) characteristics of the processes of evolution.

\section{Unbounded coefficients and growing populations}
\label{secunboundgrow}

This section goes much further in mathematical sophistication. We shall exploit the theory of the regularity and
sensitivity of ODEs in $l_1$ with unbounded coefficients. The key ingredient in our argument is
the following well-posedness and sensitivity result for kinetic equation \eqref{eqprefattmetrgsplitchar}.

\begin{theorem}
\label{thregsensODEl1addb}
For some $b$ let $a(x,b), P(x,b), C(x,b), F(x,b)$ be twice continuously differentiable in $x$
and satisfy the additive bound condition \eqref{eqaddboundrates} together with their derivatives,
that is, \eqref{eqaddboundrates} and the estimates
\begin{equation}
\label{eqaddboundrates1}
|\frac{\pa a(x,b)}{\pa x_p}|+|\frac{\pa P(x,b)}{\pa x_p}| \le c,
\quad  |\frac{\pa C_{kj}(x,b)}{\pa x_p}|\le c(k+j),
\quad  \sum_{j<k} |\frac{\pa F_{kj}(x,b)}{\pa x_p}|\le cj,
\end{equation}
\begin{equation}
\label{eqaddboundrates2}
|\frac{\pa a^2(x,b)}{\pa x_p \pa x_q}|+|\frac{\pa^2 P(x,b)}{\pa x_p \pa x_q}| \le c,
\quad  |\frac{\pa^2 C_{kj}(x,b)}{\pa x_p \pa x_q}|\le c(k+j),
\quad  \sum_{j<k} |\frac{\pa^2 F_{kj}(x,b)}{\pa x_p \pa x_q}|\le cj,
\end{equation}
uniformly for all $x,p,q$.

Then for any $x\in \MC^+(L^3)$ with the function $L_j=j$ there exists a unique global solution $X(t,x)$  of equation
\eqref{eqprefattmetrgsplitchar}, which also belong to $\MC^+(L^3)$ for all times. Moreover, the derivatives with respect
to the initial data
\[
\xi^p(t,x)=\frac{\pa X(t,x)}{\pa x_p},
\quad
\eta^{p,q}(t,x)=\frac{\pa^2 X(t,x)}{\pa x_p \pa x_q}
\]
are well defined for all times and belong to
the spaces   $\MC^+(L^2)$ and  $\MC^+(L)$ respectively, where they are uniformly bounded on bounded times and bounded $x$:
\begin{equation}
\label{eq1thregsensODEl1addb}
\sup_x \sum_j j |\xi^p_j (t,x)|\le C(\la, T)L_p,
\quad
\sup_x  \sum_j j^2 |\xi^p_j (t,x)|\le C(\la, T)L_p^2,
\end{equation}
\begin{equation}
\label{eq2thregsensODEl1addb}
\sup_x \sum_j j |\eta^{p,q}_j (t,x)|\le tC(\la, T)L_pL_q,
\end{equation}
uniformly for $t\le T$, where $\sup_x$ is over $x\in \MC_{\le \la}(L^3)$, that is, over $x$ satisfying the estimates
$\sum_j j^3 x_j \le \la$.
\end{theorem}

This theorem is a slight extension of the general sensitivity result from \cite{Ko10}.
Its full and detailed proof is given in Chapter 3 of \cite{Konewbook}.

We can now follow the line of arguments from Theorems \ref{thsiplestLLN} or \ref{th1inl1} to prove
the convergence of the Markov chains $X_h(t,x(h))$ generated by \eqref{eqprefattmetrgsplit}, that is, to obtain
the following main result of this Chapter. Let $U^t_h$ be the semigroup of the Markov chain $X_h(t,x(h))$,
and let $U^t$ be the semigroup generated by equation \eqref{eqprefattmetrgsplitchar},
that is $U^tG(x)=G(X(t,x))$.

\begin{theorem}
\label{thLLNprefatt}
Under the assumption of Theorem \ref{thregsensODEl1addb}
let the initial states $x(h)$ of the Markov chains $X_h(t,x(h))$ generated by
\eqref{eqprefattmetrgsplit} belong to $\MC^+_{\le \nu}(L^3)$ with some $\nu$ and converge
to a state $x$ so that $\sum_j j|x_j(h)-x_j| \to 0$, as $h\to 0$.
Then the Markov chains $X_h(t,x(h))$ converge in distribution to the deterministic
evolution $X(t,x)$ and moreover, for any $G\in C^2(l_1)$ and $T>0$,
\begin{equation}
\label{eq1thLLNprefatt}
\sup_{t\le T}\|U^t_hG(x(h))-U^tG(x)\|\le C_T(th+\sum_j j|x_j(h)-x_j|)\|G\|_{C^2(l_1)},
\end{equation}
\end{theorem}

\begin{proof}
Recall that
\[
\|G\|_{C^2(l_1)}=\sup|G(x)|+\sup_{x,j} \left|\frac{\pa G}{\pa x_j}\right|
+\sup_{x,j,k} \left|\frac{\pa ^2G}{\pa x_j \pa x_k}\right|.
\]

Following the argument of Theorems \ref{thsiplestLLN} or \ref{th1inl1},
we have to obtain a bound for the expression $|(L_{b,h}-\La_b)U^tG(x)|$,
which is uniform for $x\in \MC^+_{\le \nu}(L^3)$.
Looking at the expressions for $L_{b,h}$ and $\La_b$, we find that
 \[
 \|(L_{b,h}-\La_b)U^tG(x)\| =h a(x) \frac{\pa^2 \phi}{\pa x_1^2}+\frac12 hP(x)\sum_k k x_k
 \left(\frac{\pa^2 \phi}{\pa x_k^2}+\frac{\pa^2 \phi}{\pa x_{k+1}^2}-2\frac{\pa^2 \phi}{\pa x_k \pa x_{k+1}}\right)
 \]
\[
+ \frac12 h \sum_{j,k} \left( \frac{\pa^2 \phi}{\pa x_{k+j}^2}+\frac{\pa^2 \phi}{\pa x_j^2}+\frac{\pa^2 \phi}{\pa x_k^2}
+2\frac{\pa^2 \phi}{\pa x_k \pa x_j}-2\frac{\pa^2 \phi}{\pa x_k \pa x_{k+j}}-2\frac{\pa^2 \phi}{\pa x_j \pa x_{k+j}}\right)
\]
\[
+ \frac12 h \sum_{j,k} \left( \frac{\pa^2 \phi}{\pa x_{j-k}^2}+\frac{\pa^2 \phi}{\pa x_j^2}+\frac{\pa^2 \phi}{\pa x_k^2}
-2\frac{\pa^2 \phi}{\pa x_k \pa x_j}+2\frac{\pa^2 \phi}{\pa x_k \pa x_{j-k}}-2\frac{\pa^2 \phi}{\pa x_j \pa x_{j-k}}\right),
\]
where the derivatives are taken at points in small neighborhoods of $x$ and where $\phi(x)=U^tG(x)=G(X(t,x))$, so that
\[
\frac{\pa \phi}{\pa x_j}=\sum_i\frac{\pa G}{\pa x_i}(y)|_{y=X(t,x)} \xi^j_i(t,x),
\]
\[
\frac{\pa ^2\phi}{\pa x_p\pa x_q}=\sum_i\frac{\pa G}{\pa x_i}(y)|_{y=X(t,x)} \eta^{p,q}_i(t,x)
+\sum_{i,m}\frac{\pa ^2G}{\pa x_i \pa x_m}(y)|_{y=X(t,x)} \xi^p_i(t,x)\xi^q_m(t,x).
\]

Using the estimate for the derivatives of $X(t,x)$ from Theorem \ref{thregsensODEl1addb}, all terms here are
estimated in magnitude straightforwardly by some constants. For instance, let us look at the second term.
It can be estimated as
\[
P(x)\sum_k k x_k
 \left|\frac{\pa^2 \phi}{\pa x_k^2}+\frac{\pa^2 \phi}{\pa x_{k+1}^2}-2\frac{\pa^2 \phi}{\pa x_k \pa x_{k+1}}\right|
 \]
 \[
 \le 2c \|G\|_{C^2(l_1)} \sum_k k x_k \sum_{i,j} \left(|\xi^k_i(t,x)\xi^k_j(t,x)|+ |\xi^{k+1}_i(t,x)\xi^{k+1}_j(t,x)|+ |\xi^k_i(t,x)\xi^{k+1}_j(t,x)|\right)
 \]
 \[
 +2c \|G\|_{C^2(l_1)} \sum_k k x_k \sum_i \left(|\eta^{k,k}_i(t,x)| + |\eta^{k+1,k+1}_i(t,x)|+ |\eta^{k,k+1}_i(t,x)|\right)
 \]
 \[
 \le C \|G\|_{C^2(l_1)} \sum_k kx_k (k^2+(k+1)^2+k(k+1))\le C \|G\|_{C^2(l_1)},
 \]
 as required with $C$ depending on $\sum_k k^3 x_k$.
 Estimating analogously other terms we get the required uniform bound for $|(L_{b,h}-\La_b)U^tG(x)|$
 and then complete the proof as in  Theorem \ref{th1inl1}.
\end{proof}

Analogous to Section \ref{secconvcount} and using the theory of stability of the kinetic equations
 we can generalize this result to the case of Lipschitz coefficients.

 \begin{theorem}
\label{thLLNprefattLip}
Let $P(x),a(x), Q_{kl}(x), P^{m,k}(x)$ be continuous non-negative functions on $l_1$ such that
\begin{equation}
\label{eq1thwelposmergesplitpref}
P(x)+a(x) \le  c, \quad
C_{kl}(x) \le  c(k+l), \quad  \sum_{j<k} F^{kj}(x) \le  ck,
\end{equation}
\begin{equation}
\label{eq2thwelposmergesplitpref}
|P(x)-P(y)|+|a(x)-a(y)| \le  c (L, |x-y|), \quad
|C_{kl}(x)-C_{kl}(y)| \le  c(k+l) (L, |x-y|),
\end{equation}
\begin{equation}
\label{eq3thwelposmergesplitpref}
\sum_{j< k} |F^{kj}(x)-F^{kj}(y)| \le  ck (L, |x-y|),
\end{equation}
for $L=(1,2, \cdots )$ and a constant $c$.
Let the initial states $x(h)$ of the Markov chains $X_h(t,x(h))$ generated by
\eqref{eqprefattmetrgsplit} belong to $\MC^+_{\le \nu}(L^3)$ with some $\nu$ and converge
to a state $x$ so that $\sum_j j|x_j(h)-x_j| \to 0$, as $h\to 0$.
Then the Markov chains $X_h(t,x(h))$ converge in distribution to the deterministic
evolution $X(t,x)$ and moreover, for any Lipschitz continuous $G$ on $l_1$ and any $T>0$,
\begin{equation}
\label{eq4thwelposmergesplitpref}
\sup_{t\le T}\|U^t_hG(x(h))-U^tG(x(h))\|\le C_T(th)^{2/5}\|G\|_{Lip}.
\end{equation}
\end{theorem}

\begin{proof} The difference with the proof of Theorem \ref{th2inl1} is that here it is  more convenient
to approximate the coefficients $P(x),a(x), Q_{kl}(x), P^{m,k}(x)$ by their finite-dimensional
approximations $P(\PC_n(x)),a(\PC_n(x)), Q_{kl}(\PC_n(x)), P^{m,k}(\PC_n(x))$, rather than the whole r.h.s.
Let us denote the corresponding solution of the kinetic equations $X^n(t,x)$ and the corresponding
approximating Markov chain by $X_h^n(t,x)$.
As proved in  Chapter 3 of \cite{Konewbook}, the evolutions $X^n(t,x)$ and $X(t,x)$ are well defined and the
deviation of the transition operators of the evolutions
$X^n(t,x)$ and $X(t,x)$ is of order $1/n^2$ for the evolutions with the initial conditions
from $\MC^+_{\le \nu}(L^3)$ with any given $\nu$. Similarly one shows the same deviation for the transition
operators of the Markov chains $X_h^n(t,x)$ and $X_h(t,x)$. Approximating now the coefficients
 $P(\PC_n(x)),a(\PC_n(x)), Q_{kl}(\PC_n(x)), P^{m,k}(\PC_n(x))$ by the smooth functions
 $P_{\de}(\PC_n(x)),a_{\de}(\PC_n(x)), Q^{\de}_{kl}(\PC_n(x)), P_{\de}^{m,k}(\PC_n(x))$ as in Theorem \ref{th2inl1}
 we get new approximations with the deviations from the transitions without $\de$ of order $\de$.
 The required second derivative needed to apply  Theorem \ref{thLLNprefatt} to these last approximations
 is of order $n/\de$ again as in Theorem \ref{th2inl1}. Thus the total error term is of order
 \[
 \frac{ht}{\de}+n\de +\frac{1}{n^2}.
 \]
 Choosing $n$ and $\de$ that make all terms of equal decay in $ht$ yields the required order $(th)^{2/5}$.
\end{proof}

So far in this Chapter we kept the parameter $b$ of the principal in all our results, but did not use it.
However, in this chapter we extended the convergence results of Chapter \ref{chbestres} devoted to finite-state models
to the case of countable state space. Once this is done the results on the forward looking principal of Chapter \ref{strategicprincipal},
where parameter $b$ becomes an important controlled variable,
can be extended more or less automatically to this new setting of countable spaces, like in Theorem \ref{th3inl1}
of Section \ref{secconvcount}.      



\chapter{Appendix: Notes and compliments}
\label{chapbibl}

As a general reference for the analysis of multi-agent systems we can refer to
\cite{Shoham} and\cite{Ren11}. The development of multi-agent games and multi-agent
systems can be considered as a part of big data science contributing to the
understanding of smart technologies, the systems of systems and the internet of things.

The literature on both our main methodology (dynamic law of large numbers) and the
concrete areas of applications (inspection, corruption, etc)
 is quite abundant and keeps growing rapidly. In this section we try to overview
 at least the main trends most closely related to our methods and objectives.

1. {\it Dynamic LLN for optimal control}.

The dynamic law of large numbers (LLN)
for interacting particles is a well developed field of research in statistical mechanics
and stochastic analysis, see e.g. \cite{Ko10} for a review, as well as in the kinetic
models see e.g. \cite{AlmKinMod} and references therein. For agents, that is, in the
framework of controlled processes and games, the LLN was strongly advanced in the setting
of evolutionary biology leading to the deep analysis of the replicator dynamics
of evolutionary games and its various versions and modifications (see e.g. \cite{HS} or
\cite{KolMal1} for a general picture). A random multi-agent approximations for a
given dynamics can be cooked up in various ways. For instance, for the classical
replicator dynamics a multi-agent approximations was analyzed via migrations between
the states preserving the total number of agents  (see e.g. \cite{BeWe})
or via the model arising from the standard biological interpretation of payoffs
as the reproduction rate (thus changing the number of agents, see the last chapter
in \cite{KolMal1}).

For the model, where the pairwise interaction is organized in
 discrete time so that at any moment a given fraction $\al(N)$ of a homogeneous population of $N$ species
 is randomly chosen and decomposed into matching pairs, which afterwards experience simultaneous
transformations into other pairs according to a given distribution, the convergence to
 a deterministic ODE is proved in \cite{Boy95}. Paper
\cite{CorSar00} extends this setting to include several types of species
and the possibility of different scaling that may lead, in the limit $N\to \infty$, not only to ODE,
but to a diffusion process. In \cite{Ko04a} the general class of stochastic dynamic law of large number
is obtained from binary or more general $k$th order interacting particle systems
(including jump-type and L\'evy processes as a noise). The study of \cite{BeWe}
 concentrates on various subtle estimates for the deviation
of the limiting deterministic evolution from the approximating Markov chain for the evolution
that allows a single player (at any random time) to change her strategy to the strategy of another
randomly chosen player. Namely, it is shown, using the methods developed in \cite{FreWe},
that the probability that the maximum deviation of the
trajectories of the Markov chain with $N$ players from the limiting evolution exceeds some $\ep$
is bounded by the exponent $2me^{-\ep^2cN}$, where $m$ is the number of states and $c$ some constant.

 A related trend of research analyzes various choices of Markov approximation to repeated
games and their consequences to the question of choosing a particular Nash equilibrium amongst
the usual multitude of them. Seminal contribution \cite{KaMaRo} distinguishes specifically the
myopic hypothesis, the mutation or experimentation hypothesis and the inertia hypothesis in
building a Markov dynamics of interaction. As shown in \cite{KaMaRo} (with similar result in
\cite{You93}), introducing mutation of strength $\la$ and then passing to the limit $\la \to 0$
allows one to choose a certain particular Nash equilibrium, called a long run equilibrium
(or statistically stable, in the terminology of \cite{FoYou}) that for some coordination games
 turns out to coincide with the risk-dominant (in the sense of \cite{HaSe}) equilibrium. Further
important contributions in this direction include \cite{Elli}, \cite{BinSam}, \cite{BinSamVa}
showing how different equilibria could be obtained by a proper fiddling with noise (for instance
local or uniform as in \cite{Elli}) and discussing the important practical question of 'how long'
is the 'long-run' (for a recent progress on this question see \cite{KrYi13}). In particular paper
\cite{BinSamVa} exploits the model of 'musical chairs', where the changes of strategies occurs
according to a model, which is  similar to the children game, where the players having no spare
 chairs should leave the game. This paper discusses in detail the crucial question of the effect
of applying the limits $t\to \infty$, $\tau \to 0$ (the limit from discrete to continuous
replicator dynamics), $N\to \infty$ and $\la \to 0$ in various order. Further development
of the idea of local interaction leads naturally to the analysis of the corresponding Markov
processes on large networks, see \cite{LoPi06} and references therein. Some recent general
results of the link between Markov approximation to the mean field (or fluid) limit can be
found in \cite{LeBo13} and \cite{BeLeBo08}. Though in many papers on Markov approximation, the
switching probabilities of a revising player depends on the current distribution of strategies
used (assuming implicitly that this distribution is observed by all players) there exist also
interesting results (initiated in \cite{San01}, see new developments in \cite{San12}) arising
from the assumption that the switching of a revising player is based on an observed sample of
given size of randomly chosen other players.
Let us also mention papers \cite{Ga17}, \cite{Ga17} devoted to refined mean-field approximations,
where next terms are derived by assuming higher regularity of the coefficients.

A slightly different (but still very close) trend of research represents
the analysis of general stochastic approximation in association with the
 so-called method of ordinary differential equations,
see e. g. \cite{RoSa13} and \cite{Be96} and references therein.

The convergence results for a centrally controlled Markov chains of large number of constituents
to the deterministic continuous-time dynamics given by ordinary differential equations (kinetic
equations), dealt with in detail in our Chapters \ref{chbestres}, \ref{strategicprincipal}, were
initiated seemingly in papers \cite{GaGaLe}  and \cite{Ko12} from two different approaches, on
the level of trajectories (in spirit of paper \cite{BeWe}) and for the averages via the semigroup
 methods, as we do it here, respectively. In paper \cite{CecKol} the theory of coalition building
 was developed via the method of  \cite{GaGaLe}. An extension of the method of \cite{Ko12}
 is given in \cite{Averb17}.

The analysis of kinetic equations with the fractional derivatives was initiated in \cite{KochKondr17}.
However, the setting (and the equations discussed) of this paper is different from the one of our Section
\ref{secfracLLN}.

Proving the convergence of the process of fluctuations of multi-agent or multi-particle systems
from their dynamic LLN, that is, the dynamic central limit theorem (CLT) is of course harder than
getting the LLN itself. In fact, proving such CLT for the model of coagulation (or merge and splitting)
of Chapter \ref{secmodgrowthpres} was specifically mentioned as an important open problem in the
influential review \cite{Al}, the solution being provided in  \cite{Ko08} based on the general
analytic methodology \cite{Ko06}.  The latter methodology can be also directly applied to finite
state models of this book.

2. {\it Mean-field games (MFGs)}.

Mean-field games, which are dealt with in the second part of the book, present
a quickly developing area of the game theory. Mean-field games were initiated
 by Lasry-Lions \cite{LL2006} and Huang-Malhame-Caines \cite{HCM3}, \cite{HCM07}.
Roughly speaking they are dealing with the situations of many players with
similar interests, when the influence of each player on the overall outcome
becomes negligible, as the number of players tends to infinity. As insightful
illustrative examples for the theory one can mention two  problems (suggested
by P.-L. Lions): 1) 'when does the meeting start?', where agents choose time to
plan to arrive to a meeting which should start when certain fixed amount of
participants arrive and where the costs for participants arise from being too
early and too late (a game of time), and 2) 'where did I put my towel on the
beach?', where agents try to choose a place on the beach that should be possibly
closer to some point of interest (restaurant or see shore),
but possibly not too close to other agents (a congestion game).

At present there exist already several excellent surveys and monographs on various directions
of the theory,  see \cite{Ba13}, \cite{BenFr}, \cite{GLL2010}, \cite{Gomsurv}, \cite{Cain14},
\cite{CarDelbook18}. For other key recent development let us mention papers \cite{CarLac15},
\cite{Averb15}, \cite{Car13}, \cite{CarD13}, \cite{BasRa14}, \cite{KTY14}, \cite{TemBas14},
and references therein. Let us mention specifically the development of the mean-field games
with a major player, see \cite{Hu10}, \cite{NoCa13}, \cite{WaZh13}, \cite{LL2006}, where
also the necessity to consider various classes of players is well recognized,
see also \cite{Bens} and \cite{CarDel}.

The papers \cite{CMS2010} and \cite{Gom13}, as well as \cite{Gueant11}, \cite{Gueant11},
initiated the study of finite-state-space mean-field games that are the objects of our
analysis in the second part of the book. Paper  \cite{CMS2010}  develops the theory of
discrete-time MFGs showing, in particular, the convergence of the solutions of
backward-forward MFG consistency problem to a stationary problem, thus initiating a very
fruitful discussion of the precise links between stationary and time-dependent solutions.
Paper \cite{Gom13} deals already with more conventional continuous-time models proving
results on the existence and uniqueness of the solutions and the convergence of Nash
equilibria of approximating $N$-player games to a solution of the MFG consistency problem.
Papers \cite{Gueant11} and \cite{Gueant15} are devoted specifically to state spaces being
graphs of nontrivial geometry. All these finite-state models (their formulations and results)
are based on rather restrictive structure of control and payoff (for instance, the transitions
rates are linear functions of control parameter, strong convexity assumptions for payoffs
are assumed), which also yield quite specific form of the master equation, which in some cases
(for the so-called potential MFGs) is even reduced to a hyperbolic system of quasi-linear PDEs.
For models with more general dependence of dynamics and payoffs on the control parameters, the
analysis of MFGs was initiated in paper \cite{BasRa17}, \cite{KolMa15}, \cite{KolMa17}
developed further in this book. Of course, the simplest MFGs arise from the two-state models,
for which case many things can be calculated explicitly, see e.g. papers \cite{Besan} and
\cite{Gom14a}, where these two-state games are applied to the analysis of socio-economic models
of paradigm shifts in scientific research and of consumer choices of products, whose price depends
on the consumption rates. A setting for stationary problems with a finite lifetime of agents
(and thus with the varying number of agents) was suggested and analyzed in \cite{Wie}.

The applications of mean-field games are developing rapidly. For instance, \cite{Djehiche17}
and \cite{AchdLaur} analyse the problem of evacuation in a crowded building or room (say,
in case of a fire), opinion dynamics in social networks is analyzed in \cite{BaTemBa16}.
Paper \cite{BagDa14} deals with the demand management of
electrical heating or cooling appliances providing desynchronization of agents using a
bang-bang switching control. Paper \cite{BaBa16} deals with the dynamic of characteristic
functions (values of coalitions) in the transferable utility coalition games. Optimal
stopping games are dealt with in \cite{Nutz18}.  Examples with more economics content
include the standard models of the exploitation of common resources and the formation of
 prices that influence the total sales (see e.g. \cite{GraBen18}  and \cite{CarDelbook18}
 and references therein). As an example with a biological flavor one can mention the problem
 of flocking of birds, see  \cite{NoCa13} and \cite{CarDelbook18}.

Some preliminary ideas pointing into the direction of the MFG structure
can be traced back to paper \cite{Jovan} and references therein.

3. {\it Corruption and inspection games}.

Analysis of the spread of corruption in bureaucracy is a well recognized area
of the application of game theory, which attracted attention of many researchers.
General surveys can be found in \cite{Aidt}, \cite{Jain}, \cite{LeTsi98}, see also
monographs \cite{Rose-Ackerman1} and \cite{Rose-Ackerman2}. In his
Prize lecture \cite{Hurw}, L. Hurwicz gives an introduction in laymen terms of
various problems arising in an attempt to find out 'who will guard the guardians?'
and which mechanisms can be exploited to enforce the legal behavior?

In a series of papers \cite{LaMoMaRa09},\cite{LaMoMaRa08} the authors analyze
the dynamic game, where entrepreneurs have to apply to a set of bureaucrats
(in a prescribed order) in order to obtain permission for their business projects.
For an approval the bureaucrats ask for bribes with their amounts being considered
as strategies of the bureaucrats. This model is often referred to as petty corruption,
as each bureaucrat is assumed to ask for a small bribe, so that the large bureaucratic
losses of entrepreneurs occur from a large number of bureaucrats. This is an extension
of the classical ultimatum game, because the game stops whenever an entrepreneur
declines to pay the required graft. The existence of an intermediary undertaking the
contacts with bureaucrats for a fee may essentially affect the outcomes of this game.

In the series of works \cite{Vasinbook}, \cite{VasinKarUr10}, \cite{Nik14} the
authors develop an hierarchical model of corruption, where the inspectors of each
level audit the inspectors of the previous level and report their finding to the
inspector of the next upper level. For a graft they may choose to make a falsified
report. The inspector of the highest level is assumed to be honest but very costly
for the government. The strategy of the government is in the optimal determination
of the audits on various levels with the objective to achieve the minimal level of
 corruption with minimal cost. Related hierarchical models are developed in
\cite{GorUgUsbook14}, \cite{GorUgUsbook15}, where the stress is on the interaction
 of the three types of players:  benevolent dictator (the government), corrupted
 bureaucrat and an agent (a producer of goods), and in particular, on the conditions
 allowing for the stable development of the economy

Paper \cite{Stark14} develops a simple model to get an insight into the problem of
when unifying efforts result in strengthening of corruption. In paper \cite{Mal14}
the model of a network corruption game is introduced and analyzed, with the dynamics
of corrupted services between the entrepreneurs and corrupted bureaucrats
 propagating via the chain of  intermediary. In \cite{NgenZac} the dichotomy between
public monitoring and governmental corruptive pressure on the growth of economy
was modeled.  In \cite{LeeSig} an evolutionary model of corruption is developed for
ecosystem management and bio-diversity conservation.

In \cite{KoleMal17} a model of corruption with regard to psychological mimicry in the
administrative apparatus with three forms of corruption is constructed. It is given in
terms of the system of four differential equations describing the number of different
groups. The equilibrium states that allow to specify the dominant form of corruption
and investigate its stability, depending on the parameters of the psychological mimicry
and the rigor of anti-corruption laws are found.
In \cite{Mal17} the corruption dynamics is analyzed by means of the lattice model
similar to the three-dimensional Ising model: agents placed at the nodes of the corrupt
network periodically choose to perform or not to perform the act of corruption.
In \cite{Mal17a} the transportation problem of multi-agent interaction between
different goods' transporters with a corruption component is introduced and studied.
A statistical procedure of anti-corruption control of economic activity is proposed in
\cite{PichMalaf16}. A model of optimal allocation of resources for anticorruption purposes
is developed in \cite{NevMal15}. Various approaches to modeling corruption are collected
 in monograph \cite{Mal16}.

The research on the political aspects of corruption develops around the Acton's
dictum that 'power corrupts', where the elections serve usually as a major tool
of public control, see \cite{GiovSeid14} and references therein.

Closely related are the so-called inspection games,  see surveys e. g. in
\cite{AvSZ2002}, \cite{KolMal1}, \cite{KatKolYan}. Our evolutionary approach
was initiated in \cite{KoPaYa}, similar ideas can be found in \cite{Gubar17}.
Inspection games model non-cooperative interactions between two strategic parties,
called inspector and inspectee. The inspector aims to verify that certain regulations,
imposed by the benevolent principal he/she is acting for, are not violated by the
inspectee. On the contrary, the inspectee has an incentive to disobey the established
 regulations, risking the enforcement of a punishment fine in the case of detection.
 The introduced punishment mechanism is a key element of inspection games, since
deterrence is generally the inspector's highest priority. Typically, the inspector
has limited means of inspection at his/her disposal, so that his/her detection
efficiency can only be partial.
One of the first models was a two-person zero-sum recursive inspection game
proposed by Dresher \cite{Dresher}, where it was supposed that a given number $n$ of periods
are available for an inspectee to commit, or not, a unique violation,
and a given number $m\le n$ of one-period lasting inspections available
for the inspector to investigate the inspectee's abidance by the rules, assuming
 that a violator can be detected only if he/she is caught (inspected) in the act.
 Important extensions were given in papers \cite{Diamond} and \cite{Hopfinger}.
 This work initiated the application of inspection games to arms control and
 disarmament, see \cite{ACKvSZ1996} and references therein. This basic model was generalised
 in \cite{Masch66} to a non-zero-sum game adopting the notion of inspector leadership and showing
(among others) that the inspector's option to pre-announce and commit to a mixed
inspection strategy actually increases his/her payoff.

In \cite{Thom76}a similar framework was applied to investigate the problem of a
patroller aiming to inhibit a smuggler's illegal activity. In their so-called
customs-smuggler game, customs patrol, using a speedboat, in order to detect a
smuggler's motorboat attempting to ship contraband through a strait. They introduced
 the possibility of more than one patrolling boats, namely the possibility of two or
 more inspectors, potentially not identical, and suggested the use of linear programming
  methods for the solution of those scenario. A closed-form solution for the case of
two patrolling boats and three patrolling boats were provided in \cite{Bast91} and
\cite{Gar94} respectively. This research initiated the flux of literature on the
so-called {\it patrolling games}, see \cite{Alp11} and \cite{Alp13} for further development.

In a series of papers Von Stengel (see \cite{VonSten14} and references therein)
Von Stengel introduced a third parameter in Dresher's game, allowing multiple
violations, but proving that the inspector's optimal strategy is independent of the
 maximum number of the inspectee's intended violations. He studied several variation,
(i) optimising the detection time of a unique violation that is detected at the following
inspection, given that inspection does not currently take place, (ii)  adding different
rewards for the inspectee's successfully committed violations, and (iii) extending
Maschler's inspector leadership version under the multiple intended violations assumption.
Papers \cite{Ferg98} and \cite{Sakaguchi} studied  a similar three-parameter, perfect-capture,
sequential game, where: (i) the inspectee has the option to "legally" violate at an additional
cost; (ii) a detected violation does not terminate the game; (iii) every non-inspected
violation is disclosed to the inspector at the following stage.

Non-zero-sum inspection became actively studied in the 1980s,
in the context of the nuclear non-proliferation treaty (NPT).
The prefect-capture assumption was partly abandoned, and errors
of Type 1 (false alarm) and Type 2 (undetected violation given
that inspection takes place) were introduced to formulate the
so-called imperfect inspection games, see e.g. \cite{Canty}
and references therein for the solution of imperfect, sequential
and non-sequential games, assuming that players ignore any
information they collect during their interaction, where an
illegal action must be detected within a critical timespan
before its effect is irreversible. Imperfect inspection and
timely detection in the context of environmental control were
 developed in \cite{RothenZa}.

Avenhaus and Kilgour \cite{AK2004} developed a non-zero-sum, imperfect (Type 2 error) inspection game,
where a single inspector can continuously distribute his/her effort-resources between two
non-interacting inspectees, exempted from the simplistic dilemma whether to inspect or not.
They related the inspector's detection efficiency with the inspection effort through a non-linear
 detection function and derived results for the inspector's optimum strategy subject to its
convexity. Paper \cite{Hohzaki} extended this model, considering a similar $N+1$ players inspection game,
where the single inspection authority not only intends to optimally distribute his/her effort
among $N$ inspectee countries, but also among several facilities within each inspectee country.
These and related models are now presented in detail in the fundamental monograph \cite{AvKi2018}.

The methodology of the dynamic LLN, which we develop in this book, can be applied in fact to
almost all these models  allowing one to deal effectively with the situations when many agents
take part and some bulk characteristics of the dynamics of the game are of interest.

4. {\it Security and bioterrorism}.

Game-theoretic papers dealing with counterterrorism modeling were briefly mentioned
in Section \ref{seccyber-secint}. Our approaches with the LLN makes the bridge between
these papers and another trend of research, where the development of the extremists
activity is modelled by some system of ODEs of low dimension, which present the variations
of the models of the propagation of infectious diseases in epidemiology. The standard
abbreviations for the population classes in the latter theory are $S$ for susceptible,
$L$ or $E$ for latent or exposed (infected but yet not infectious), $I$ for infectious
and $R$ for recovered individuals. Depending on which classes are taken into account
several standard models were developed like $SI$ epidemics, $SEI$ epidemics, $SIR$ epidemics
and $SEIR$ (or $SLIR$ in other notations) epidemics. For instance, the latter model studies
the process of the evolution of the number of classes $S,E,I,R$ under the intuitively very
meaningful system of $4$ ODEs
\[
\dot S=-\la SI, \quad \dot L=\la XI -\al L, \quad \dot I=\al L -\mu I, \quad \dot R =\mu I
\]
with some constants $\al,\mu,\la$, and the $SIR$ model is obtained by replacing the
middle two equations by the equation $\dot I=\la XI -\mu I$, see reviews e.g. in
\cite{Heth00} or \cite{Rass03} for various modifications.

Modifying this class of models to the propagation of extremism, the authors of \cite{Cast03}
group the population in 4 classes: fanatics $F$, semifanatics (not fully converted) $E$,
susceptible $S$ and raw material (indifferent) $G$. With some reasonable
assumption on the propagation of influences via binary interaction the authors build a system
of ODEs on four variables with a quadratic r.h.s. (depending on certain numeric coefficients)
and study its rest points. In \cite{Sant08}, devoted concretely to modeling ETA (Basque
nationalist organization) in Spain, the authors distinguish the following classes of
population: those against independence, $E$, those striving for independence but without violence,
$N$, those supporting the fight for independence with violence, $V$, and the rest, $A$. Again
a reasonable system of ODEs on 4 variables is built and its predictions are compared with real
figures in attempt to evaluate the key parameters. In \cite{Goy14} the authors group the
population in three classes: extremists $E$, susceptible $S$ and reserved (isolated in jails).
Moreover they add an additional variable $G$ measuring the level of efforts of the government $G$,
and build and study the system of ODEs on four variables (depending on certain numeric coefficients).
Paper \cite{Udwadia} builds a model with three variables, the number of terrorists $T$, the number
of susceptible (to terrorist ideas) $S$ and the number of nonsusceptible $N$,
where the efforts of society are distributed between two types
of interference: direct military or police intervention (the analog of the preemptive measures of
\cite{RosSand04} discussed in Section \ref{seccyber-secint}) decreasing $T$, and propaganda or concessions
decreasing $S$, with the usual transitions between the groups via binary interactions. The system
of ODEs suggested in  \cite{Udwadia} looks like
\[
 \dot T=aTS-bT^2 +cT, \quad \dot S=-aTS -eT^2S+fT+gS, \quad \dot N=eT^2S-f_1T+hN
 \]
with constants $a,b,c,e,f,g,f_1,h$. The less intuitive quadratic and third order terms
with $b$ and $e$ reflect two methods above of the influence of the society. Rest points
of this dynamics are classified depending on the 'control parameters' $b$ and $e$.
Paper \cite{Saper08} builds a model with terrorist groups of size $T_1$ and $T_2$ acting
in two neighboring country with counterterrorist measures in these country, $N_1, N_2$,
measured in some units, the corresponding system of ODEs being
\[
\dot T_1 =-a_1 N_1 T_1 -b_1 \dot T_1 +g_{12} T_2, \,\, \dot T_2 =-a_2 N_2 T_2 -b_2 \dot T_2 +g_{21} T_1,
\,\,  \dot N_1=-\ga_1 T_1, \,\,  \dot N_2=-\ga_2 T_2.
\]

We showed just some examples of models. Their variety manifests that mathematical modeling
of these very complex  processes is effectively in the initial state without any consensus
about the relevant parameters and the ways of their estimations.

 All these models can be enhanced by our methodology by linking their evolutions with more real multi-agent
models and by considering the efforts of the government as a strategic parameter (in spirit of Chapter
\ref{strategicprincipal}) or including some optimization objective for the individuals of each group (in spirit of
Chapters \ref{chapintkolMaCor} or \ref{chapthreestate}), leading to the estimates of the investments
needed to control the situation.

An extensive introduction to various approaches for analysing terrorism can be found in \cite{Endersbook}
including statistical parameter estimations, linear regression, utility minimization on $\R^2_+$ (for the
 choice between proactive and defensive measures) and several game-theoretic formulations, for example,
 as social dilemmas (static 2 player games) or as dynamic modelling in the spirit of the entry deterrence
 games. The latter consider the question of committing or not a terrorist attack under the same setting as
 entering or not the market occupied by a monopolist.

 Let us also mention paper \cite{WrzaczekKapFeich17} (and several references therein)
 that uses the methods of deterministic optimal control to analyse governmental investments in
 the detection and interdiction of terrorist's attacks. The main model of this paper looks
 at the evolution of the number of undetected $X$ and detected $Y$ plots of attacks subject to the ODEs
 \[
 \dot X= \al -\mu X -\de (f-Y)X, \quad \dot Y =\de (f-Y)X -\rho Y,
 \]
 where $\al$ and $f$ are the control parameters of terrorists and the government respectively.

An extensive empirical analysis of crime models from the point of view of statistical physics is
developed in \cite{OrsonPerc15} and \cite{Perc17}.

Similar models can be used for the analysis of the propagation of scientific ideas. Paper
\cite{BetCin} uses all basic versions of deterministic epidemics above ($SEI$, $SIR$, $SEIR$) to
assess the propagations of the method of Feynmann's diagrams through the community of physicists,
the method that had many influential opponents at its initiation, like Oppenheimer in USA and
Lev Landau in USSR. Yet another twist of the problem concerns the illegal consumption of energy
resources, see \cite{Alesk05} and references therein.

Recently more attention has been paid to the models of cyber-security, as was discussed in Section
\ref{seccyber-secint} with reference to \cite{BenKaHo} for a review.

5. {\it Optimal resource allocation}. Approaches to the optimal allocation of resources are numerous
and are applied in various frameworks. For instance, one can distinguish resource allocation in the
framework of reliability theory, where one looks for the optimal amount of redundant units to sustain
the work of complex systems without interruption (see e.g.  \cite{Ushak13}), the approaches based on
the mechanism design methodology (distinguishing market and non-market mechanisms for allocations,
see e.g. \cite{Cond13} and references therein) and on the Bayesian
statistical inference (see e.g. \cite{Zhang17} and references therein). An approach from the swarm
intelligence (see \cite{BonaSwarmbook}) models the allocation of tasks and resources by the analogy
with the procedures found in the insect communities (ants or bees), and the market-based control by the
analogy with the market bidding - clearing mechanisms (see \cite{ClearMBCbook} and \cite{Lynch02}).
Closely related are the questions of the exploitation of limited common resources, often formulated
 as the fish wars, (see e.g. \cite{Mazabook}, \cite{MazRetFish}, \cite{KorKor}), project management
problems (see e. g.  \cite{NovikovProj07}), distribution of investments between economic sectors
(see \cite{TarasMgta16}) and general control of crowd behavior and the information propagation
through the crowds (see \cite{BarKor10}, \cite{BreerNovMob17} and references therein). Our approach
to optimal allocation, taken from \cite{Ko17},  deals with the distribution of the efforts of the principal
for better management of mean-field interacting particle systems and is also close in spirit to \cite{Nowzari17}.

\backmatter

\addcontentsline{toc}{chapter}{Index}
\printindex


\addcontentsline{toc}{chapter}{Bibliography}

\begin{thebibliography}{200}

\bibitem{AcerOnKuram05} J. A. Acebr\'on, L. L. Bonilla, J. P'erez Vicente, F. Riort and R. Spiegler.
The Kuramoto Model: A Simple
Paradigm for Synchronization Phenomena. Rev. Mod. Phys. 77 (2005), 137-185.

\bibitem{Achd13} Y. Achdou, F. Camilli and I. Capuzzo-Dolcetta.
 Mean field games: convergence of a finite difference method.
 SIAM J. Numer. Anal. 51:5 (2013), 2585 -- 2612.

\bibitem{AchdLaur} Y. Achdou and M. Laurier. On the system of partial differential equations
arising in mean field type control. Discrete and Continuous Dynamical Systems, A 35 (2015), 3879-3900.

\bibitem{Aidt} T. S. Aidt. Economic Analysis of corruption: a survey.
The Economic Journal {\bf 113: 491} (2009), F632-F652.

\bibitem{Al} D.J. Aldous.
Deterministic and stochastic models for coalescence (aggregation and
coagulation): a review of the mean-field theory for probabilists.
Bernoulli {\bf 5:1} (1999), 3-48.

\bibitem{Alesk05} F. T. Aleskerov and B. M. Schit. Production, legal and illegal consumption of electric energy
(in Russian). Problems of Control (Problemy upravlenia) 2 (2005), 63 - 69.


\bibitem{AlfMal}
G.V. Alferov, O.A. Malafeyev and A. S. Maltseva. Programming the robot in tasks of inspection and interception.
Proc IEEE 2015 International Conference on mechanics. Seventh Polyakov's Reading, Petersburg, Russia (2015), pp. 1-3.

\bibitem{AlmKinMod}
J. Almquist, M. Cvijovic, V. Hatzimanikatis, J. Nielsen and M. Jirstrand.
Kinetic models in industrial biotechnology - improving cell factory performance.
Metabolic Engineering 24 (2014),  38 - 60.

\bibitem{Alp13}
 S. Alpern and Th. Lidbetter. Mining coal or finding terrorists: the expanding search paradigm.
 Oper. Res. {\bf 61:2} (2013), 265 - 279.

\bibitem{Alp11}
S. Alpern, A. Morton and K. Papadaki. Patrolling games. Oper. Res. {\bf 59:5} (2011), 1246 - 1257.

\bibitem{And} W.J. Anderson. Continuous -Time Markov Chains. Probability and its
Applications. Springer Series in Statistics. Springer 1991.

\bibitem{A2008} L. Andreozzi. Inspection games with long-run inspectors.
European Journal of Applied Mathematics, 21:4-5 (2010), 441 - 458.

\bibitem{ArSa05}
D. Arce and T. Sandler.
Counterterrorism: A Game-Theoretic Analysis.
Journal of Conflict Resolution {\bf 49} (2005), 183-200.

\bibitem{Av04} R. Avenhaus. Applications of inspection games. Math. Model. Anal. 9 (2004), 179 - 192.

\bibitem{AC2005} R. Avenhaus, M.J. Canty. Playing for time: a sequential inspection game.
European Journal of Operational Research, 167:2 (2005), 475-492.

\bibitem{ACKvSZ1996}  R. Avenhaus, M. D. Canty, D. M. Kilgour, B. von Stengel, S. Zamir (1996).
Inspection games in arms control. European Journal of Operational research,  {\bf 90:3}, 383-394.

\bibitem{AK2004} R. Avenhaus, D. Kilgour. Efficient distributions of arm-control inspection effort.
 Naval Research Logistics, 51:1(2004), 1-27.

\bibitem{AvSZ2002} R. Avenhaus, B. Von Stengel, S. Zamir (2002).
Inspection games. In: R. Aumann, S. Hart (Eds.)
Handbook of Game Theory with Economic Applications, Vol. 3 North-Holland, Amsterdam, 1947- 1987.

\bibitem{AvKi2018} R. Avenhaus and T. Krieger.
Inspection games over time - Fundamental models and approaches.
Monograph to be published in 2018.

\bibitem{Averb17}
Averboukh Yu. Extremal shift rule for continuous-time zero-sum Markov games.
Dynamical games and application 7:1 (2017) 1-20.

\bibitem{Averb15} Yu. Averboukh. A minimax approach to mean field games.
Mat. Sbornik 206:7 (2015), 3 - 32.

\bibitem{BagDa14}
F. Bagagiolo and D. Bauso.
Mean-field games and dynamic demand management in power grids. (English summary)
Dyn. Games Appl. 4:2 (2014), 155 - 176.

\bibitem{BaBa16}
D. Bauso and T. Basar.
Strategic thinking under social influence: scalability, stability and robustness of allocations.
Eur. J. Control 32 (2016), 1 - 15.

\bibitem{BaCa} J. M. Ball and J. Carr.
The Discrete Coagulation-Fragmentation Equations: Existence, Uniqueness, and Density Conservation.
Journ. Stat. Phys. {\bf 61: 1/2}, 1990.

\bibitem{BaCaPe86} J. M. Ball, J. Carr and O. Penrose.
The Becker-D\"oring cluster equations: basic properties and asymptotic behaviour of solutions.
Comm. Math. Phys. {\bf 104:4} (1986), 657 - 692.

\bibitem{BarKor10}
N. Barabanov, N. A. Korgin, D. A. Novikov and A. G. Chkhartishvili.
Dynamic Models of Informational Controlin Social Networks.
Automation and Remote Control 71:11 (2010), 2417-2426.

\bibitem{BaAl99}
A.-L. Barab\'asi and R. Albert.
Emergence of Scaling in Random Networks. Science {\bf 286} (1999), 509-512.


\bibitem{Ba13} M. Bardi, P. Caines and I. Capuzzo Dolcetta.
Preface: DGAA special issue on mean field games. Dyn. Games Appl. 3:4 (2013), 443 -- 445.

\bibitem{BasRa14}
R. Basna, A. Hilbert and V. Kolokoltsov. An epsilon-Nash equilibrium for non-linear Markov
 games of mean-field-type on finite spaces. Commun. Stoch. Anal. {\bf 8:4} (2014), 449 - 468.

\bibitem{BasRa17}
R. Basna, A. Hilbert and V. N. Kolokoltsov.
An approximate Nash equilibrium for pure jump Markov games of mean-field-type on continuous state space
Stochastics 89 (2017), no. 6-7, 967 - 993.
http://arxiv.org/abs/1605.05073, 2016

\bibitem{Bast91} V.J. Baston and F. A. Bostock.
Naval. Res. Logist. 38 (1991), 171-182.


\bibitem{BaTemBa16}
D. Bauso, H. Tembine, and T. Basar.
Opinion dynamics in social networks through mean-field games. SIAM J. Control Optim. 54:6 (2016), 3225 - 3257.

\bibitem{TomCab} H. Beecher Stowe.
Uncle Tom's cabin. Blackie, 1963.

\bibitem{Beck} G. S. Becker and G. J. Stigler. Law enforcement, Malfeasance, and Compensation of Enforces.
The Journal of Legal Studies {\bf 3:1} (1974), 1-18.

\bibitem{BeK3} V.P. Belavkin,  V. Kolokoltsov.
 On general kinetic equation for many particle systems with
interaction, fragmentation and coagulation.
Proc. Royal Soc. Lond. A
{\bf 459} (2003), 727-748.

\bibitem{BetCin}
L. M. A. Bettencourt, A. Cintron-Arias, , D. I. Kaiser
and C. Castillo-Chavez.
The power of a good idea: Quantitative modeling of the spread
of ideas from epidemiological models.
Physica A 364 (2006), 513 - 536.

\bibitem{BeWe} M. Benaim, J. Weibull. Deterministic approximation of stochastic evolution in games.
Econometrica 71:3 (2003), 873 - 903.

\bibitem{Be96} M. Benaim.
A dynamical system approach to stochastic approximations.
SIAM J. Control Optim. {\bf 34:2} (1996), 437 - 472.

\bibitem{BeLeBo08}
M. Benaim and J.-Y. Le Boudec,
A class of mean field interaction models for computer and
communication systems. Performance Evaluation {\bf 65}
(2008), 823 - 838.

\bibitem{BenFr}
A. Bensoussan, J. Frehse, P. Yam. Mean Field Games and Mean Field Type Control Theory,
Springer, 2013.


\bibitem{BenKaHo}
A Bensoussan, S. Hoe , M. Kantarcioglu, A Game-Theoretical Approach for Finding Optimal Strategies
in a Botnet Defense Model. Decision and Game Theory for Security First International Conference,
GameSec 2010, Berlin, Germany, November 22-23, 2010. Proceeding, T. Alpcan, L. Buttyan, and J. Baras (Eds.),
Vol. 6442 pp. 135-148.

\bibitem{Bens} A. Bensoussan, J. Frehse and P. Yam. The Master equation in mean field theory.
Journal de Math\'ematiques Pures et Appliqu\'ees,  (9) 103:6 (2015), 1441 - 1474.

\bibitem{Besan} D. Beancenot and H. Dogguy. Paradigm shift: a mean-field game approach.
Bulletin of Economic Research 67:3 (2015), 0307-3378.

\bibitem{BellNet17}
M. Bell, S. Perera, M. Piraveenan, M. Bliemer, T. Latty and Ch. Reid.
Network growth models: A behavioural basis for attachment proportional to fitness.
Scientific Reports 7 (2017), Article number 42431.

\bibitem{Besl93} T. Besley and J. McLaren.
Taxes and Bribery: The Role of Wage Incentives. The Economic Journal {\bf 103: 416} (1993), 119-141.

\bibitem{BinSam} K. Binmore and L. Samuelson. Muddling through: noisy equilibrium selection.
J. Econom. Theory {\bf 74:2} (1997), 235 - 265.

\bibitem{BinSamVa} K. Binmore, L. Samuelson and R. Vaughan.
Musical chairs: modeling noisy evolution. Games Econom. Behav. {\bf 11:1} (1995), 1 - 35.

\bibitem{BonaSwarmbook} E. Bonabeau, M. Dorigo and G. Theraulaz. Swarm Intelligence. From Natural to Artificial Systems.
Santa Fe Institute Studies in the Sciences of Complexity. Oxford University Press, 1999.

\bibitem{B2004} S. Bowles. Microeconomics. Behavior, Institutions and Evolution. Russell Sage Foundation, 2004.

\bibitem{Boy95} R. Boylan.
Continuous approximation of dynamical systems with randomly matched individuals.
J. Econom. Theory {\bf 66:2} (1995), 615 - 625.

\bibitem{BrKi13}
S. J. Brams and M. Kilgour. Kingmakers and Leaders in Coalition Formation. Social
Choice and Welfare {\bf 41:1} (2013), 1-18.

\bibitem{BrKi88}
S. J. Brams and M. Kilgour.
National Security Games. Synthese {\bf 76} (1988), 185-200.

\bibitem{BreerNovMob17} V. V. Breer, D. A. Novikov and A. D.  Rogatin.
Mob Control: Models of Threshold Collective Behavior. Heidelberg: Springer, 2017.
Transl. from the Russian Edition, Moscow, Lenand, 2016.

\bibitem{Cain14}
P. E. Caines, ``Mean Field Games'', {\it Encyclopedia of Systems and
Control},  Eds. T. Samad and J. Ballieul. Springer Reference 364780;
DOI 10.1007/978-1-4471-5102-9 30-1, Springer-Verlag, London, 2014.

\bibitem{Canty} M. J. Canty, D. Rothenstein and R. Avenhaus.
Timely inspection and deterrence. Eur. J. Oper. Res. 131 (2001),
208-223.

\bibitem{Car15} P. Cardaliaguet, F. Delarue, J-M. Lasry and P-L. Lions.
The master equation and the convergence problem in mean field
games.  https://arxiv.org/abs/1509.02505

\bibitem{Car13} P. Cardaliaguet, J-M. Lasry, P-L. Lions and A. Porretta.
Long time average of mean field games with a nonlocal coupling. SIAM J. Control Optim. 51:5 (2013),
3558 -- 3591.

\bibitem{CarD13} R. Carmona and F. Delarue.
 Probabilistic analysis of mean-field games. SIAM J. Control Optim. 514 (2013), 2705 -- 2734.

\bibitem{CarDel} R. Carmona and F. Delarue. The master equation for large population equilibriums.
Stochastic Analysis and Applications. Springer Proc. Math. Stat. 100, Springer, Cham, 2014, pp. 77-128.

\bibitem{CarDelbook18} R. Carmona and F. Delarue. Probabilistic Theory
of Mean Field Games with Applications, v. I, II.
Probability Theory and Stochastic Modelling v. 83, 84. Springer, 2018.

\bibitem{CarLac15} R. Carmona and D. Lacker. A probabilistic weak formulation
 of mean field games and applications. Ann. Appl. Probab. 25:3 (2015), 1189 -- 1231.

\bibitem{Cast03} C. Castillo-Chavez and B. Song. Models for the transmission dynamics
of fanatic behaviors. In: H.T. Banks and  C. Castillo-Chavez (Eds.) Bioterrorism: Mathematical
 Modelling Applications in Homeland Security. SIAM Frontiers in Applied Mathematucs, v. 28, SIAM.
 Philadelphia, 2003, 155-172.

\bibitem{CecKol}
A. Cecchin and V. N. Kolokoltsov. Evolutionary game of coalition building under external pressure
arXiv:1705.08160. Annals of ISDG (2017), 71-106.

\bibitem{Chatterjee93} K. Chatterjee, B. Dutta, D. Ray and K. Sengupta.
A noncooperative Theory of Coalition Bargaining. Review of Economic studies 60 (1993), 463-477.

\bibitem{ClYoGl07} A. Clauset, M. Young and K. S. Gleditsch.
On the Frequency of Severe Terrorist Events.
arXiv:physics/0606007v3
Journal of Conflict Resolution February 2007 {\bf 51} (2007), 58-87.

\bibitem{ClearMBCbook} S. H. Clearwater (Ed.) Market-Based Control:
A Paradigm for Distributed Resource Allocation. Singapore, World Scientific, 1995.

\bibitem{Cond13} D. Condorelli. Market and non-market mechanisms for the optimal allocation
 of scarce resources. Games and Economic Behavior 82 (2013), 582-591.

\bibitem{CorSar00}
V. Corradi, R. Sarin. Continuous approximations of stochastic evolutionary game dynamics.
 J. Econom. Theory {\bf 94:2} (2000), 163 - 191.

\bibitem{DeMo13}
St. Dereich and P. M\"orters.
Random networks with sublinear preferential attachent: the giant component.
The Annals of Probability {\bf 41:1} (2013), 329 - 384.

\bibitem{Deutsch13} Y. Deutsch, B. Golany, N. Goldberg and U. G. Rothblum.
Inspecion games with local and global allocation bounds.
Naval. Res. Logist. 60 (2013), 125-140.

\bibitem{Diamond}
H. Diamond. Minimax policies for unobservable inspections.
Mathematics of Operations Research, 7:1 (1982), 139 - 153.

\bibitem{Djehiche15}
B. Djehiche, H. Tembine and R. Tempone.
 A stochastic maximum principle for risk-sensitive mean-field type control.
 IEEE Trans. Automat. Control 60:10 (2015), 2640 - 2649.

 \bibitem{Djehiche17}
B. Djehiche,  A. Tcheukam and H. Tembine.
A mean-field game of evacuation in mutlilevel building.
IEEE Trans. Automat. Control 62:10 (2017), 5154 - 5169.

\bibitem{Dobr18} U. Dobramysl, M. Mobilia, M. Pleimling and U. T\"auber.
Stochastic population dynamics in spatially extended predator-prey systems.
J. Phys. A 51 (2018), no. 6, 063001,

\bibitem{DoSaSo}
R. Dorfman, P. Samuelson and R. Solow. Linear programming and economic analysis. McGraw-Hill,
New York, 1958.

\bibitem{Dresher}
M. Dresher, A sampling inspection problem in arms control agreements: a game theoretical analysis.
Memorandum No. RM-2872-ARPA (Rand Corporation), Santa Monica, 1962.

\bibitem{Duffie} D. Duffie, S. Malamud and G. Manso. Information percolation with equilibrium
search dynamics. Econometrica 77 (2009), 1513-1574.

\bibitem{ElKaKa} N. El Karoui, I. Karatzas. Dynamic allocation problems in continuous time.
Ann. Appl. Probab. {\bf 4:2} (1994), 255 - 286.

\bibitem{Elli} G. Ellison. Learning, local interaction, and coordination. Econometrica
{\bf 61:5} (1993), 1047 - 1071.

\bibitem{Endersbook} W. Enders and T. Sandler.
The Political Economy of Terrorism. Cambridge University Press, 2012.

\bibitem{FaAr12}
J. R. Faria and D. Arce.
A Vintage Model of Terrorist Organizations.
Journal of Conflict Resolution {\bf 56:4} (2012), 629-650.

\bibitem{Ferg98} T. S. Ferguson and C. Melolidakis. On the inspection game.
Naval. Res. Logist. 45 (1998), 327-334.

\bibitem{Fieldi}
H. Fielding. The history of the life of the late Mr Jonathan Wilde the Great.
H. Hamilton, 1947.

\bibitem{Filip60} A.F. Filippov. Differential equations with discontinuous right-hand side.
Mat. Sb. 51(93):1 (1960), 99-128.

\bibitem{Filip88}  A.F. Filippov. Differential Equations with Discontinuous Righthand Sides.
Kluwer Academic, Dordrecht, 1988.

\bibitem{FleSo}
W. H. Fleming, H. M. Soner.
Controlled Markov Processes and
Viscosity Solutions. Sec. Ed. Springer 2006.

\bibitem{FoYou} D. Foster and P. Young.
Stochastic evolutionary game dynamics. Theoret. Population Biol. {\bf 38:2} (1990), 219 - 232.

\bibitem{FreWe}
M. Freidlin and A. Wentzell. Random Perturbations of Dynamic Systems. Springer, Berlin, 1984.

\bibitem{Gar94} A. Y. Garnaev. A remark on the customs and smuggler game.
Naval. Res. Logist. 41 (1994), 287-293.

\bibitem{GaGaLe} N. Gast, B. Gaujal and J.-Y. Le Boudec. Mean Field for Markov Decision Processes:
 From Discrete to Continuous Optimization. IEEE Trans. Automat. Control {\bf 57:9} (2012), 2266-2280.

\bibitem{Ga17} N. Gast and B. Van Houdt. A Refined Mean Field Approximation. Proceedings of the ACM
on Measurement and Analysis of Computing Systems , ACM, 2017, 1 (28), 10.1145/3152542.
https://hal.inria.fr/hal-01622054


\bibitem{Ga18}
N. Gast, L. Bortolussi and M. Tribastone. Size Expansions of Mean Field Approximation:
Transient and Steady-State Analysis. To appear in Performance Evaluation, 2018.

\bibitem{GinBandit89} J. C. Gittins. Multi-Armed Bandit and Allocation Indices. Wiley, 1989.

\bibitem{GiovSeid14} F. Giovannoni and D. J. Seidmann. Corruption and power in democracies.
Soc Choice Welf {\bf 42} (2014), 707-734.

\bibitem{CMS2010} D. A. Gomes, J. Mohr, R. Souza. Discrete time, finite state space mean field
games. J. Math. Pures Appl. (9) {\bf 93:3} (2010), 308 - 328.

\bibitem{Gom13} D. A. Gomes, J. Mohr and R. R. Souza.
Continuous time finite state space mean field games. Appl. Math. Optim. 68:1 (2013), 99-143.

\bibitem{Gom14} D. A. Gomes, S. Patrizi and V. Voskanyan.
 On the existence of classical solutions for stationary extended mean field games. Nonlinear Anal. 99 (2014), 49 -- 79.

\bibitem{Gom14a} D. Gomes, R. M. Velho and M-T. Wolfram.
 Socio-economic applications of finite state mean field games.
 Philos. Trans. R. Soc. Lond. Ser. A Math. Phys. Eng. Sci. 372 (2014), no. 2028, 20130405.

\bibitem{Gomsurv} D. A. Gomes and J. Saude. Mean field games models -- a brief survey. Dyn. Games Appl. 4:2 (2014), 110 -- 154.

\bibitem{GoShah11} A. Gorban and M. Shahzad. The Michaellis-Menten-Stueckelberg theorem.
Entropy {\bf 13: 5} (2011), 966-1019.

\bibitem{GoKo14} A. Gorban and V. Kolokoltsov.
Generalized Mass Action Law and Thermodynamics for Generalized Nonlinear Markov Processes.
To appear in: The Mathematical Modelling of Natural Phenomena (MMNP), 2015.

\bibitem{GorUgUsbook14} O.I. Gorbaneva, G. A. Ugolnitskii and A. B. Usov. Modeling corruption
 in hierarchical systems of control (in Russian). Monograph. Rostov-na-Donu, South Federal University, 2014.

 \bibitem{GorUgUsbook15} O.I. Gorbaneva, G. A. Ugolnitskii and A. B. Usov.
  Models of corruption
 in hierarchical systems of control (in Russian).
 Control Sciences 1 (2014), 2-10.

\bibitem{Goy14} A. Goyal, J. B. Shukla, A. K. Misra and A. Shukla. Modeling the role of
government efforts in controlling extremism in a society. Math. Meth. Appl. Sci. 38 (2015), 4300-4316.

\bibitem{Gueant11} O. Gu\'eant. From infinty to one: THe reduction of some mran feld hames to a global control problem.
https://arxiv.org/abs/1110.3441

\bibitem{Gubar17} E. Gubar, S. Kumacheva, E. Zhitkova, Z. Kurnosykh, T. Skovorodina.
Modelling of Information Spreading in the Population of Taxpayers: Evolutionary Approach.
Contributions to Game Theory and Management 10 (2017), 100-128.

\bibitem{GraBen18}
P. J. Graber and A. Bensoussan.
Existence and uniqueness of solutions for Bertrand and Cournot mean field games.
Appl. Math. Optim. 77:1 (2018), 47 - 71.

\bibitem{Gueant15} O. Gu\'eant. Existence and Uniqueness Result for Mean Field Games with Congestion Effect on Graphs.
Appl. Math Optim 72 (2015), 291-303.

\bibitem{GLL2010} O. Gu\'eant, J-M. Lasry and P-L. Lions. Mean Field Games and Applications.
Paris-Princeton Lectures on Mathematical Finance 2010. Lecture Notes in Math. 2003, Springer, Berlin, p. 205-266.

\bibitem{GuLeSchaWa}
R. Gunther,  L. Levitin,  B. Schapiro and P. Wagner.
Zipf's Law and the Effect of Ranking on Probability Distributions.
International Journal of Theoretical Physics {\bf 35:2} (1996), 395-417.

\bibitem{HaSe} J. Harsanyi and R. Selten. A General Theory of Equilibrium Selection in Games.
Cambridge MA, MIT Press, 1988.

\bibitem{HeLa} O. Hernandez-Lerma and J. B. Lasserre. Discrete-Time Markov Control Processes. Springer, New York, 1996.

\bibitem{Her2}
O. Hernandez-Lerma. Lectures on continuous-time Markov control
processes. Aportaciones Matematicas: Textos, 3. Sociedad Matematica
Mexicana, Mexico.

\bibitem{Heth00} H. W. Hethcote.
The mathematics of infectious diseases.
SIAM Rev. 42:4 (2000), 599 - 653.

\bibitem{HS} J. Hofbauer, K. Sigmund.
Evolutionary Games and Population Dynamics. Cambridge University
Press, 1998.

\bibitem{Hohzaki} R. Hohzaki. An inspection game with multiple inspectees.
Eur. J. Oper. Res. 178 (2007), 894 - 906.

\bibitem{Hopfinger}
E. H\"opfinger. A game-theoretic analysis of an inspection problem.
Technical Report, C-Notiz No. 53 (preprint), University of Karlsruhe, 1971.


\bibitem{HCM3} M. Huang, R. Malham\'e, P. Caines. Large population stochastic dynamic games:
closed-loop Mckean-Vlasov systems and the Nash certainty equivalence principle.
Communications in information and systems 6 (2006), 221 -- 252.

\bibitem{HCM07} M. Huang, P. Caines and R. Malham\'e. Large-Population Cost-Coupled LQG Problems
With Nonuniform Agents: Individual-Mass Behavior and Decentralized $\epsilon$-Nash Equilibria.
IEEE Trans Automat Control 52:9 (2007), 1560 -- 1571.

\bibitem{Hu10}  M. Huang. Large-population LQG games involving a major player: the Nash certainty equivalence principle.
SIAM J Control Optim 48 (2010), 3318 -- 3353.

\bibitem{Hurw} L. Hurwicz. But who will guard the guardians? Prize Lecture 2007.
Available at www.nobelprize.org

\bibitem{Jain} A. K. Jain. Corruption: a review. Journal of Economic Surveys {\bf 15: 1} (2001), 71 - 121.

\bibitem{Jovan} B. Jovanovic and R. W. Rosental.
Anonymous sequential games.
Journal of Mathematical Economics 17 (1988) 77-87.


\bibitem{KaMaRo} M. Kandori, G. J.  Mailath and R. Rob.
Learning, mutation, and long run equilibria in games.
Econometrica {\bf 61:1} (1993), 29 - 56.

\bibitem{KatKolYan}
S. Katsikas, V. Kolokoltsov and W. Yang. Evolutionary Inspection and Corruption Game.
Games 7:4 (2016), 31 (open access).

\bibitem{KatKol}
S. Katsikas and V. Kolokoltsov.
Evolutionary, Mean-Field and Pressure-Resistance
Game Modelling of Networks Security. Submitted.

\bibitem{KochKondr17}
A. N. Kochubei and Yu. Kondratiev. Fractional kinetic hierarchies and intermittency.
 Kinet. Relat. Models 10:3 (2017), 725 - 740.

\bibitem{KoleMal17}
I. Kolesin, O. A. Malafeyev, M. Andreeva and G. Ivanukovich.
 Corruption: Taking into account the psychological mimicry of officials, 2017, AIP Conference Proceedings
AIP Conference Proceedings 1863, 170014 (2017); https://doi.org/10.1063/1.4992359

\bibitem{Ko04}
V. N. Kolokoltsov. {\sl
Hydrodynamic limit of coagulation-fragmentation type models of
$k$-nary interacting particles.} Journal of Statistical Physics
{\bf 115}, 5/6 (2004), 1621-1653.

\bibitem{Ko04a}
 V. Kolokoltsov. {\sl
Measure-valued limits of interacting particle systems with
$k$-nary interaction II. Finite-dimensional limits.} Stochastics
and Stochastics Reports {\bf 76:1} (2004), 45-58.

 \bibitem{Ko06}
 V. N. Kolokoltsov. {\sl Kinetic
equations for the pure jump models of $k$-nary interacting
particle systems.} Markov Processes and Related Fields {\bf 12}
(2006), 95-138.

\bibitem{Ko07}  V. Kolokoltsov. Nonlinear
Markov Semigroups and Interacting L\'evy Type Processes. Journ.
Stat. Physics {\bf 126:3} (2007), 585-642.

\bibitem{Ko07a} V. N. Kolokoltsov. Generalized Continuous-Time Random Walks (CTRW),
Subordination by Hitting Times and Fractional Dynamics.
arXiv:0706.1928v1[math.PR] 2007. Probab. Theory and Applications {\bf  53:4} (2009).

\bibitem{Ko08} V. N. Kolokoltsov.
The central limit theorem for the Smoluchovski coagulation model.
arXiv:0708.0329v1[math.PR] 2007. Prob. Theory Relat. Fields {146: 1} (2010), Page 87. Published online
http://dx.doi.org/10.1007/s00440-008-0186-2

 \bibitem{Ko09}  V. N. Kolokoltsov. Noninear Markov games.
Talk given at Adversarial and Stochastic Elements in Autonomous Systems Workshop, 22-24 March 2009,
at FDIC, Arlington, VA 22226. Semantic Scholar.
https://pdfs.semanticscholar.org/3eab/2cf5bcb98548bd5a862932370a3bfc792fe7.pdf

\bibitem{Ko10}
V. N. Kolokoltsov. {\sl Nonlinear Markov processes and kinetic
equations}. Cambridge Tracks in Mathematics 182, Cambridge Univ. Press, 2010.

\bibitem{Ko11} V. N. Kolokoltsov. Markov processes, semigroups and generators.
DeGruyter Studies in Mathematics v. 38, DeGruyter, 2011.

\bibitem{Ko12}
 V. N. Kolokoltsov.
Nonlinear Markov games on a finite state space (mean-field and binary interactions).
International Journal of Statistics and Probability {\bf 1:1} (2012), 77-91.
http://www.ccsenet.org/journal/index.php/ijsp/article/view/16682

\bibitem{Ko15}
V. N.  Kolokoltsov.
On fully mixed and multidimensional extensions of the Caputo and Riemann-Liouville derivatives,
 related Markov processes and fractional differential equations. Fract. Calc. Appl. Anal. 18:4 (2015), 1039 - 1073.

\bibitem{Ko17} V. N. Kolokoltsov.
The evolutionary game of pressure (or interference), resistance and collaboration (2014).
http://arxiv.org/abs/1412.1269. MOR (Mathematics of Operations Research), 42 (2017), no. 4, 915 - 944.

\bibitem{Ko17a}
V. N. Kolokoltsov. Chronological operator-valued Feynman-Kac formulae for generalized
  fractional evolutions. arXiv:1705.08157

\bibitem{Konewbook} V.  Kolokoltsov. Differential Equations on Measures and Functional Spaces.
To appear in Birkhauser, 2019.

\bibitem{KoBens}
V. N. Kolokoltsov and A. Bensoussan.
Mean-field-game model of botnet defence in cyber-security (2015).
http://arxiv.org/abs/1511.06642
AMO (Applied Mathematics and Optimization) 74(3), 669-692,

\bibitem{KolMal1}
V. N. Kolokoltsov and O. A. Malafeyev. Understanding Game Theory. World Scientific, Singapore, 2010.

\bibitem{KolMa15} V. N. Kolokoltsov and O. A. Malafeyev.
Mean field game model of corruption (2015).
 Dynamics Games and Applications.
 7:1 (2017), 34-47.  Open Access.

\bibitem{KolMa17}
V. Kolokoltsov and O. Malafeyev.
Corruption and botnet defense: a mean field game approach
International Journal of Game Theory.
DOI 10.1007/s00182-018-0614-1
Online First:
http://link.springer.com/article/10.1007/s00182-018-0614-1

\bibitem{KolMas}
V. N. Kolokoltsov and V. P. Maslov. Idempotent Analysis an its Applications. Kluwer Academic, 1987.

\bibitem{KoPaYa}
V. N. Kolokoltsov, H. Passi, W. Yang.
Inspection and crime prevention: an evolutionary perspective (2013). http://arxiv.org/abs/1306.4219

\bibitem{KTY14}
V. Kolokoltsov, M. Troeva and W. Yang. On the rate of convergence
for the mean-field approximation of controlled diffusions with large number of players.
Dyn. Games Appl. 4:2 (2014), 208 -- 230.

\bibitem{KolVer14}
V. Kolokoltsov and M. Veretennikova.
Fractional Hamilton Jacobi Bellman equations for scaled limits of controlled Continuous Time Random Walks.
Communications in Applied and Industrial Mathematics {\bf 6:1} (2014), e-484.
http://caim.simai.eu/index.php/caim/article/view/484/PDF

\bibitem{KolVer17}
V. N. Kolokoltsov and M Veretennikova. The fractional Hamilton-Jacobi-Bellman equation.
Journal of Applied Nonlinear Dynamics 6:1 (2017), 45 - 56.

\bibitem{KoWei12}
V. Kolokoltsov and W. Yang.
The turnpike theorems for Markov games.
Dynamic Games and Applications {\bf 2: 3} (2012), 294-312.

\bibitem{KoWe13}
V. Kolokoltsov and W. Yang.
Existence of solutions to path-dependent kinetic equations and related forward - backward systems.
Open Journal of Optimization {\bf 2:2}, 39-44 (2013), http://www.scirp.org/journal/ojop/

\bibitem{KorKor}
N. A. Korgin and V. O. Korepanov.
An efficient solution of the resource allotment problem with the Groves - Ledyard mechanism under transferable utility
Automation and remote control 77: 5 (2016), 914 - 942.

\bibitem{KraRed}
P. L. Krapivsky and S. Redner.
Organization of growing random networks.
Phys. Rev. E {\bf 63} (2001), 066123.

\bibitem{KrYi13}
G. E. Kreindler and H. P. Young.
Fast convergence in evolutionary equilibrium selection.
Games Econom. Behav. {\bf 80} (2013), 39 - 67.

\bibitem{LaMoMaRa09}
A. Lambert-Mogiliansky, M. Majumdar and R. Radner.
Strategic analysis of petty corruption with an intermediary. Rev. Econ. Des. {\bf 13: 1-2} (2009), 45 - 57.

\bibitem{LaMoMaRa08}
A. Lambert-Mogiliansky, M. Majumdar and R. Radner. Petty corruption: a game-theoretic approach.
J Econ Theor {\bf 4} (2008), 273 - 297.

\bibitem{LL2006}  J-M. Lasry, P-L. Lions. Jeux \`a champ moyen, I. Le cas stationnaire.
Comptes Rendus Mathematique
 Acad. Sci. Paris 343:9 (2006), 619 - 625.

\bibitem{LL2018}  J-M. Lasry, P-L. Lions.
Mean-field games with a major player. Jeux a champ moyen avec agent dominant.
Comptes Rendus Mathematique  Acad. Sci. Paris
356:8 (2018), 886-890.


\bibitem{LeBo13}
J.-Y. Le Boudec.
The stationary behaviour of fluid limits of reversible processes is concentrated on stationary points.
Netw. Heterog. Media {\bf 8:2} (2013), 529 - 540.

\bibitem{LeeSig} J.-H. Lee, K. Sigmund, U. Dieckmann and Yoh Iwasa.
Games of corruption: How to suppress illegal logging.
Journal of Theoretical Biology {\bf 367} (2015), 1-13.

\bibitem{LeTsi98} M. I. Levin and M. L. Tsirik. Mathematical modeling of corruption (in Russian) . Ekon. i Matem. Metody
{\bf 34:4} (1998), 34-55.

\bibitem{LiLiSt} Zh. Li, Q. Liao and A. Striegel. Botnet Economics: Uncertainty Matters.
http://weis2008.econinfosec.org/papers/Liao.pdf

\bibitem{LiuTaYa} J. Liu, Y. Tang and Z. R. Yang. The spread of disease with birth and death on networks.
arXiv:q-bio/0402042v3. Published in Journal of Statistical Mechanics: Theory and Experiment, 2004.

\bibitem{LoPi06}
D. L\'opez-Pintado.
Contagion and coordination in random networks.
Internat. J. Game Theory {\bf 34:3} (2006), 371 - 381.

\bibitem{LyWi}
K-W. Lye, J. M. Wing.
Game strategies in network security.
Int J Inf Secur 4 (2005), 71 - 86.

\bibitem{Lynch02} J. P. Lynch and K. H. Law.
Market-based control of linear structural systems.
Earthquake Engng Struct. Dyn. 31 (2002), 1855-1877.


\bibitem{Mal00} O. A. Malafeyev. Controlled conflict systems (In Russian).
St.Petersburg State University Publication, 2000, ISBN 5-288-01769-7.


\bibitem{Mal16} O. A. Malafeyev (Eds.)
Introduction to modeling corruption systems and processes, (in Russian), v. 1 and 2. Stavropol, 2016.

\bibitem{Mal14}
O. A. Malafeyev, N. D. Redinskikh and G. V. Alferov.
Electric circuits analogies in economics modeling: Corruption networks.
Proceedings of ICEE-2014 (2nd International Conference on Emission Electronics),
DOI: 10.1109/Emission.2014.6893965, Publisher: IEEE

\bibitem{Mal17}
O. A. Malafeyev, S. A. Nemnyugin, E. P. Kolpak and A. Awasthi. Corruption dynamics model,
AIP Conference Proceedings 1863, 170013 (2017); https://doi.org/10.1063/1.4992358

\bibitem{Mal17a}
O. A. Malafeyev, D. Saifullina, G. Ivaniukovich, V. Marakhov and I. Zaytseva.
 The model of multi-agent interaction in a transportation problem with a corruption component.
AIP Conference Proceedings 1863, 170015 (2017); https://doi.org/10.1063/1.4992360


\bibitem{Masch66} M.A. Maschler. A price leadership method for solving the inspector's non-constant-sum game.
Naval. Res. Logist. Q. 13 (1966). 11-33.

\bibitem{Mazabook} V. V. Mazalov. Mathematical game theory and Applications. John Wiley, 2014.

\bibitem{MazRetFish} V. V. Mazalov and A. N. Rettieva. Fish war and cooperation maintenance.
Ecological Modelling 221 (2010), 1545-1553.

\bibitem{Meerbook}
M. M. Meerschaert and A. Sikorskii.
Stochastic Models for Fractional Calculus.
 De Gruyter Studies in Mathematics Vol. 43, NY (2012).

\bibitem{Mesq} B. B. De Mesquita. The Predictioneer's Game. Random House 2010.

\bibitem{MGT} M. Mobilia, I.T. Georgiev, U.C. T\"auber and C. Uwe. Phase
transitions and spatio-temporal fluctuations in stochastic lattice
Lotka-Volterra models.  J. Stat. Phys.  {\bf 128}  (2007),  no. 1-2,
447-483.

\bibitem{NevMal15}
E. G. Neverova and O. A. Malafeyev. A model of interaction between anticorruption authority and corruption groups.
 AIP Conference Proceedings 1648, 450012 (2015); https://doi.org/10.1063/1.4912671

\bibitem{NgenZac} F. Ngendafuriyo and G. Zaccour. Fighting corruption: to precommit or not?
Economics Letters {\bf 120} (2013), 149-154.

\bibitem{Nik14} P. V. Nikolaev. Corruption suppression models: the role of inspectors' moral level.
Conmputational Mathematics and Modeling {\bf 25:1}  (2014), 87 - 102.

\bibitem{NovikovProj07} D. A. Novikov. Project Management: Organizational Mechanisms (In Russian). Moscow, PMSOFT, 2007.

\bibitem{Nor00} J. Norris. Cluster Coagulation. Comm. Math. Phys. {\bf 209} (2000), 407 - 435.

\bibitem{Nobook} J. Norris. Markov Chains. Cambridge Univ. Press.

\bibitem{NoCa13}
M. Nourian and P. E. Caines. $\ep$-Nash mean field game theory for nonlinear stochastic dynamical
 systems with major and minor agents. SIAM J. Control Optim. {\bf 51:4} (2013), 3302 - 3331.

\bibitem{NourCaiMal} M. Nourian, P. E. Caines and R. P. Malham\'e.
Mean field anlysis of controlled Cucker-Smale type flocking: Linear anaysis and perturbation equations.
In: S. Bittani (Ed.). Proc. of hte 18th IFAC World Congress, milan 2011. Curran Associates, 2011, 4471 - 4476.

\bibitem{Nowzari17} C. Nowzari, V. Preciado and G. J. Pappas.
Optimal Resource Allocation for Control of Networked Epidemic Models.
IEEE transation on control of network sytems 4:2 (2017), 159 - 169.


\bibitem{Nutz18}
M. Nutz. A mean field game of optimal stopping.
SIAM J. Control Optim. 56:2 (2018), 1206 - 1221.

\bibitem{OrsonPerc15}
M. R. D'Orsogna and M. Perc. Statistical physics of crime:
A review. Physics of life reviews 12 (2015), 1-21.

\bibitem{Perc17}
M. Perc M, J. J. Jordan, D. G. Rand, Z. Wang, S. Boccaletti and A. Szolnoki.
Statistical physics of human cooperation. Physics Reports. 687 (2017), 1-51.

\bibitem{Petr1}
I. G.  Petrovski. Ordinary differential equations. Translated from the Russian and edited by Richard A. Silverman.
Prentice-Hall, Inc., Englewood Cliffs, N.J. 1966.

\bibitem{Petr2} I. G. Petrovsky.
Lectures on partial differential equations. Translated from the Russian by A. Shenitzer. Dover Publications, Inc., New York, 1991.

\bibitem{PetrZen} L. A. Petrosyan and N. A. Zenkevich. Game Theory. World Scietific, Sec Ed., 2016.

\bibitem{PichMalaf16} Y. A. Pichugin, O. A. Malafeyev. Statistical estimation of corruption indicators in the firm.
 Applied Mathematical Sciences 10:42 (2016), 2065 - 2073.

\bibitem{Pushkin04}
D. O. Pushkin and H. Aref. Bank mergers as scale-free coagulation.
Physica A {\bf 336} (2004) 571 - 584.

\bibitem{Rass03} L. Rass and J. Radcliffe. Spatial Deterministic Epidemics.
Mathematical Surveys and Monographs, v. 102. AMS 2003.

\bibitem{Ren11}
Ren W., Yongcan C. Distributed Coordination of Multi-agent Networks. Springer, 2011.

\bibitem{Rich48} L. F. Richardson. Variation of the frequency of fatal quarrels with magnitude.
Journ. Amer. Stat. Ass. {\bf 43} (1948), 523.

\bibitem{Rose-Ackerman1}  S. Rose-Ackerman.
Corruption: a Study in Piblic Economy. N.Y. Academic Press, 1978.

\bibitem{Rose-Ackerman2}  S. Rose-Ackerman.
Corruption and Government: Cuases, Consequences and Reforms.
Cambridge University Press, 1999.

\bibitem{RosSand04}
B. P. Rosendorff and T. Sandler.
Too Much of a Good Thing?: The Proactive Response Dilemma.
Journal of Conflict Resolution {\bf 48} (2005), 657-671.

\bibitem{Ross} Sh. Ross. Introduction to stochastic dynamic programming. Wiley, 1983.

\bibitem{RoSa13} G. Roth and W. Sandholm.
Stochastic approximations with constant step size and differential inclusions.
SIAM J. Control Optim. {\bf 51:1} (2013), 525 - 555.

\bibitem{RothenZa} D. Rothenstein and S. Zamir. Imperfect inspection games over time.
Ann. Oper. Res. 109 (2002), 175-192.

\bibitem{SaMaSo} A. Saichev, Ya. Malvergne and D. Sornette. Theory of Zipf's Law and Beyond.
Lecture Notes in Economics and Mathematicl Systems 632, Springer, Berlin 2010.

\bibitem{Sakaguchi} M. A. Sakaguchi. A sequential allocation game for targets with
 varying values. J. Oper. Res. Soc. Jpn. 20 (1977), 182-193.

\bibitem{SBambi}
F. Salten. Bambi, A life in the Woods. Engl. Transl. Simon and Schuster, 1928.

\bibitem{San01} W. Sandholm.
Almost global convergence to p-dominant equilibrium.
Internat. J. Game Theory {\bf 30:1} (2001), 107 - 116.

\bibitem{San12} W. Sandholm.
Stochastic imitative game dynamics with committed agents.
J. Econom. Theory {\bf 147:5} (2012), 2056 - 2071.


\bibitem{SaAr03}
T. Sandler and D. Arce.
Terrorism and Game Theory.
Simulation and Gaming {\bf 34:3} (2003), 319 - 337.

\bibitem{SandLa88}
T. Sandler and H. E. Lapan.
The Calculus of Dissent: An Analysis of Terrorists' Choice of Targets. Synthese {\bf 76:2} (1988), 245-261.

\bibitem{Sant08}
F. J. Santonja, A. C. Tarazona and R. J. Villanueva.
A mathematical model of the pressure of an extreme ideology on a society.
Computers and Mathemsatics with Applications 56 (2008), 836-846.


\bibitem{Saper08} A. Saperstein.
Mathematical modeling of the interaction between terrorism and counter-terrorism and its policy implications.
Complexity 14:1 (2008), 45 - 49.

\bibitem{Shoham} Y. Shoham and K. Leyton-Brown.  Multiagent Systems: Algorithmic, Game-theoretic and Logical foundations.
Cambridge University Press, 2008.


\bibitem{SimRoy} M.V. Simkin and V.P. Roychowdhury.
Re-inventing Willis.
Physics Reports {\bf 502} (2011), 1-35.

\bibitem{Stark14} R. Starkermann. Unity is strength or corruption! (a mathematical model).
Cybernetics and Systems: An International Journal. {\bf 20 :2} (1989), 153-163.

\bibitem{Subbotin} A.I. Subbotin. Generalized Solutions of First Order of PDEs: The Dynamical Optimization Perspectives.
Boton, Birkhauser, 1995.

\bibitem{TarasMgta16}
A. M. Tarasiev, A. A. Usova and Yu. V. Schmotina. Calculation of predicted trajectories of the
economic development under structural changes (In Russian). Mathematical theory of games and applications
(Matematicheskaya teoria ugr i priloshenia) 8:3 (2016), 34 - 66.

\bibitem{TemBas14}
H. Tembine, Q. Zhu and T. Basar.
 Risk-sensitive mean-field games. IEEE Trans. Automat. Control 59:4 (2014), 835 -- 850.

\bibitem{Thom76} M. U. Thomas and Y. Nisgav. An infiltration game with time dependent payoff.
Naval. Res. Logist. Q. 23 (1976), 297-302.

\bibitem{UchBook} V. V. Uchaikin. Fractional Derivatives for Physicists and Engineers.
Springer (2012).

\bibitem{Udwadia} F. Udwadia, G. Leitmann and L. Lambertini.
A dynamical model of terrorism. Discrete Dyn. Nat. Soc. 2006, Art. ID 85653.

\bibitem{Ushak13}
I. A. Ushakov.
Optimal resource allocation. With practical statistical applications and theory.
John Wiley and Sons, Inc., Hoboken, NJ, 2013.

\bibitem{Vasinbook} A. A. Vasin. Noncooperative Games in Nature and Society (in Russian). MAKS Press, Moscow 2005.

\bibitem{VasinKarUr10} A. A. Vasin, P. A. Kartunova and A. S. Urazov. Models of Organization of State Inspection
 and Anticorruption Measures. Matem. Modeling {\bf 22:4} (2010), 67-89.

\bibitem{VonSten14} B. Von Stengel. Recursive inspection games. Math. Oper. Res. 41:3 (2016), 935 - 952.

\bibitem{WaZh13}
B.-Ch. Wang and J.-F. Zhang.
Distributed output feedback control of Markov jump multi-agent systems.
Automatica J. IFAC {\bf 49:5} (2013), 1397 - 1402.

\bibitem{West} B. J. West. Fractional calculus View of Complexity.
Tomorrow's Science. CRC Press, Boca Raton, 2016.

\bibitem{Wie} P. Wiecek. Total Reward Semi-Markov Mean-Field Games
with Complementarity Properties. Prerint 2015. Submitted for publication.

\bibitem{WrzaczekKapFeich17}
S. Wrzaczek, E. Kaplan, J. P. Caulkins, A. Seidl ans G. Feichtinger.
Differential Terror Queue Games. Dyn Games Appl 7 (2017), 578-593.

\bibitem{YNRS}
G. Yaari, A. Nowak, K. Rakocy and S. Solomon.
{\sl Microscopic study reveals the singular origins of growth.}
Eur. Phys. J. B 62, 505 - 513 (2008)
DOI: 10.1140/epjb/e2008-00189-6

\bibitem{YinOnKuram12} H. Yin. P. G. Mehta, S. P. Meyn and U. V. Shanbhag.
Synchronisation of Coupled Oscillators is a Game.
IEEE Trans. Automatic Control 57:4 (2012), 920-935.

\bibitem{You93} H. P. Young.
The evolution of conventions.
Econometrica {\bf 61:1} (1993), 57 - 84.

\bibitem{Zas06} A. J. Zaslavski. Turnpike properties in the calculus of variations and optimal control. Springer,
New York, 2006.

\bibitem{Zhang17}
J. Zhang. Approximately Optimal Computing Budget Allocation for Selection of the Best and Worst Designs.
IEEE Transaction on Automatic Control 62:7 (2017), 3249-3261.

\end{thebibliography}
\end{document}